\documentclass[a4paper]{article}
\usepackage{minitoc}
\usepackage[titletoc]{appendix}
\usepackage[utf8]{inputenc}
\usepackage[T1]{fontenc}
\usepackage{graphicx}
\usepackage{amsmath}
\usepackage{listings}
\usepackage{moreverb}
\usepackage{eurosym}
\usepackage{fancyvrb}
\usepackage{afterpage}
\usepackage[final]{pdfpages}
\usepackage[nottoc]{tocbibind}
\usepackage{listings}
\usepackage{amsthm}
\usepackage{amsfonts}
\usepackage{appendix}
\usepackage{geometry}
\usepackage[]{algorithm2e}
\usepackage{hyperref}
\usepackage{fixltx2e}
\usepackage{amssymb}
\usepackage{tikz}
\usetikzlibrary{calc}
\usepackage{float}
\usepackage{color}
\usepackage{lipsum}
\usepackage{multirow}
\usepackage{pdflscape}

\newtheorem{theorem}{Theorem}[section]
\newtheorem{lemma}[theorem]{Lemma}
\newtheorem{proposition}[theorem]{Proposition}

\newtheorem{definition}[theorem]{Definition}
\newtheorem{example}[theorem]{Example}

\newtheorem{remark}[theorem]{Remark}

\newtheorem{convention}[theorem]{Convention}
\newtheorem{notation}[theorem]{Notation}
\newtheorem{procdur}[theorem]{Procedure}

\newtheorem{innercustomthm}{Theorem}
\newenvironment{customthm}[1]
{\renewcommand\theinnercustomthm{#1}\innercustomthm}
{\endinnercustomthm}

\makeatletter
\def\v@rt#1#2{\m@th\ooalign{$\hfil#1|\hfil$\crcr$#1#2$}}
\def\captr{\mathrel{\mathpalette\v@rt\cap}}
\makeatother

\makeatletter

\newcommand{\Rmnum}[1]{\expandafter\@slowromancap\romannumeral #1@}
\makeatother

\newcommand\blfootnote[1]{
	\begingroup
	\renewcommand\thefootnote{}\footnote{#1}
	\addtocounter{footnote}{-1}
	\endgroup
}

\def \Refl {T}
\def \RRefl {\boldsymbol{T}}
\def \refl {t}
\def \rrefl {\boldsymbol{t}}

\def \cS {{\mathcal S}}

\def \GG {{\boldsymbol G}}
\def \SS {{\boldsymbol{\mathcal{S}}}}
\def \ss {{\boldsymbol s}}

\def\bs#1{\boldsymbol{#1}}
\def\braid#1#2{{\hbox{\sc b}}(#1, #2)}
\def \cc#1#2#3#4{{\hbox{\sc cc}}(#1,#2,#3,#4)}
\def \tc#1#2#3#4{{\hbox{\sc tc}}(#1,#2,#3,#4)}

\begin{document}

\title{Interval groups related to finite Coxeter groups I}
\author{Barbara Baumeister, Georges Neaime, and Sarah Rees}
\date{}
	
\maketitle

\begin{abstract}
We derive presentations of the interval groups related to all quasi-Coxeter elements in the Coxeter group of type $D_n$. 
Type $D_n$ is the only infinite family of finite Coxeter groups that admits proper quasi-Coxeter elements. The presentations we obtain are over a set of generators in bijection with what we call a Carter generating set, and the relations are those defined by the related Carter diagram together with a twisted or a cycle commutator relator, depending on whether the quasi-Coxeter element is a Coxeter element or not. The proof is based on the description of two combinatorial techniques related to the intervals of quasi-Coxeter elements.

In a subsequent work \cite{BaumNeaReesPart2}, we complete our analysis to cover all the exceptional cases of finite Coxeter groups, and establish that almost all the interval groups related to proper quasi-Coxeter elements are not isomorphic to the related Artin groups, hence establishing a new family of interval groups with nice presentations. Alongside the proof of the main results, we establish important properties related to the dual approach to Coxeter and Artin groups.
\end{abstract}

\blfootnote{Date: \today}

\begin{small}
	\tableofcontents 
\end{small}

\section{Introduction}\label{SectionIntro}

The philosophy of interval Garside theory is that starting from suitable intervals in a given group, we construct an interval Garside monoid and group, along with a complex whose fundamental group is the interval Garside group, such that the divisibility relations of the interval provide relevant information about the interval Garside group. Part of the information we obtain are efficient solutions to the word and conjugacy problems, as well as important group-theoretical properties \cite{DehornoyEtAl}. Interval Garside groups also enjoy important homological, and homotopical properties~\cite{CharneyMeierWhittlesey}.

Garside theory is relevant in the context of Coxeter and Artin groups. Actually, Garside structures first arose out of observations of properties of Artin's braid group that were made in Garside's Oxford thesis \cite{GarsideThesis} and his article \cite{Garside}. It was then realised that Garside's approach extend to all Artin groups of spherical type, independently by Brieskorn--Saito and Deligne in two adjacent articles in the Inventiones \cite{BrieskornSaito} and \cite{Deligne}. This approach is called the standard approach to Coxeter and Artin groups.

The dual approach consists of analysing the Coxeter group as a group generated by all its reflections. Spherical Artin groups are constructed from intervals called the generalised non-crossing partitions. This approach was established by Bessis in \cite{BessisDualMonoid}. These intervals consist of elements lying below the so-called Coxeter elements that play a prominent role within the dual approach. The Coxeter elements are all conjugate to one another and some of them can be found by taking the product of the elements in the standard generating set in any order.

Coxeter elements are of maximal length over the set of reflections, but they do not exhaust all the elements of maximal length. Quasi-Coxeter elements \cite{BaumGobet} are of maximal length such that the reflections in a certain reduced decomposition generate the Coxeter group. Among them are Coxeter elements. We call a proper quasi-Coxeter element a quasi-Coxeter element that is not a Coxeter element. Amongst the infinite families of finite Coxeter groups, proper quasi-Coxeter elements exist only in type $D_n$. Carter \cite{Carter} classified the conjugacy classes in Weyl groups. Among them are the conjugacy classes of quasi-Coxeter elements. He also defined diagrams related to these classes that we call Carter diagrams. Cameron-Seidel-Tsaranov \cite{Cameron} defined presentations of Weyl groups defined on Carter diagrams by adding cycle commutator relators.

We establish presentations of the interval groups related to all quasi-Coxeter elements. Our presentations are compatible with the analysis of Carter \cite{Carter}. Actually, they are always nicely defined on Carter diagrams by adding either cycle commutator relators or twisted cycle commutator relators depending whether the quasi-Coxeter element is a Coxeter element or not. Twisted cycle and cycle commutator relators can be written as relations between positive words. For Coxeter elements, where the interval group is the Artin group, some of our group presentations also arise from cluster algebras (see \cite{BarotMarsh, GrantMarsh} and also \cite{Hayley}). For almost all the other proper quasi-Coxeter elements, we can establish that the interval group related to each of them is not isomorphic to the corresponding Artin group. Although we obtain nice presentations of these groups, the intervals of proper quasi-Coxeter elements are not lattices in almost all the cases, hence not giving rise to Garside structures. This classifies the interval Garside structures one obtains for quasi-Coxeter elements within the dual approach. Along with the description of the presentations of interval groups, we describe important properties for quasi-Coxeter elements, their divisors, and their lifts to the interval groups. 

We divide our analysis into two parts. This paper is the first part of the series. It concerns the only infinite family of Coxeter groups (that of type $D_n$), where proper quasi-Coxeter elements exist. This family needs a special treatment. The second part deals with the exceptional cases and establishes the non-isomorphism results.

The main theorem of this paper is actually the following. We refer to Sections~\ref{SectionDefPrel} and \ref{SecDualCox} for the definitions of a quasi-Coxeter element, its associated Carter diagram $\Delta$, and the group $A(\Delta)$.

\begin{customthm}{A}\label{ThmPres}
	Let $w$ be a quasi-Coxeter element of the Coxeter group $W$ of type $D_n$ and $\Delta$ its associated Carter diagram, as shown in Figure~\ref{FigureCarterDiagramDn} of 
	Section~\ref{SubsectionQuasiCoxeter}. Then the interval group $G([1,w])$ admits a presentation over the generators $x_1,\dotsc,x_n$ corresponding to the vertices of $\Delta$ together with the relations described by $\Delta$ and the twisted cycle commutator relator 
	$\tc{x_{i}}{x_{j}}{x_{k}}{x_{l}}$, associated with the $4$-cycle $(x_i,x_j,x_k,x_l)$ \mbox{within $\Delta$,}
	that is, \[ G([1,w]) \cong A(\Delta)/\langle \langle
	[x_{i},x_{j}^{-1}x_{k}x_{l}x_{k}^{-1}x_{j}]\rangle \rangle. \]
\end{customthm}

We shall reformulate Theorem~\ref{ThmPres} as Theorem~\ref{ThmPresDn1}, and prove that theorem in Section~\ref{SecProofPres}.

Our use of the word `twisted' comes from the fact that when following the cycle $(x_i,x_j,x_k,x_l)$ in the cycle commutator relator, we invert the element $x_j$. We call the set $\{x_1, \dotsc, x_n\}$ of generators that appears in Theorem~\ref{ThmPres} a Carter generating set.

The case where $w$ is a Coxeter element is a particular case of Theorem~\ref{ThmPres}, where there is no $4$-cycle. Therefore, in this case, we get a new proof of the fact that $G([1,w])$ is the Artin group of type $D_n$, that was covered before in \cite{BessisDualMonoid}.

Within the proof of Theorem~\ref{ThmPres}, we describe an important combinatorial technique that derives reduced expressions over the set of reflections for the divisors of length $n-1$ of
quasi-Coxeter elements. This
reveals important information on the poset of quasi-Coxeter elements,
on parabolic subgroups, and enables us to establish nice presentations
of the interval groups in accordance with Carter diagrams. The
algorithms we define use the description of the elements in the
Coxeter group of type $D_n$ as monomial matrices. This is relevant to
the dual approach for complex reflection groups. We suspect that our
algorithms generalise to the context of the infinite families of
complex reflection groups.

This first paper is structured as follows.
After some preparations and after introducing the notation, Section~\ref{SectionDefPrel} contains our strategy for the proof of Theorem~\ref{ThmPres}. The section also contains a good summary of our results (see Section~\ref{SubsectionStrategyProof}). In Section~\ref{SecDualCox}, we recall the dual approach to the Coxeter group of type $D_n$. Next, we describe the combinatorial technique that defines reduced decompositions  and introduce diagrams for these decompositions in Section \ref{SecMaxDivisors}. Within our proof, parabolic subgroups play an important role. In Section~\ref{SecDecompReflections}, we decompose the reflections over the Carter generating set and define the lift of these decompositions to the interval groups. Finally, Section~\ref{SecProofPres} finishes our proof by induction.

\textbf{Acknowledgements}. The second author would like to thank the DFG (project BA2200/5-1). The three authors would like to thank Bielefeld University which hosted the visit of the third author, and to thank the MFO in Oberworfach that hosted them for two weeks in September 2020 in a Research in Pairs meeting. The authors would also like to thank Derek Holt who helped them to optimise their use of \verb|kbmag| \cite{KBMAG} within \verb|GAP| \cite{GAP}.

\section{Definitions and Preliminaries}\label{SectionDefPrel}

\subsection{Coxeter groups and Artin groups}\label{SubsectionCoxeterArtinGrps}

\begin{definition}\label{DefinitionCoxeterGroups}
	
	Suppose that $W$ is a group and $S$ is a subset of $W$. For $s$ and $t$ in $S$, let $m_{st}$ be the order of $st$ if this order is finite, and be $\infty$ otherwise. We say that $(W,S)$ is a Coxeter system, and that $W$ is a Coxeter group with Coxeter system $S$, if $W$ admits the presentation with generating set $S$ together with the quadratic relations: $s^2=1$ for all $s \in S$, and the braid relations: $\underset{m_{st}}{\underbrace{sts\dotsc}}=\underset{m_{st}}{\underbrace{tst\dotsc}}$  for $s,t \in S$, $s \neq t$ and $m_{st} \neq \infty$. We define an element $t$ of $W$ to be a reflection if it is a conjugate of an element of $S$.
\end{definition}

We define the Artin group $A(W)$ associated with a Coxeter system $(W,S)$ as follows.

\begin{definition}\label{DefinitionArtinTits}
	The Artin group $A(W)$ associated with a Coxeter system $(W,S)$ is defined by a presentation with generating set $\SS$ in bijection with $S$ and the braid relations: $\underset{m_{st}}{\underbrace{\boldsymbol{s}\boldsymbol{t}\boldsymbol{s}\dotsc}}=\underset{m_{st}}{\underbrace{\boldsymbol{t}\boldsymbol{s}\boldsymbol{t}\dotsc}}$ for $\boldsymbol{s}, \boldsymbol{t} \in \SS$ and $\boldsymbol{s} \neq \boldsymbol{t}$, where $m_{st} \in \mathbb{Z}_{\geq 2}$ is the order of $st$ in $W$.
\end{definition}

These presentations are often represented
graphically using a Coxeter diagram $\Gamma$. This is a graph with vertex 
set $S$, in which the edge $\{ s,t\}$ exists if $m_{st} \geq 3$, and is labelled with $m_{st}$ when $m_{st} \geq 4$.
Let $\Gamma$ be such a diagram. 
We denote by $W(\Gamma)$ and $A(\Gamma)$ the related Coxeter and Artin groups $W$ and $A(W)$. The finite Coxeter groups are precisely the real reflection groups, and the spherical Artin groups are the Artin groups related to the finite Coxeter groups. 
The corresponding Coxeter diagrams are the three infinite families of types $A$, $B$, and $D$, and the exceptional cases of types $E_6$, $E_7$, $E_8$, $F_4$, $H_3$, $H_4$, and $I_2(e)$. In the remainder of the article, $W$ will  always be a finite Coxeter group.

Recall that the Coxeter group of type $A_n$ ($n \geq 1$) is the symmetric group \mbox{Sym$(n+1)$} and the related Artin group is the 
usual braid group $$\mathcal{B}_{n+1}  = \langle \ss_1, \ldots , \ss_n~|~ \boldsymbol{s}_{i}\boldsymbol{s}_{i+1}\boldsymbol{s}_i = \boldsymbol{s}_{i+1}\boldsymbol{s}_i\boldsymbol{s}_{i+1}~\mbox{for}~1 \leq i \leq n-1~\mbox{and}$$
$$\boldsymbol{s}_i\boldsymbol{s}_j=\boldsymbol{s}_j\boldsymbol{s}_i~\mbox{for}~|i-j| > 1 \rangle.$$

%
%
%
%

\subsection{Quasi-Coxeter elements}\label{SubsectionQCox} 

Let $(W,S)$ be a finite Coxeter system, and let $\Refl := \cup_{w \in W} S^w$ be the set of all its reflections. As each $w \in W$ is a product of reflections in $\Refl$, we can define 
$$
\ell_\Refl(w):= \min \{ k \in \mathbb{Z}_{\geq 0} \mid w=\refl_{1}\refl_2 \dotsc \refl_{k};\ \refl_i \in \Refl \},
$$
the reflection length of $w$. If $w = \refl_1\refl_2 \dotsc \refl_k$ with $\refl_i \in \Refl$ and $k=\ell_\Refl (w)$, we call $(\refl_1,\refl_2, \dotsc, \refl_k)$ (or $\refl_1\refl_2 \dotsc \refl_k$ by abuse of notation) a reduced decomposition of $w$.

Now we define the notion of quasi-Coxeter elements.

\begin{definition}\label{DefQuasiCoxInGen}
	An element $w$ of a finite Coxeter group $W$ is called a quasi-Coxeter element if there exists a reduced decomposition
	$\refl_1\refl_2 \dotsc \refl_n$ of $w$ where $n$ is the cardinality of $S$ such that $\langle \refl_1,\refl_2, \dotsc, \refl_n \rangle = W$.
\end{definition}

A Coxeter element is a conjugate of any element that is written as the product of the simple generators of $W$ in any order. Note that every Coxeter element is a quasi-Coxeter element. A quasi-Coxeter element is called proper if it is not a Coxeter element.

It is shown in \cite{BW} that the quasi-Coxeter elements in simply laced Coxeter groups are precisely those elements that admit a reduced decomposition into reflections such that the roots related to these reflections form a basis of the related root lattice. In the non-simply laced case, it is also required that the system of coroots generates the coroot lattice.

Recall that a parabolic subgroup of $W$ is a subgroup generated by a conjugate of a subset of $S$. Note that a more general definition of parabolic subgroups, which is in fact equivalent to our definition for finite Coxeter systems, is used in \cite{BDSW, BaumGobet}. We call an element in $W$ a parabolic quasi-Coxeter element if it is a quasi-Coxeter element in a parabolic subgroup of $W$.

Since the set $\Refl$ of reflections is closed under conjugation, there is a natural way to obtain new reflection decompositions from a given one. The braid group $\mathcal{B}_n$ acts on the set $\Refl^n$ of $n$-tuples of reflections via
\begin{align*}
	\boldsymbol{s}_i (\refl_1 ,\dotsc , \refl_n ) &= (\refl_1 ,\dotsc , \refl_{i-1} , \hspace*{5pt} \refl_i \refl_{i+1} \refl_i,
	\hspace*{5pt} \phantom{\refl_{i+1}}\refl_i\phantom{\refl_{i+1}}, \hspace*{5pt} \refl_{i+2} ,
	\dotsc , \refl_n), \\
	\boldsymbol{s}_i^{-1} (\refl_1 ,\dotsc , \refl_n ) &= (\refl_1 ,\dotsc , \refl_{i-1} , \hspace*{5pt} \phantom
	{\refl_i}\refl_{i+1}\phantom{\refl_i}, \hspace*{5pt} \refl_{i+1}\refl_i\refl_{i+1}, \hspace*{5pt} \refl_{i+2} ,
	\dotsc , \refl_n),\quad i=1,\ldots,n-1,
\end{align*}
\noindent the so-called Hurwitz action of $\mathcal{B}_n$ on $\Refl^n$. It is readily observed that this action restricts to the set of all reduced reflection decompositions of a given element $w \in W$. If the latter action is transitive, then we say that the dual Matsumoto property holds for $w$.

The dual Matsumoto property characterises the parabolic quasi-Coxeter elements (see Theorem~1.1 in \cite{BaumGobet}).

\begin{theorem}\label{LemmaTransitivityQCox}
An element  $ w \in W$ is a parabolic quasi-Coxeter element if and only if the dual Matsumoto property holds for $w$.
\end{theorem}  

We recall the following fact, and thereby introduce the notation $P_w$  for parabolic quasi-Coxeter elements $w \in W$. The result is an easy consequence of Theorem~6.1 of \cite{BaumGobet}.

\begin{lemma}\label{LemParaSub}
Let $w \in W$ be a parabolic quasi-Coxeter element and $w = \refl_1\refl_2 \dotsc \refl_k$ be a reduced decomposition into reflections.
Then $P_w:= \langle  \refl_1, \ldots , \refl_k \rangle$ is a parabolic subgroup and the definition of $P_w$ is independent of the choice
of the reduced reflection decomposition of $w$.
\end{lemma}

\subsection{Decomposition diagrams}

We introduce diagrams related to reduced decompositions.

\begin{definition}\label{DefDiagramRedDecomp}
	Let  $\refl_1\refl_2 \dotsc \refl_k$ be a reduced decomposition of $w \in W$. We define a decomposition diagram related to  $\refl_1\refl_2 \dotsc \refl_k$ as follows. The vertices of the diagram correspond to the reflections $\refl_1,\refl_2, \dotsc, \refl_k$. If two reflections commute, we put no edge between the related vertices. Otherwise, we put an edge, and we label it by the order of the product of the two reflections when this order is strictly bigger than 3.
\end{definition}

In  Carter's classification of the conjugacy classes in the Weyl groups \cite{Carter}, it is shown that every element $w$ in $W$ is the product $w = w_1w_2$ of two involutions, and that each involution is the product of commuting reflections, which then provides a bipartite decomposition of $w$. Carter exhibited the list of conjugacy classes of proper quasi-Coxeter elements by describing for each class a diagram related to a bipartite decomposition for a representative of the class (see Table 2 in \cite{Carter}) which we call a Carter diagram. Note that Definition~\ref{DefDiagramRedDecomp} generalises the notion of Carter diagrams.

\subsection{Interval groups of quasi-Coxeter elements}\label{SubsectionIntervalGrpsOfQCoxElts}

We start by defining left and right division.

\begin{definition}\label{DefAbsoluteOrder}
	
	Let $v, w \in W$. We say that $v$ is a (left) divisor of $w$, and write $v \preceq w$, if $w = v u$ with $u \in W$ and $\ell_{\Refl}(w) = \ell_{\Refl}(v) + \ell_{\Refl}(u)$, where $\ell_{\Refl}(w)$ is the length over $\Refl$ of $w \in W$. The order relation $\preceq$ is called the absolute order relation on $W$.
	
	The interval $[1,w]$ related to an element $w \in W$ is defined to be the set of divisors of $w$ for $\preceq$, that is $[1,w] = \{v \in W ~|~ v \preceq w\}$.
	
	Similarly, we define the division from the right. We say that $v$ is a right divisor of $w$, and write $v \preceq_{r} w$, if $w = u v$ with $u \in W$ and $\ell_{\Refl}(w) = \ell_{\Refl}(v) + \ell_{\Refl}(u)$. Similarly, we also define the interval $[1,w]_{r}$ of right divisors of an element $w \in W$.	
\end{definition}

\begin{remark}\label{RemParabolicQuasiCox}
A quasi-Coxeter element has the inductive  property that every left divisor of it is a parabolic  quasi-Coxeter element  (see Corollary~6.11 in \cite{BaumGobet}).
Therefore, if $w$ is a  quasi-Coxeter element, then every element in the interval  $[1,w]$ is a parabolic  quasi-Coxeter element.
\end{remark}

Now we introduce the definition of an interval group related to quasi-Coxeter elements in $W$. Let $w$ be a quasi-Coxeter element in $W$. Consider the interval $[1,w]$ of divisors of $w$.

\begin{definition}\label{DefIntervalGrpsDn}
	
	We define the group $G([1,w])$ by a presentation with set of generators $\bs{[1,w]}$ in bijection with the interval $[1,w]$, and relations corresponding to the relations in $[1,w]$, meaning that $\bs{uv} = \bs{r}$ if $u,v,r \in [1,w]$, $uv=r$, and $u \preceq r$, i.e. $\ell_{\Refl}(r) = \ell_{\Refl}(u) + \ell_{\Refl}(v)$. 
	
\end{definition}

By transitivity of the Hurwitz action on the set of reduced decompositions of $w$ (see Lemma \ref{LemmaTransitivityQCox}), the next result follows immediately.

\begin{proposition}\label{PropDualPresDn}
	 Let $w \in W$ be a quasi-Coxeter element, and let $\RRefl \subset \bs{[1,w]}$ be the copy of the set of reflections $\Refl$ in $W$. Then $$G([1,w]) = \langle  \RRefl ~|~  \rrefl \rrefl' = \rrefl' \rrefl'' ~\mbox{for}~\rrefl , \rrefl',  \rrefl''  \in   \RRefl~\mbox{if}~\refl \neq \refl' , \refl'' \in \Refl~\mbox{and}~\refl \refl' = \refl' \refl'' \preceq w \rangle$$
	is a presentation of the interval group with respect to $w$.
\end{proposition}

Notice that the relations described in Proposition~\ref{PropDualPresDn} are the relations that are visible in the poset $([1,w],\preceq)$ in heights one and two. They are called the dual braid relations (see \cite{BessisDualMonoid}).

The following result due to Bessis--Digne--Michel \cite{BessisDigneMichel} is the main theorem in interval Garside theory.

\begin{theorem}\label{TheoremMainThmIntGarTheory}
	If the two intervals $[1,w]$ and $[1,w]_{r}$ are equal (we say that $w$ is balanced) and if both posets $([1,w],\preceq)$ and $([1,w]_{r},\preceq_{r})$ are lattices, then the interval group $G([1,w])$ is an interval Garside group.
\end{theorem}

Since $\Refl$ is stable under conjugation, quasi-Coxeter elements are always balanced. The only obstruction to obtaining interval Garside groups is the lattice property. In the case when the quasi-Coxeter element is a Coxeter element, Bessis \cite{BessisDualMonoid} showed the following.

\begin{theorem}\label{TheoremBessisDualMonoid}
	Let $W$ be a finite Coxeter group. The interval group $G([1,w])$ for $w \in W$ a Coxeter element is an interval Garside group isomorphic to the corresponding Artin group $A(W)$.
\end{theorem}
 
The main purpose of our work is to continue the analysis of the interval groups related to all quasi-Coxeter elements.
 
We introduce the following notation, which we shall use in the remainder of the article.

\begin{notation}\label{NotationBraidRelators}
	We denote by $\braid{x}{y}$ the braid relator $xyx(yxy)^{-1}$ or $xy(yx)^{-1}$, by $\tc{x}{y}{z}{t}$ the twisted cycle commutator relator $[x,yz^{-1}tzy^{-1}]$, and by $\cc{x}{y}{z}{t}$ the cycle commutator relator $[x,yztz^{-1}y^{-1}]$.
\end{notation}

\subsection{Strategy of the proof}\label{SubsectionStrategyProof}

We describe here our general strategy for the proof of Theorem~\ref{ThmPres} (see also Theorem~\ref{ThmPresDn1}). We are also going to mention some important results that we established within the proof, as they are interesting in themselves. 

Let $W$ be the Coxeter group of type $D_n$. We employ the description of $W$ as a group of monomial matrices as will be explained in Section~\ref{SubsectionCoxGrp}. Let $w$ be a quasi-Coxeter element of the Coxeter group $W$ of type $D_n$. Actually, there exists a reduced decomposition of $w$ whose reflections $s_1,s_2, \dotsc, s_n$ satisfy the relations that can be described by the Carter diagram $\Delta$ (see Figure~\ref{FigureCarterDiagramDn} in Section~\ref{SubsectionQuasiCoxeter}). The quasi-Coxeter elements in type $D_n$ are characterised in Proposition~\ref{PropConjQCoxBipartite}. From now on, we let $S := \{s_1, s_2, \dotsc, s_n\}.$

By \cite{Cameron}, the Coxeter group $W$ admits a presentation on the set $S$ of generators whose relations are the quadratic relations ($s_i^2 = 1$ for $1 \leq i \leq n$) along with the relations of the diagram $\Delta$ and the cycle commutator relator $\cc{s_i}{s_j}{s_k}{s_l}$ in correspondence with the unique $4$-cycle $(s_i,s_j,s_k,s_l)$ of $\Delta$. Note that in $W$, the cycle commutator and the twisted cycle commutator relators associated with the $4$-cycle are the same. All this is described in Section \ref{SecDualCox}.

Consider the interval group $G([1,w])$.  By Proposition \ref{PropDualPresDn}, the group $G([1,w])$ is generated by a copy $\RRefl$ of $\Refl$ along with the dual braid relations $\rrefl \rrefl' = \rrefl' \rrefl''$ ($\rrefl \in \RRefl$ corresponds to $\refl \in \Refl$), whenever $\refl \refl' = \refl' \refl'' \preceq w$, $\refl \neq \refl'$ and $\refl, \refl', \refl'' \in \Refl$.

We want to prove that $G([1,w])$ is isomorphic to the group $\GG$ that is defined by a presentation on the set of generators $\SS = \{\ss_1, \ss_2, \dotsc, \ss_n\} \subset \RRefl$ corresponding to the subset $S$ of $\Refl$ with the corresponding relations described by $\Delta$, along with the twisted cycle commutator relator $\tc{\ss_i}{\ss_j}{\ss_k}{\ss_l}$ in correspondence with the $4$-cycle $(\ss_i,\ss_j,\ss_k,\ss_l)$, where $\ss_i, \ss_j, \ss_k, \ss_l$ correspond to $s_i,s_j,s_k,s_l$, respectively.\\

\textit{Step 1: Definition of $f$.} We define a map $f$ from $\GG$ to $G([1,w])$ by setting $f(\ss_i) = \ss_i$ for each $i$. It will follow from Proposition~\ref{PropPreparationHomoF} and Lemma~\ref{LemmaTC} that the braid relators $\braid{\ss_i}{\ss_j}$ and the twisted cycle commutator relator $\tc{\ss_i}{\ss_j}{\ss_k}{\ss_l}$ specified by the presentation given for $\GG$ hold in $G([1,w])$ as well. Hence $f$ extends to a homomorphism from $\GG$ to $G([1,w])$.\\

\textit{Step 2: Reduced decompositions and their diagrams.} Let $w_0$ be a divisor of length $n-1$ of $w$. We describe a particular reduced decomposition of $w_0 = \refl_1 \refl_2 \dotsc \refl_{n-1}$ in Sections \ref{SubsectionRedDecompI} and \ref{SubsectionRedDecompII}, and characterise whether $w_0$ is a Coxeter element or a proper quasi-Coxeter element in the subgroup  $P_{w_0} := \langle  \refl_1 ,  \dotsc  ,\refl_{n-1} \rangle \subseteq W$ (see for instance Propositions \ref{PropDiagTypeI}, \ref{PropCoxDiagTypeII,III}, and \ref{PropQCoxDiagTypeII,III}). In order to describe these reduced decompositions, we describe a combinatorial technique (in Section~\ref{SubsectionDivisorsln-1}) by using the description of the elements of $W$ as monomial matrices. 

The reduced decomposition $\refl_1 \refl_2 \dotsc \refl_{n-1}$ of $w_0$ corresponds to a decomposition diagram (see Definition~\ref{DefDiagramRedDecomp}) that we denote by $\Delta_0$. We show that  $\Delta_0$  is a disjoint union of Coxeter diagrams of types $A$ or $D$ or of the same type as $\Delta$ (but with fewer generators) with a (single) $4$-cycle (see Propositions \ref{PropDiagTypeI}, \ref{PropCoxDiagTypeII,III}, and \ref{PropQCoxDiagTypeII,III}). Thereby we are able to determine the type of $w_0$.

If $w_0$ is a Coxeter element in $P_{w_0}$, then by \cite{BessisDualMonoid}, the dual braid 
relation $\rrefl \rrefl'=\rrefl' \rrefl''$ is satisfied in $G([1,w_0])$ where $\refl \refl'=\refl' \refl'  \preceq  w$ is a consequence of the relators $\braid{\rrefl_i}{\rrefl_j}$ for $\refl_i, \refl_j \in \{\refl_1, \refl_2, \dotsc, \refl_{n-1}\}$ ($i \neq j$). If $w_0$ is a proper quasi-Coxeter element, then induction on $n$ implies that the dual braid relation $\rrefl \rrefl'=\rrefl' \rrefl''$ satisfied in $G([1,w_0])$ is a consequence of the relations $\braid{\rrefl_i}{\rrefl_j}$ and $\tc{\rrefl_{i}}{\rrefl_{j}}{\rrefl_{k}}{\rrefl_{l}}$ for $\refl_i, \refl_j \in \{\refl_1, \refl_2, \dotsc, \refl_{n-1}\}$ ($i \neq j$) and $(\refl_i, \refl_j, \refl_k, \refl_l)$ the $4$-cycle of $\Delta_0$. In this way, we have shown that all the dual braid relations are consequences of the relations between $\rrefl_1$, $\rrefl_2$, $\dotsc$, $\rrefl_{n-1}$ in correspondence with the relations between $\refl_1$, $\refl_2$, $\dotsc$, $\refl_{n-1}$ implied by each decomposition diagram $\Delta_0$ corresponding to a divisor $w_0$ of length $n-1$ of $w$.\\

\textit{Step 3: Decomposition of elements in $\Refl$ and $\RRefl$.} In order to find a homomorphism $g: G([1,w])\rightarrow \GG$, we describe a decomposition of each element $\rrefl$ in $\RRefl$ in terms of elements in $\SS$. This can be done because of the dual Matsumoto property for $w$, i.e. the transitive Hurwitz action on the set of reduced decompositions over $\Refl$ of $w$. This is based on particular decompositions over $S$ of the reflections in $\Refl$. This is done in Propositions~\ref{PropDecReflections}~and~\ref{PropLiftReflections}.

Let $g$ be the map that sends $\bs{\refl_i} \in G([1,w])$ to its decomposition over $\bs{S}$ as given in Proposition~\ref{PropLiftReflections}. We show that the map $g$ is a homomorphism. Suppose that $\rrefl \rrefl'=\rrefl' \rrefl''$ is a dual braid relation of $G([1,w])$. We need to check that the image of this relation under $g$ holds within $\GG$. There exists a divisor $w_0$ of length $n-1$ of $w$ such that $tt'$ is a prefix of $w_0$. Therefore, $\rrefl \rrefl'=\rrefl' \rrefl''$ holds in $G([1,w_0])$. So rather than checking that the images by $g$ of all the dual braid relations hold in $\GG$, our strategy is to shift the analysis to the groups $G([1,w_0])$, from which we can establish the desired homomorphism. This analysis was done in Step 2.\\

\textit{Step 4: Lift of the relations.} In order to conclude homomorphism for $g$, we finally need to show that the image by $g$ of all the defining relations between the elements $\rrefl_1$, $\rrefl_2$, $\dotsc$, $\rrefl_{n-1}$ of $G([1,w_0])$ in correspondence with the relations between $\refl_1$, $\refl_2$, $\dotsc$, $\refl_{n-1}$ implied by each decomposition diagram $\Delta_0$ can be derived from the relations between the elements of $\SS$ that are implied by the diagram $\Delta$. This is proved in Section \ref{SubsectionLiftRedDecomp} by induction on $n$. The base of our induction is the cases $n=4$ and $n=5$ (see Sections \ref{SubsectionN4} and \ref{SubsectionN5}). Note that we also separate $n=5$ as a base of induction so that we do not need anymore to show twisted cycle commutator relators in Section \ref{SubsectionLiftRedDecomp}. This is possible since, apart from one special case (Equation \ref{Eqt11} with $i=n-1$), in the reduced decomposition $\refl_1 \refl_2 \dotsc \refl_{n-1}$ of each $w_0$, the only reflection $\refl_i$ such that $g(\rrefl_i)$ contains $\ss_n$ in its decomposition over $\SS$ is precisely the last one, that is $\refl_{n-1}$.\\

Section \ref{SubsectionProofn>5} concludes the proof that $g$ is a homomorphism. Isomorphism between $G([1,w])$ and $\GG$ is proved once the composites $f \circ g$ and $g \circ f$ have been shown to be identity maps.

Note that it might be possible to apply this strategy more generally,
but this article only deals with the case where $W$ is of type $D_n$.

\section{Dual approach to the Coxeter group of type $D_n$}\label{SecDualCox}

\subsection{The Coxeter group of type $D_n$}\label{SubsectionCoxGrp}

We employ the description of the Coxeter group $W$ of type $D_n$ ($n \geq 4$) as the group of $n \times n$ monomial matrices such that the nonzero coefficients are equal to $1$ or $-1$ and their product is equal to $1$. This description will help us to describe our combinatorial technique and to easily explain our arguments.

Note that this description of $W$ corresponds to the case $d=1, e=2$ of the infinite series $G(de,e,n)$ of complex reflection groups (see \cite{ShephardTodd}).


\begin{notation}\label{NotnMarkedPerm}
	
A monomial matrix $w \in W$ is associated with a permutation in Sym$(n)$ that has been marked by overlining some elements within its cycles. We call the result, $\sigma_w$, a marked permutation, where an entry $i$ ($1 \leq i \leq n$) of a cycle indicates that the coefficient in row $i$ of $w$ is equal to $1$, while an entry $\overline{i}$ indicates that this coefficient is equal to $-1$.

\end{notation}

The monomial matrix $w$ is denoted by $\sigma_w$, so we have $w = \sigma_w$. When there is no confusion, we remove the cycles $(i)$, for $1 \leq i \leq n$ of length $1$ from $\sigma_w$.

We note that, for $w \in W$, the marked permutation $\sigma_w$ must always have
an even number of overlined entries.
We also note
that the notation $\sigma_w$ is not the cycle decomposition of
the permutation $\pi_w$ of the unit vectors $\pm e_i, 1 \leq i \leq n$ of $\mathbb{R}^n$ that
is also naturally associated with $w$. In fact, each cycle of length $k$ in $\sigma_w$ corresponds to either two cycles of length $k$ or a single cycle of length $2k$ in $\pi_w$.

\begin{example}\label{ExampleMarkedPerm}

Let \begin{small}$w = \begin{pmatrix}
0 & 0 & -1 & 0 & 0 & 0\\
0 & 0 & 0 & 0 & 1 & 0\\
0 & 0 & 0 & 1 & 0 & 0\\
1 & 0 & 0 & 0 & 0 & 0\\
0 & 0 & 0 & 0 & 0 & -1\\
0 & 1 & 0 & 0 & 0 & 0
\end{pmatrix}$\end{small} be an element of $W$ of type $D_6$. Using Notation~\ref{NotnMarkedPerm}, we have $w = (\overline{1},3,4)(2,\overline{5},6)$, while we have $\pi_w=(1,-3,-4,-1,3,4)(2,5,-6,-2,-5,6)$ where, for brevity, we label the unit vectors $\pm i$ with $1 \leq i \leq 6$. .
\end{example}

We set the following convention for the remainder of the article.

\begin{convention}\label{ConventionNotations}
	When $i = i'$, we interpret each cycle
	$(i',i'+1,\dotsc,\overline{i})$, $(i',i'-1,\dotsc,\overline{i})$, $(\overline{i},i-1,\dotsc,i')$, or $(\overline{i},i+1,\dotsc,i')$ as the 1-cycle $(\overline{i})$. 
	By convention, we also set $\overline{\overline{i}} = i$ for any positive integer~$i$.
	We also suppose that a decreasing-index cycle of the form $(x_i, x_{i-1}, \dotsc, x_{i'})$ is the identity element when $i  \leq i'$ and an increasing-index expression of the form $(x_i, x_{i+1}, \dotsc, x_{i'})$ is the identity element when $i \geq  i'$. Finally, we also assume that a cycle of length $\geq 3$ that contains $n$ should start by $n$.
\end{convention}

\begin{lemma}\label{PropDnReflections}

The set $\Refl$ of reflections in $W$ 
is represented by the set of elements
\[ \{ (i,j),\,(\overline{i},\overline{j})\ |\ 1\leq i\neq j \leq n\} .\]
\end{lemma}

Note that the reflection $(i,j)$ represents a transposition matrix 
whose entries are all 1, while $(\overline{i},\overline{j})$ represents a matrix derived from the previous one by changing the signs in rows $i$ and $j$. 
For $2 \leq i \leq n$, 
we denote by $s_i$ the reflection $(i-1,i)$.\\

The following lemma is straightforward to prove.

\begin{lemma}\label{LemmaIntersectIndices}

Let $\refl$ and $\refl'$ be two reflections in $W$,
with 
$\refl = (i,j)$ or $(\overline{i},\overline{j})$ and $\refl' = (k,l)$ or $(\overline{k},\overline{l})$.
If $\{i,j\}$ does not intersect $\{k,l\}$, then the reflections $\refl$ and $\refl'$ commute. If the cardinality of the intersection is $1$, then we get $(i,j)(k,l)(i,j) = (k,l)(i,j)(k,l)$.

\end{lemma}

\subsection{Length function over the set of reflections}\label{SubsectionLengthShi}

Shi computed in \cite{ShiReflectionLength} the length function over the set of reflections in the infinite series of complex reflection groups. The Coxeter group $W$ of type $D_n$ corresponds to the group $G(2,2,n)$. The length function over the set of reflections in $G(2,2,n)$ appears in Corollary 3.2 in \cite{ShiReflectionLength}. Let us recall this result.

\begin{proposition}\label{PropLengthShi}

Let $w \in G(2,2,n)$, and suppose that $w$ is represented by a marked 
permutation $\sigma_w$ as described in Notation~\ref{NotnMarkedPerm}
Suppose that $\sigma_w$ is written as a product of $r$ cycles,
and define $e$ to be the number of these cycles that have an even number of overlined entries. Then, the length $\ell_{\Refl}(w)$ over $\Refl$ of $w$ is equal to $n-e$. 
\end{proposition}

\begin{example}

Let $w \in G(2,2,6)$ be as in Example \ref{ExampleMarkedPerm}. Then $w = (\overline{1},3,4)(2,\overline{5},6)$. Here we have $C_1 = (\overline{1},3,4)$ and $C_2 = (2,\overline{5},6)$. Both cycles have an odd number (equal to $1$) of overlined entries. Hence $e = 0$ and we get $\ell_{\Refl}(w) = 6$.

\end{example}

Note also that the $n \times n$ identity matrix corresponds to the marked 
permutation with $n$ cycles each containing a single entry $i$. Each
cycle then contains 0 overlined entries. So $e=n$ and we see that the
$\ell_{\Refl}(Id)=0$.

\subsection{Quasi-Coxeter elements in type $D_n$}\label{SubsectionQuasiCoxeter}

Let $W$ be a Coxeter group of type $D_n$. By Carter\cite{Carter}, $W$ contains $\lfloor{\frac{n}{2}}\rfloor$ conjugacy classes of quasi-Coxeter elements. We fix an integer $m$ with $1 \leq m \leq \lfloor n/2 \rfloor$;
this fixes a conjugacy class of quasi-Coxeter elements in $W$. The $m$-th conjugacy class is associated by Carter \cite{Carter} with the diagram $\Delta_{m,n}$  displayed in Figure~\ref{FigureCarterDiagramDn}. When there is no confusion, we denote $\Delta_{m,n}$ by $\Delta$.

\begin{figure}[H]
\begin{tikzpicture}

\node[draw, shape=circle, label=below:$s_2$] (2) at (0,0) {};
\node[draw, shape=circle, label=below:$s_3$] (3) at (1,0) {};
\node[] (0) at (2,0) {};
\node[] (00) at (3,0) {};
\node[draw, shape=circle, label=below:$s_{m-1}$] (4) at (4,0) {};
\node[draw, shape=circle, label=above:$s_m$] (5) at (5,0) {};

\node[draw, shape=circle, label=above:$s_{m+1}$] (6) at (6,1) {};
\node[draw, shape=circle, label=below:$s_1$] (7) at (6,-1) {};
\node[draw, shape=circle,label=above:$s_{m+2}$] (8) at (7,0) {};
\node[draw, shape=circle,label=below:$s_{m+3}$] (9) at (8,0) {};
\node[] (10) at (9,0) {};
\node[] (11) at (10,0) {};
\node[draw,shape=circle, label=below:$s_{n-1}$] (12) at (11,0) {};
\node[draw, shape=circle, label=below:$s_n$] (13) at (12.2,0) {};


\draw[thick,-] (2) to (3);
\draw[thick,-] (3) to (0);
\draw[thick,dashed,-] (0) to (00);
\draw[thick,-] (00) to (4);
\draw[thick,-] (4) to (5);
\draw[thick,-] (5) to (6);
\draw[thick,-] (5) to (7);
\draw[thick,-] (8) to (6);
\draw[thick,-] (8) to (7);
\draw[thick,-] (8) to (9);
\draw[thick,-] (9) to (10);
\draw[thick,dashed,-] (10) to (11);
\draw[thick,-] (11) to (12);
\draw[thick,-] (12) to (13);

\end{tikzpicture}\caption{Carter diagram $\Delta_{m,n}$ of type $D_n$.}\label{FigureCarterDiagramDn}
\end{figure}
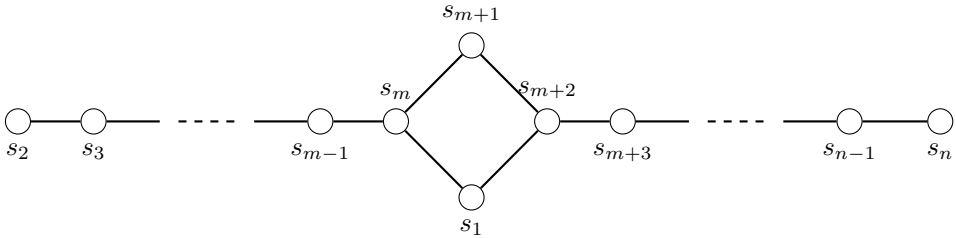

In $\Delta$ an edge between two nodes $s_i$ and $s_j$ describes the relation $s_is_js_i = s_js_is_j$, and 
when there is no edge between $s_i$ and $s_i$, this means that the two reflections commute. In the next proposition, we choose a particular representative of each conjugacy class of quasi-Coxeter elements that will be helpful in the description of our main result.

\begin{proposition}\label{PropConjQCoxBipartite}
Choose the reflections  $s_1:= (\overline{m},\overline{m+1})$ and $s_i:= (i-1,i)$ for $2 \leq i \leq n$.
The $m$-th conjugacy class of quasi-Coxeter elements contains a representative $$w = (m, m-1, \dotsc, 2, \overline{1})(n, n-1, \dotsc, \overline{m+1}).$$ The element $w$ can be written as the product $s_2 s_3 \dotsc s_m s_1 s_{m+1}s_{m+2} s_{m+3} s_{m+4} \dotsc s_n$.

\end{proposition}

\begin{proof}
See Proposition 25 in \cite{Carter} for representatives of the conjugacy classes, where Carter defines the notion of signed cycle-type. The second sentence of the proposition is readily checked.
\end{proof}

We call  the set $\{s_1, s_2, \dotsc, s_n\}$ a Carter generating set. As $w$ is a quasi-Coxeter element, every Carter generating set generates the Coxeter group.

Note that the Carter diagram $\Delta$ contains $m-2$  and $n-m-2$ vertices on the left- and right-hand sides of the single $4$-cycle within $\Delta$, respectively.

For $m=1$, the element $w$ is a Coxeter element and $\Delta_{1,n}$ is the Coxeter diagram of type $D_n$, and $w$ becomes $ (\overline{1})(n, n-1, \dotsc, \overline{2})$.
We call a proper Carter diagram of type $D_n$
a diagram $\Delta_{m,n}$ that is not the Coxeter diagram of type $D_n$, that is with $m \geq 2$.

The Carter diagram $\Delta$  is the decomposition diagram (see Definition~\ref{DefDiagramRedDecomp}) related to the reduced decomposition $s_2 s_3 \dotsc s_m s_1 s_{m+1}s_{m+2} s_{m+3} s_{m+4} \dotsc s_n$ of the quasi-Coxeter element $w$. We are using this particular decomposition of the quasi-Coxeter element since it will be helpful to describe the divisors of length $n-1$ of $w$ in Section~\ref{SecMaxDivisors} and to describe necessary combinatorial techniques for our analysis in Sections~\ref{SecMaxDivisors} and~\ref{SecDecompReflections}.

\begin{example}
	The element $w = (3,2, \overline{1})(6, 5, \overline{4})$ is a representative of a conjugacy class of quasi-Coxeter elements in type $D_6$. In this case $m=3$.
\end{example}

\begin{lemma}\label{LemDynkinDiag}
If $h \in W$ has a reduced  reflection decomposition $h = t_1\dotsc t_r$, $t_i \in T$ whose decomposition diagram is a Coxeter diagram, then this diagram is the Carter diagram of the conjugacy class which contains $h$, and $h$ is a Coxeter element in $ \langle t_1, \dotsc , t_r\rangle$.
\end{lemma}

\begin{proof}
According to Dyer \cite[Theorem~(3.3)]{ Dyer1} the reflection subgroup $W_h:= \langle t_1, \dotsc , t_r\rangle$ of $W$ is a Coxeter group itself.
If the decomposition diagram of $h$ is a Coxeter diagram, we can choose the signs of the roots $\alpha_1, \dotsc,\alpha_r$ related to the reflections $t_1, \dotsc, t_r$ such that their dihedral angles are obtuse,
as the diagram does not contain any cycles.  Therefore $\{\alpha_1, \dotsc ,\alpha_r\}$ is a simple system for $W_h$, see \cite[Lemma~4.1]{BW}. This yields that $h$ is a Coxeter element in $W_h$. Therefore the decomposition diagram is  the Carter diagram of the conjugacy class which contains $h$.
\end{proof}

The next result is a consequence of Theorem~3.10 of Cameron--Seidel--Tsaranov \cite{Cameron}.

\begin{proposition}\label{PropCameron}

	The Coxeter group has a presentation with set of generators the Carter generators. The relations are $s_i^2 = 1$ for $1 \leq i \leq n$ and the relations described by $\Delta_{m,n}$ together with the cycle commutator relation $$[s_m,s_{m+1}s_{m+2}s_1s_{m+2}s_{m+1}] = (s_ms_{m+1}s_{m+2}s_1s_{m+2}s_{m+1})^2 = 1.$$

%
%

\end{proposition}

We end this section by the following statement that will be used to construct the homomorphism $f$ introduced in Step 1 of the strategy of our proof in Section~\ref{SubsectionStrategyProof}.

	\begin{proposition}\label{PropPreparationHomoF}
		\begin{itemize}
			\item[(1)] We have that  $s_is_j \preceq w$ and $s_is_j $ is of order $2$, for $|i-j| > 1$.
			\item[(2)] We have that $s_{i}s_{i+1} \preceq w$, and $s_{i}s_{i+1} $ is of order $3$, for $2 \leq i  \leq n-1$.
			\item[(3)] Let $2 \leq i \leq n$. We have that $s_1s_m \preceq w$, $s_{m+2}s_1 \preceq w$, and $s_1 s_i \preceq w$.
			Further the elements $s_1s_m$ and $s_{m+2}s_1$ are of order $3$, and $s_1 s_i $ is of order $2$.
			\item[(4)] Let $\refl = s_{1}^{s_ms_{m+1}} = (\overline{m-1},\overline{m})$. We have that $\refl s_{m+2} \preceq w$ and $\refl s_{m+2}$ is of order~$2$.
		\end{itemize}
		
\end{proposition}

\begin{proof}
	The result is an immediate consequence of the Hurwitz action and of the choice of the elements $s_i$.
\end{proof}


\section{Maximal divisors of quasi-Coxeter elements}\label{SecMaxDivisors}

As we pointed out in the strategy of our proof (Step 2), our method depends on an analysis of maximal divisors
of a quasi-Coxeter element $w$, and in particular of the decomposition of each 
such as a product of $n-1$ reflections.
In Section~\ref{SubsectionDivisorsln-1} we identify 11 different cases for
such maximal divisors $w_0$, which fall into three types, I, II and III, 
and then in the following sections,
we find reduced decompositions for elements $w_0$ of type I 
(in Section~\ref{SubsectionRedDecompI}), and of types II and III
(in Section~\ref{SubsectionRedDecompII}), as well as their decomposition diagrams (see Definition~\ref{DefDiagramRedDecomp}).

\subsection{Divisors of length $n-1$}\label{SubsectionDivisorsln-1}


Let $w = (m, m-1, \dotsc, 2, \overline{1})(n, n-1, \dotsc, \overline{m+1})$ be a quasi-Coxeter element. 
The elements of length $n-1$ that divide $w$ consist of all the products $w (i,j)$ and $w (\overline{i},\overline{j})$ for which $1 \leq i < j \leq n$. We denote by $w_0$ a divisor of length $n-1$ of $w$. We compute these divisors in Equations \ref{Eqt1} to \ref{Eqt11} below. We distinguish $3$ types that we denote by I, II, and III and that are displayed in the following Tables \ref{TableTypeI},~\ref{TableTypeII},~and~\ref{TableTypeIII}. The first column of each table represents the cases for $i$ and $j$. The second column is the divisor $w_0$. Notice that we get from type II to type III by applying symmetry.\\

Remark also that Equation \ref{Eqt2} is similar to Equation \ref{Eqt1}; the difference is that
two entries are further overlined in Equation \ref{Eqt2}. 
We see the same similarities between Equations \ref{Eqt6} and \ref{Eqt7}, and Equations \ref{Eqt9} and \ref{Eqt10}.

Notice that each element $w_0$ of the 11 Equations admits exactly one cycle with an even number of overlined elements (we assume that $0$ is even). Hence each element is of length $n-1$ by Proposition~\ref{PropLengthShi}. In Sections~\ref{SubsectionRedDecompI} and~\ref{SubsectionRedDecompII}, we describe a reduced decomposition over the set $\Refl$ of reflections for each divisor $w_0$ of $w$  of type I and of types II and III, respectively.\\

We provide an example where we explicitly write the monomial matrices.

\begin{example}\label{ExampleTypes}
	
	Let $W$ be a Coxeter group of type $D_5$. Let $m=2$ and let $w = (2,\overline{1})(5,4,\overline{3})$ be a proper quasi-Coxeter element. As a monomial matrix,
	\begin{center}
		$w = \begin{pmatrix}
			0 & -1 & 0 & 0 & 0\\
			1 & 0 & 0 & 0 & 0\\
			0 & 0 & 0 & 0 & -1\\
			0 & 0 & 1 & 0 & 0\\
			0 & 0 & 0 & 1 & 0\\
		\end{pmatrix}$.
	\end{center}
	Let us multiply $w$ from the right by the transposition $(1,4)$ as described in  Equation (1). So we get
	\begin{center}
		$w (1,4) = \begin{pmatrix}
			0 & -1 & 0 & 0 & 0\\
			0 & 0 & 0 & 1 & 0\\
			0 & 0 & 0 & 0 & -1\\
			0 & 0 & 1 & 0 & 0\\
			1 & 0 & 0 & 0 & 0\\
		\end{pmatrix}$.
	\end{center}
	
	Using the marked permutation notation introduced in Notation~\ref{NotnMarkedPerm}, we have that $w (1,4) = (5,\overline{1},2,4,\overline{3})$  (the coefficient is equal to $-1$ on row numbers $1$  and $3$ of the matrix $w_1$) which is compatible with the result of Equation \ref{Eqt1}.
\end{example}

\newpage
\newgeometry{paperwidth=18cm,
	left=1.5cm,
	right=2cm,
	paperheight=32cm,
	top=1cm,
	bottom=2cm}

\begin{table}[]
	\centering
	\begin{small}
	\begin{tabular}{|p{2.5cm}|p{12cm}|}
		\hline
		\vfill \mbox{$1 \leq i \leq m$,}\linebreak \mbox{$(m+1) \leq j \leq n$} &
		\begin{equation}\label{Eqt1}
			w(i,j) = (n, n-1, \dotsc, j+1, i, i-1, \dotsc, \overline{1}, m, m-1, \dotsc, i+1,  j,  j-1, \dotsc, \overline{m+1}).
		\end{equation}\\
		\hline
		\vfill \mbox{$i \neq m$,}\linebreak \mbox{$j \neq n$} &
		\begin{equation}\label{Eqt2}
			w(\overline{i},\overline{j})=(n, n-1, \dotsc, \overline{j+1}, i, i-1, \dotsc, \overline{1}, m, m-1, \dotsc, \overline{i+1}, j, j-1, \dotsc, \overline{m+1}).
		\end{equation}\\
		\hline 
		\vfill \mbox{$i=m$,}\linebreak \mbox{$j \neq n$} &
		\begin{equation}\label{Eqt3}
			w(\overline{m},\overline{j}) = (n, n-1, \dotsc, \overline{j+1}, m, m-1, \dotsc, 1,  j,  j-1, \dotsc, \overline{m+1}).
		\end{equation}\\
		\hline
		\vfill \mbox{$i \neq m$,} \linebreak \mbox{$j=n$} &
		\begin{equation}\label{Eqt4}
			w(\overline{i},\overline{n}) = (n,  n-1, \dotsc, m+1, i, i-1, \dotsc, \overline{1}, m, m-1, \dotsc, \overline{i+1}).
		\end{equation}\\
		\hline
		\vfill \mbox{$i=m$,} \linebreak \mbox{$j=n$} &
		\begin{equation}\label{Eqt5}
			w(\overline{m},\overline{n}) = (n, n-1, \dotsc, m+1, m, m-1, \dotsc, 1) = (n,n-1, \dotsc, 1).
		\end{equation}\\
		\hline 
	\end{tabular}
	\end{small}
	\caption{Type I: $1 \leq i \leq m$ and $(m+1) \leq j \leq n$.}\label{TableTypeI}
\end{table}

\begin{table}[]
	\centering
	\begin{small}
		\begin{tabular}{|p{2cm}|p{12.5cm}|}
			\hline
			\vfill \mbox{$1 \leq i < j \leq m$} &
			\begin{equation}\label{Eqt6}
				w(i,j)=(m, m-1, \dotsc, j+1, i, i-1, \dotsc, \overline{1})(i+1, j, j-1, \dotsc, i+2)(n, n-1, \dotsc, \overline{m+1}).
			\end{equation}\\
			\hline
			\vfill \mbox{$j \neq m$} &
			\begin{equation}\label{Eqt7}
				w(\overline{i},\overline{j})=
				(m, m-1, \dotsc, \overline{j+1}, i, i-1, \dotsc, \overline{1})( \overline{i+1},  j, j-1, \dotsc, i+2)(n, n-1, \dotsc, \overline{m+1}).
			\end{equation}\\
			\hline 
			\vfill \mbox{$j=m$} &
			\begin{equation}\label{Eqt8}
				w(\overline{i},\overline{m})=
				(i, i-1, \dotsc, 1)( \overline{i+1}, m, m-1, \dotsc, i+2)(n, n-1, \dotsc, \overline{m+1}).
			\end{equation}\\
			\hline
		\end{tabular}
	\end{small} \caption{Type II: $1 \leq i < j \leq m$.}\label{TableTypeII}
\end{table}

\begin{table}[]
	\centering
	\begin{small}
		\begin{tabular}{|p{2cm}|p{12.5cm}|}
			\hline
			\vfill \mbox{$(m+1) \leq i <$} \linebreak
				 \mbox{$j \leq n$} &
			\begin{equation}\label{Eqt9}
				w(i,j)=
				(m, m-1, \dotsc, \overline{1})(n, n-1, \dotsc, j+1, i, i-1, \dotsc,  \overline{m+1})(j, j-1, \dotsc, i+1).
			\end{equation}\\
			\hline
			\vfill \mbox{$j \neq n$} &
			\begin{equation}\label{Eqt10}
				w(\overline{i},\overline{j})=
				(m, m-1, \dotsc, \overline{1})(n, n-1, \dotsc, \overline{j+1}, i, i-1, \dotsc,  \overline{m+1})(j, j-1, \dotsc, \overline{i+1}).
			\end{equation}\\
			\hline 
			\vfill \mbox{$j=n$} &
			\begin{equation}\label{Eqt11}
				w(\overline{i},\overline{n})=
				(m, m-1, \dotsc, \overline{1})(i, i-1, \dotsc, m+1)(n, n-1, \dotsc, \overline{i+1}).
			\end{equation}\\
			\hline
		\end{tabular}
	\end{small} \caption{Type III: $(m+1) \leq i < j \leq n$.}\label{TableTypeIII}
\end{table}

\restoregeometry

\subsection{Reduced decompositions and diagrams for type I}\label{SubsectionRedDecompI}

Suppose that $w_0$ has type I (see Table~\ref{TableTypeI}). As a marked permutation, it is a cycle of the form $(x_1,x_2,x_3, \dotsc, x_n)$, where each $x_k$ is equal to $p$ or $\overline{p}$ ($1 \leq p \leq n$),
with $\{x_1,x_2,\dotsc,x_n\} = \{1,2,\dotsc,n\}$, and with an even number of overlined entries (see Equations \ref{Eqt1} to \ref{Eqt5}). We will describe how to produce a reduced reflection decomposition of length $n-1$ for this element.

We continue the study of Example~\ref{ExampleTypes} that will help the understanding of a procedure that describes the reduced decompositions. The general idea is to multiply the marked permutation $w_0 = (x_1,x_2,x_3, \dotsc, x_n)$ from the right by a sequence of reflections in order to obtain the identity matrix. A decomposition of $w_0$ is given by the product in reverse order of all the reflections used in the procedure. It turns out that this decomposition is reduced.

\begin{example}\label{ExampleRedDecomp1}\normalfont

Let $w = (2,\overline{1})(5,4,\overline{3})$ and  $w_0 = w(1,4) = (5,\overline{1},2,4,\overline{3})$ be as in Example~\ref{ExampleTypes}.
We follow the cycle $(5,\overline{1},2,4,\overline{3})$. The first two entries are $5$ and $\overline{1}$. We multiply $w_0$ from the right by the transposition $(1,5)$ and get $$w_1 = w_0(1,5) = \begin{small}\begin{pmatrix}
0 & -1 & 0 & 0 & 0\\
0 & 0 & 0 & 1 & 0\\
-1 & 0 & 0 & 0 & 0\\
0 & 0 & 1 & 0 & 0\\
0 & 0 & 0 & 0 & 1
\end{pmatrix}\end{small} = (\overline{1},2,4,\overline{3})(5),$$ 
and $5$ becomes a fixed point.

We continue by following the cycle $(\overline{1},2,4,\overline{3})$. The first two entries are $\overline{1}$ and $2$. We multiply $w_1$ from the right by the reflection $(\overline{1},\overline{2})$ and get $$w_2 = w_1(\overline{1},\overline{2}) = \begin{small}\begin{pmatrix}
1 & 0 & 0 & 0 & 0\\
0 & 0 & 0 & 1 & 0\\
0 & 1 & 0 & 0 & 0\\
0 & 0 & 1 & 0 & 0\\
0 & 0 & 0 & 0 & 1
\end{pmatrix}\end{small} = (2,4,3)(5)(1).$$ 
This operation permutes the columns $1$ and $2$ and multiplies their entries by $-1$. This then yields that $1$ becomes a fixed point.
Remark that by this operation, the third entry in $(2,4,3)(5)(1)$ is not overlined anymore.

We continue by following the cycle $(2,4,3)$. Here, the first two entries are $2$ and $4$. Then, we multiply $w_2$ from the right by the transposition $(2,4)$. We therefore obtain a coefficient equal to $1$ in diagonal position $[2,2]$. We get
 $$w_3 = w_2(2,4) = \begin{small}\begin{pmatrix}
1 & 0 & 0 & 0 & 0\\
0 & 1 & 0 & 0 & 0\\
0 & 0 & 0 & 1 & 0\\
0 & 0 & 1 & 0 & 0\\
0 & 0 & 0 & 0 & 1
\end{pmatrix}\end{small} = (4,3)(5)(1)(2).$$ 

Now, we arrive at the last step. We multiply $w_3$ from the right by the transposition $(3,4)$ and finally obtain the identity matrix denoted by $I_5 = (5)(1)(2)(4)(3)$.

This implies that $w_0 = (3,4)(2,4)(\overline{1},\overline{2})(1,5) = s_4(2,4)(\overline{1},\overline{2})(1,5)$. These are the reflections that we used before in reverse order. By Proposition \ref{PropLengthShi}, the length of $w_0$ is equal to $4$. Hence $(3,4)(2,4)(\overline{1},\overline{2})(1,5)$ is a reduced decomposition of $w_0$ as it consists of $4$ reflections. 
By Lemma~\ref{LemmaIntersectIndices}, the related decomposition diagram is described by the following standard type $A_4$ diagram.

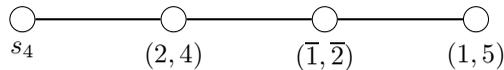
\begin{center}
\begin{tikzpicture}

\node[draw, shape=circle, label=below:$s_4$] (1) at (0,0) {};
\node[draw, shape=circle, label=below:{$(2,4)$}] (2) at (2,0) {};
\node[draw, shape=circle, label=below:{$(\overline{1},\overline{2})$}] (3) at (4,0) {};
\node[draw, shape=circle, label=below:{$(1,5)$}] (4) at (6,0) {};

\draw[thick,-] (1) to (2);
\draw[thick,-] (3) to (2);
\draw[thick,-] (3) to (4);

\end{tikzpicture}
\end{center}

\end{example}

The general procedure is as follows.

\begin{procdur}\label{ProcedureTypeI}

Let $x = (x_1, x_2, \dotsc, x_r)$. For $k$ from $1$ to $r-1$,
\begin{itemize}
\item If $x_k= p$, then whether $x_{k+1} = q$ or $\overline{q}$, we multiply $(x_k,x_{k+1}, \dotsc, x_r)$ from the right by the reflection $(p,q)$ 
and get $(x_k)(x_{k+1},x_{k+2}, \dotsc, x_r)$,
whose length is one less than the length of $(x_k,x_{k+1}, \dotsc, x_r)$.
\item If $x_k = \overline{p}$, then whether $x_{k+1} = q$ or $\overline{q}$, we multiply $(x_k,x_{k+1}, \dotsc, x_r)$ from the right by the reflection $(\overline{p},\overline{q})$
and get $(x_k)(x_{k+1},x_{k+2}, \dotsc, \overline{x_r})$,
whose length is one less than the length of $(x_k,x_{k+1}, \dotsc, x_r)$.
\end{itemize}

\end{procdur}

Remark that the entry $\overline{x_r}$ can be equal to $\overline{\overline{u}}$ for $1 \leq u \leq n$ in which case it is equal to $u$, by the convention that $\overline{\overline{i}} = i$ for any positive integer $i$ (see Convention~\ref{ConventionNotations}).

\begin{proposition}\label{PropReducedDecompI}

Let $w_0$ be a divisor of $w$ of type I. 
A reduced decomposition of $w_0$ is obtained as the product in reverse order of the reflections that are applied in Procedure \ref{ProcedureTypeI}.

\end{proposition}

\begin{proof}
Let $w_0$ be a divisor of $w$ of type I. It is of the form $(x_1,x_2,x_3, \dotsc, x_n)$. 
Applying Procedure \ref{ProcedureTypeI} for all $k$ from $1$ to $n-1$, 
the element $(x_1,x_2,x_3, \dotsc, x_n)$ is transformed to the identity matrix 
$(x_1)(x_2) \dotsc (x_n)$ in $n-1$ steps. Since all reflections are of order 
$2$, a decomposition of the element $(x_1,x_2,x_3, \dotsc, x_n)$ is given by 
the product in reverse order of all the reflections used in this procedure.

Since $\ell(w_0) = n-1$ by Proposition \ref{PropLengthShi}, the decomposition we obtain is reduced (as it consists of $n-1$ reflections).
\end{proof}

\begin{proposition}\label{PropDiagTypeI}
	
	The decomposition diagram of each reduced decomposition corresponding to Type I is represented by a Coxeter diagram of type $A_{n-1}$, and $w_0$ is a Coxeter element in $P_{w_0}$. In particular, $w_0$ is a parabolic Coxeter element in $W$.
	
\end{proposition}

\begin{proof}
	It is an immediate consequence of Proposition~\ref{PropReducedDecompI} that the decomposition diagram of the reduced decomposition of $w_0$ produced in Procedure~\ref{ProcedureTypeI} is a string. Therefore, Lemmas~\ref{LemDynkinDiag} and \ref{LemParaSub} yield the assertion.
\end{proof}

We apply Proposition~\ref{PropReducedDecompI}
to establish decompositions of type I divisors $w_0$ of $w$, as described in the following lemmas.
 
\begin{lemma}\label{LemRedDecEqt1}
Let $w_0:=w(i,j)$ be a divisor of $w$ of type I, where $1\leq i\leq m$, $m+1\leq j \leq n$, so that we are in the situation of Equation~\ref{Eqt1}. Then $w_0$ has one of the following decompositions of length $n-1$ \mbox{over $\Refl$.}
\begin{eqnarray*}
(1)&& \hbox{\rm Suppose } i \neq m, j\neq n-2,n-1,n:\\
&&w_0 = s_{m+2} s_{m+3} \dotsc s_{j-1}s_j (i+1,j) s_{i+2} s_{i+3} \dotsc s_{m-1}s_m  (\overline{1},\overline{m}) s_2 s_3 \dotsc\\
&&\dotsc s_{i-1} s_i (i,j+1)s_{j+2} \dotsc s_{n-1}s_n.\\
(2)&&\hbox{\rm Suppose}\,i \neq m, j=n-2:\\
&& w_0 = s_{m+2} \dotsc s_{n-2}s_{n-1}(i+1,n-2)s_{i+2} 
\dotsc s_{m-1}s_m(\overline{1},\overline{m})s_2 \dotsc s_{i-1}s_i(i,n-1)s_n.\\
(3)&&\hbox{\rm Suppose }\,i =m, j=n-2:\\
&& w_0 = s_{m+2} \dotsc s_{n-3}s_{n-2}(\overline{1},\overline{n-2})
s_2 \dotsc s_{m-1}s_m(m,n-1)s_n.\\
(4)&&\hbox{\rm Suppose }\,i \neq m, j=n-1:\\
&& w_0 = s_{m+2} \dotsc s_{n-2}s_{n-1}(i+1,n-1)s_{i+2} 
\dotsc s_{m-1}s_m (\overline{1},\overline{m})s_2 \dotsc s_{i-1}s_i(i,n).\\
(5)&&{\rm Suppose }\,i =m, j=n-1:\\
&& w_0 = s_{m+2} \dotsc s_{n-2}s_{n-1}(\overline{1},\overline{n-1}) s_2 \dotsc s_{m-1}s_m(m,n).\\
(6)&&{\rm Suppose }\,i \neq m, j=n:\\
&& w_0 = s_{i+2} \dotsc s_{m-1}s_m (\overline{1},\overline{m}) s_2
\dotsc s_{i-1} s_i (\overline{i},\overline{m+1})s_{m+2} \dotsc s_{n-1}s_n.\\
(7)&&{\rm Suppose }\,i =m, j=n:\\
&&w_0 = s_2 \dotsc s_{m-1}s_m s_1 s_{m+2} \dotsc s_{n-1}s_n.
\end{eqnarray*}
\end{lemma}

\begin{proof}
We apply
Proposition~\ref{PropReducedDecompI} to Equation \ref{Eqt1}.
\end{proof}

\begin{lemma}\label{LemRedDecEqt2}
Let $w_0=w(\overline{i},\overline{j})$
be a divisor of $w$ of type I,
where $1\leq i<m$, $m+1\leq j <n$, so that we are in the situation of Equation~\ref{Eqt2}.
Then $w_0$ has one of the following decompositions of length $n-1$ over $\Refl$.
\begin{eqnarray*}
(1)&& \hbox{\rm Suppose }\,i \neq m,\,j < n-2:\\
&&w_0 = s_{m+2} \dotsc s_{j-1}s_{j}(\overline{i+1},\overline{j})s_{i+2} \dotsc s_{m-1}s_m(\overline{1},\overline{m})s_2 \dotsc s_{i-1}s_{i}(\overline{i},\overline{j+1})s_{j+2} \dotsc s_n.\\
(2)&& \hbox{\rm Suppose }\,i \neq m,\,j = n-2:\\
&& w_0 = s_{m+2} \dotsc s_{n-3}s_{n-2}(\overline{i+1},\overline{n-2})s_{i+2} \dotsc s_{m-1}s_m(\overline{1},\overline{m})s_2 \dotsc 
s_{i-1}s_{i}(\overline{i},\overline{n-1}) s_n.\\
(3)&&\hbox{\rm Suppose }\,i \neq m,\,j = n-1:\\
&& w_0 = s_{m+2} \dotsc s_{n-2}s_{n-1}(\overline{i+1},\overline{n-1})s_{i+2} \dotsc s_{m-1}s_m(\overline{1},\overline{m})s_2 \dotsc s_{i-1}s_{i}(\overline{i},\overline{n}).
\end{eqnarray*}
\end{lemma}
\begin{proof}
We apply Proposition~\ref{PropReducedDecompI}
to Equation \ref{Eqt2}.
\end{proof}

\begin{lemma}\label{LemRedDecEqt3}
Let $w_0=w(\overline{m},\overline{j})$
be a divisor of $w$ of type I,
where $m+1\leq j <n$, so that we are in the situation of Equation~\ref{Eqt3}.
Then $w_0$ has one of the following decompositions of length $n-1$ over $\Refl$.
\begin{eqnarray*}
(1)&&\hbox{\rm Suppose }\,j < n-2:\\
&& w_0 = s_{m+2} \dotsc s_{j-1}s_{j}(1,j) \dotsc s_{m-1}s_m(\overline{m},\overline{j+1})s_{j+2} \dotsc s_{n-1}s_n.\\
(2)&& \hbox{\rm Suppose }\,j = n-2:\\
& & w_0 = s_{m+2} \dotsc s_{n-3}s_{n-2}(1,n-2)s_2 \dotsc s_{m-1}s_{m}(\overline{m},\overline{n-1})s_n.\\
(3)&&\hbox{\rm Suppose }\,j = n-1:\\
&& w_0 = s_{m+2} \dotsc s_{n-2}s_{n-1}(1,n-1)s_2 \dotsc s_{m-1}s_{m}(\overline{m},\overline{n}).  
\end{eqnarray*}
\end{lemma}

\begin{proof}
We apply Proposition~\ref{PropReducedDecompI} to Equation \ref{Eqt3}.
\end{proof}

Similarly, we get the next result.

\begin{lemma}\label{LemRedDecEqt45}

Let $w_0$ be as in the situation of Equations~\ref{Eqt4} and \ref{Eqt5}. Then $w_0$ has the following decompositions of length $n-1$ over $\Refl$.
\begin{eqnarray*}
	(1)&&\hbox{\rm  Suppose $i \neq m$ and $j = n$, so that we are in the situation of Equation~\ref{Eqt4}:}\\
	&& w_0 = s_{i+2} \dotsc s_{m-1}s_m(\overline{1},\overline{m})s_2 \dotsc s_{i-1}s_i (i,m+1) s_{m+2} \dotsc s_{n-1}s_n\\
	(2)&& \hbox{\rm Suppose $i = m$ and $j = n$, so that we are in the situation of Equation~\ref{Eqt5}:}\\
	& & w_0 = s_2 s_3 \dotsc s_n.
\end{eqnarray*}
\end{lemma}

\subsection{Reduced decompositions and diagrams for types II and III}\label{SubsectionRedDecompII} 

In this section, we find reduced decompositions for the maximal divisors $w_0$ of $w$ that are of types II and III. They are listed in Equations~\ref{Eqt6} to \ref{Eqt11}.

We define a combinatorial technique that enables us to obtain a reduced 
decomposition, whose decomposition diagram $\Delta_0$ is the union of 
 Coxeter diagrams of type $A$ or $D$, or a proper Carter diagram of type $D$. 

First, observe that each element $w_0$ defined in  one of the 
Equations~\ref{Eqt6}--~\ref{Eqt11} is the product of three cycles. 
Since $\ell(w_0) = n-1$, by Proposition \ref{PropLengthShi}  each $w_0$ 
admits exactly one cycle with an even number of overlined elements. The other two cycles contain an odd 
number of overlined elements. Observing these equations, we recognise that 
these cycles contain exactly one overlined element at the end of the cycle. 
Assume that the cycles are $x:= (x_1,x_2, \dotsc, x_p), y:= (y_1, y_2, \dotsc, y_q), 
z:= (z_1, z_2, \dotsc, z_r)$ such that an even number of entries of 
$(z_1, z_2, \dotsc, z_r)$ are overlined, $x_p = \overline{u}$, and 
$y_q = \overline{v}$. Also, observe that we always have $p+q+r = n$.

The combinatorial technique is based on Procedure \ref{ProcedureTypeI}. We formulate it in the following procedure.

\begin{procdur}\label{ProcedureTypesII,III}
\noindent
\begin{itemize}
\item[Step 1.] If $r \geq 2$, then apply Procedure \ref{ProcedureTypeI} to the cycle $(z_1, z_2, \dotsc, z_r)$. We obtain $(z_1)(z_2) \dotsc (z_r)$.

\item[Step 2.] If $p \geq 2$, then apply Procedure \ref{ProcedureTypeI} to the cycle $(x_1,x_2, \dotsc, x_p)$. We obtain $(x_1)(x_2) \dotsc (x_{p-1})(\overline{u})$.

\item[Step 3.] If $q \geq 2$, the apply Procedure \ref{ProcedureTypeI} to the cycle $(y_1,y_2, \dotsc, y_q)$. We obtain $(y_1)(y_2) \dotsc (y_{q-1})(\overline{v})$.
\end{itemize}

\end{procdur}

Furthermore, we impose an additional condition: If one of the two cycles $(x_1,x_2, \dotsc, x_p)$, $(y_1,y_2, \dotsc, y_q)$ contains $n$, then Step 2 deals with the cycle that contains $n$.

\begin{proposition}\label{PropReducedDecompII,III}
Let $w_0$ be a divisor of $w$ of type II or III. We continue with the notations 
introduced at the beginning of the section. Without loss of generality, 
assume that $u<v$. A reduced decomposition of $w_0$ is obtained as the 
product $(u,v) (\overline{u},\overline{v})$ followed by the reflections used in Procedure \ref{ProcedureTypesII,III} in reverse order.
\end{proposition}

\begin{proof}
After application of Procedure \ref{ProcedureTypesII,III}, the monomial 
matrix $w_0$ is transformed to the diagonal matrix with diagonal coefficients 
equal to $1$ everywhere apart from diagonal positions $[u,u]$ and $[v,v]$, 
where the two coefficients are equal to $-1$. Multiplying this diagonal 
matrix by $(u,v)(\overline{u},\overline{v})$, it is transformed to the 
identity matrix. A decomposition of $w_0$ is therefore the product of 
$(u,v)$ by $(\overline{u},\overline{v})$ followed by the reflections used 
in Procedure \ref{ProcedureTypesII,III} in reverse order.

The decomposition is reduced if its length is equal to $n-1$. 
In the first step of Procedure \ref{ProcedureTypesII,III}, the number of 
reflections that have been used is equal to $r-1$, while $p-1$ and $q-1$ reflections are used in each of Steps 2 and 3. In addition, we multiplied at the end by two reflections: $(u,v)$ and $(\overline{u},\overline{v})$. Therefore, the number of reflections used in this decomposition is $(r-1) + (p-1) + (q-1) +2 = (r+p+q)-1 = n-1$.
\end{proof}

We explain Procedure \ref{ProcedureTypesII,III} and Proposition \ref{PropReducedDecompII,III} in the following two examples. The first example corresponds to type II and the second to type III.

\begin{example}\label{ExampleII}\normalfont

We continue with our running Example~\ref{ExampleTypes}, so $n=5$, $m=2$ and $w = (2,\overline{1})(5,4,\overline{3})$. Let $w_0 = w(1,2) = (\overline{1})(2)(5,4,\overline{3})$, whose cycle decomposition 
is given in Equation \ref{Eqt6}.

We apply Procedure \ref{ProcedureTypesII,III} to $w_0$.

\textit{Step 1.} The even cycle is $(2)$ and corrsponds to $r=1$ in Procedure~\ref{ProcedureTypesII,III}. Step 1 does not apply and we move to Step 2.

\textit{Step 2.} The cycle containing a unique overlined entry and containing $n=5$ is $(5,4,\overline{3})$. Here, we have $p=3$. Then, this step applies, and we execute Procedure \ref{ProcedureTypeI} to the cycle $(5,4,\overline{3})$. Therefore, we first multiply $(5,4,\overline{3})$ by $(4,5)$  and then the result  by $(3,4)$ from the right, and  get $(5,4,\overline{3}) (4,5)(3,4) =  (5)(4)(\overline{3})$. We set $w_2:=w_0(4,5)(3,4)$.

\textit{Step 3.} The second cycle contains a unique overlined entry $(\overline{1})$. In this case, we have $q=1$. Therefore, we do not modify $w_2 = (\overline{1})(2)(5)(4)(\overline{3})$ in Step~3. 

By Proposition \ref{PropReducedDecompII,III}, we multiply $w_2$ from the right by $(1,3)(\overline{1},\overline{3})$ so that it is transformed to the identity matrix. A reduced decomposition of $w_0$ is then obtained by 
adding to $(1,3)(\overline{1},\overline{3})$ the reflections used in 
Procedure \ref{ProcedureTypesII,III} in reverse order. 
Hence we obtain $w_0 = (1,3)(\overline{1},\overline{3})(3,4)(4,5)$.

The decomposition diagram associated to this reduced decomposition is a Coxeter diagram of type $D_4$. This is readily checked by Lemma \ref{LemmaIntersectIndices}. The diagram is then the following.

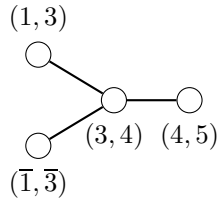
\begin{center}\begin{tikzpicture}
\node[draw, shape=circle, label=below:{$(3,4)$}] (1) at (0,0) {};
\node[draw, shape=circle, label=below:{$(4,5)$}] (2) at (1,0) {};
\node[draw, shape=circle, label=above:{$(1,3)$}] (3) at (-1,0.6) {};
\node[draw, shape=circle, label=below:{$(\overline{1},\overline{3})$}] (4) at (-1,-0.6) {};
\draw[thick,-] (1) to (2);
\draw[thick,-] (1) to (3);
\draw[thick,-] (4) to (1);
\end{tikzpicture}\end{center}

The element $w_0$ is therefore a Coxeter element in the subgroup generated by the reflections $(1,3), (\overline{1},\overline{3}), (3,4), (4,5)$ that compose the reduced decomposition.

\end{example}

\begin{example}\label{ExampleIII}\normalfont

Let $n=6$ and $m=2$. Consider the proper quasi-Coxeter element $w = (2,\overline{1})(6,5,4,\overline{3})$ and $w_0 = w(\overline{3},\overline{6})$. By Equation \ref{Eqt11},  it is $w_0 = (2,\overline{1})(3)(6,5,\overline{4})$.

We apply now Procedure \ref{ProcedureTypesII,III}. The cycle $(3)$ contains only one element.  So, we move to Step~2.

\textit{Step 2.} We apply Procedure \ref{ProcedureTypeI} to the cycle $(6,5,\overline{4})$ that contains $n=6$. We obtain  $(6,5,\overline{4}) (5,6)(4,5) =  (6)(5)(\overline{4})$.
We set $w_2:= w_0 (5,6)(4,5)$.

\textit{Step 3.} We apply Procedure \ref{ProcedureTypeI} to the cycle $(2,\overline{1})$, and get 
$w_3 :=w_2(1,2) =$ $ (2)(\overline{1})(3)(6)(5)(\overline{4})$. 
By Proposition \ref{PropReducedDecompII,III}, a reduced decomposition of 
$w_0$ is obtained by adding to $(1,4)(\overline{1},\overline{4})$ the 
reflections used in Procedure \ref{ProcedureTypesII,III} in reverse order. Therefore, we obtain $w_0 = (1,4)(\overline{1},\overline{4})(1,2)(4,5)(5,6)$.

By Lemma \ref{LemmaIntersectIndices}, the decomposition diagram associated with this reduced decomposition is a proper Carter diagram of type $D_5$. The diagram is the following.

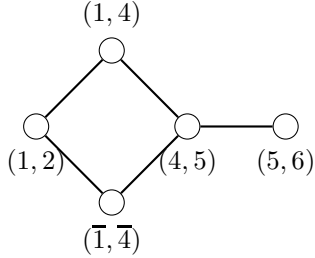
\begin{center}\begin{tikzpicture}
\node[draw, shape=circle, label=below:{$(1,2)$}] (5) at (5,0) {};
\node[draw, shape=circle, label=above:{$(1,4)$}] (6) at (6,1) {};
\node[draw, shape=circle, label=below:{$(\overline{1},\overline{4})$}] (7) at (6,-1) {};
\node[draw, shape=circle,label=below:{$(4,5)$}] (8) at (7,0) {};
\node[draw, shape=circle,label=below:{$(5,6)$}] (9) at (8.3,0) {};
\draw[thick,-] (5) to (6);
\draw[thick,-] (5) to (7);
\draw[thick,-] (8) to (6);
\draw[thick,-] (8) to (7);
\draw[thick,-] (8) to (9);
\end{tikzpicture}\end{center}

We will show later that the element $w_0 = (1,4)(\overline{1},\overline{4})(1,2)(4,5)(5,6)$ 
is a proper quasi-Coxeter element in the subgroup generated by the reflections
$(1,4)$, $(\overline{1}$, $\overline{4})$, $(1,2)$, $(4,5)$, $(5,6)$ that
compose the reduced decomposition.

\end{example}

Now, we characterise the diagrams $\Delta_0$ of the reduced decompositions, and whether the elements $w_0$ are Coxeter
or proper quasi-Coxeter elements in the subgroup $P_{w_0}$ generated by the reflections
that compose the reduced decomposition. In fact, Procedures~\ref{ProcedureTypeI}
and \ref{ProcedureTypesII,III} are tailored in order to obtain a Coxeter
diagram of type $A$ or $D$, or a proper Carter diagram of type $D$.
Observe first the following. We continue to use the notation of $x,y$ and $z$ introduced at the beginning 
of this section.

\begin{lemma}\label{LemParaQuasiCox}
The elements $xy$ and $z$ are parabolic quasi-Coxeter elements in $W$.
\end{lemma}

\begin{proof}
By Proposition~\ref{PropReducedDecompII,III}, we have $\ell_T(z) + \ell_T(w_0z^{-1}) = \ell_T(w_0)$ which yields 
$z, xy \preceq w_0 \preceq w$. Therefore $z$ as well as $xy$ are parabolic quasi-Coxeter elements in $W$, see Corollary~6.11 in \cite{BaumGobet}.
\end{proof}

\begin{proposition}\label{PropCoxDiagTypeII,III}

Let $w_0$ be a divisor of $w$ of type II or III. Let $x= (x_1,x_2, \dotsc, x_p), y= (y_1, y_2, \dotsc, y_q)$ and 
$z= (z_1, z_2, \dotsc, z_r)$ be as introduced before.
Consider the reduced decomposition of
$w_0$ as described in Procedure~\ref{ProcedureTypesII,III} and
Proposition~\ref{PropReducedDecompII,III}.
\begin{itemize}
\item If $p$ and $q$ are equal to $1$, then the diagram of the reduced 
decomposition is the disjoint union of a diagram of type $A_{r-1}$ and
two nodes.
\item If $p=1$ and $q=2$ (or $q=1$ and $p=2$), then the diagram of the 
reduced decomposition is a disjoint union of a type $A_3$ and type $A_{r-1}$ diagrams.
\item If $p=1$ and $q > 2$ (or $q=1$ and $p>2$), then the diagram of the 
reduced decomposition is a disjoint union of a Coxeter diagram of type
$D_{q+1}$ and a Coxeter diagram of type $A_{r-1}$ (or a disjoint union of
diagrams of types $D_{p+1}$ and $A_{r-1}$, respectively).
\end{itemize}
In all these  cases, the element $w_0$ is a Coxeter element in $P_{w_0}$ and therefore a parabolic Coxeter element in $W$.
\end{proposition}

\begin{proof}

By Lemma~\ref{LemParaQuasiCox}, $xy$ as well as $z$ are parabolic quasi-Coxeter elements in $W$. Here we prove that $xy$ and $z$ are Coxeter elements in $P_{xy}$
and $P_z$, respectively.

By Proposition~\ref{PropReducedDecompII,III} and Lemma~\ref{LemDynkinDiag} the cycle $z$ is a Coxeter element of type $A_{r-1}$  in $P_z$. As $xy$ and $z$ are disjoint cycles, the two elements commute.

If $p = q = 1$, then $w_0$  equals $(u,v)(\overline{u},\overline{v}) z$. As $(u,v)$ and $(\overline{u},\overline{v})$ commute, the first bullet of the proposition follows.

If $p = 1$ and $q \geq 2$, then  it is straightforward to check that the decomposition diagram of the reduced decomposition of $xy$ 
given in Proposition~\ref{PropReducedDecompII,III} is a Coxeter diagram of type $D_{q+1}$. Thus, by Lemma~\ref{LemDynkinDiag} 
$xy$ is a Coxeter element of type $D_{q+1}$ in $P_{xy}$.
This yields the other two bullets, as $A_3 = D_3$.
\end{proof}

\begin{proposition}\label{PropQCoxDiagTypeII,III}
	Suppose $p,q \geq 2$. Then the decomposition diagram of $w_0$ is a disjoint union of a proper Carter diagram of type $D_{p+q}$ and a Coxeter diagram of type $A_{r-1}$. Further
		\begin{itemize}
			\item the element $z^\prime:= xy$ is a proper quasi-Coxeter element  of type $D_{p+q}$ in $P_{z^\prime}$,
			\item the decomposition of $z^\prime$ is related to its Carter diagram as described in  Proposition~\ref{PropConjQCoxBipartite},
			\item the element $z$ is a Coxeter element of type $A_{r-1}$ in $P_z$ .
		\end{itemize}
		In particular, $w_0$ is a proper parabolic quasi-Coxeter element in $W$.
\end{proposition}

\begin{proof}
	
	Since $xy$ and $z$ commute, we can apply Proposition~\ref{PropDiagTypeI} to $z$, and obtain the assertion for $z$,
	as well as Procedure~\ref{ProcedureTypesII,III} along with Proposition~\ref{PropReducedDecompII,III} to $xy$. The latter yields the decomposition
	$$xy = (u,v)(\overline{u},\overline{v}) (v, y_{q-1}) (y_{q-1}, y_{q-2}) \dotsc (y_2 ,y_1) (v, x_{p-1}) (x_{p-1}, x_{p-2}) \dotsc (x_2,x_1),$$ whose decomposition diagram is the Carter diagram $\Delta_{q-1,n}$
	with $p+q$ vertices. 
	In particular the decomposition diagram is connected, which yields that $P_{xy}$ is either of type $A_{p+q}$ or of type $D_{p+q}$.  By~\cite[Theorem A]{Carter}, $\Delta_{q-1,n}$ is not a Carter diagram in a group of type $A_{p+q}$ (as in the latter type Carter diagrams contain no cycles). Therefore $P_{xy}$ is of type $D_{p+q}$. All the parabolic subgroups of type $D_{n}$, $n \geq 4$,  are conjugate in $W$, and all the Coxeter elements in a finite Coxeter group are conjugate. As $xy$ has a different cycle-type than the elements appearing in Proposition~\ref{PropCoxDiagTypeII,III}, it is not a Coxeter element in $P_{xy}$. Thus, in this case $xy$ is a proper quasi-Coxeter element in $P_{xy}$.
		Therefore, $w_0$ is a proper parabolic quasi-Coxeter element in $W$.
\end{proof}

\begin{example}

Consider Equation \ref{Eqt11} for $m+1 \leq i < j \leq n$ and $n$ large enough. We have that $$w(\overline{i},\overline{n}) = (m,m-1, \dotsc, \overline{1})(i,i-1, \dotsc, m+1)(n,n-1, \dotsc, \overline{i+1}).$$ By Proposition \ref{ProcedureTypesII,III}, a reduced decomposition of $w(\overline{i},\overline{n})$ is
 $$w(\overline{i},\overline{n}) = (1,i+1)(\overline{1},\overline{i+1}) s_2 s_3 \dotsc s_m s_{i+2} \dotsc s_{n-1}s_n s_{m+2} \dotsc s_{i-1}s_i.$$

Its decomposition diagram is described in Figure~\ref{FigEqt11}.

\begin{figure}[H]
\begin{tikzpicture}

\node[draw, shape=circle, label=below:$s_{m}$] (2) at (0,0) {};
\node[draw, shape=circle, label=below:$s_{m-1}$] (3) at (1,0) {};
\node[] (0) at (2,0) {};
\node[] (00) at (3,0) {};
\node[draw, shape=circle] (4) at (4,0) {};
\node[draw, shape=circle, label=below:{\begin{small}$s_2$\end{small}}] (5) at (5,0) {};

\node[draw, shape=circle, label=above:{\begin{small}$(1,i+1)$\end{small}}] (6) at (6,1) {};
\node[draw, shape=circle, label=below:{\begin{small}$(\overline{1},\overline{i+1})$\end{small}}] (7) at (6,-1) {};
\node[draw, shape=circle,label=below:{\begin{small}$s_{i+2}$\end{small}}] (8) at (7,0) {};
\node[draw, shape=circle] (9) at (8,0) {};
\node[] (10) at (9,0) {};
\node[] (11) at (10,0) {};
\node[draw,shape=circle, label=below:$s_{n-1}$] (12) at (11,0) {};
\node[draw, shape=circle, label=below:{$s_{n}$}] (13) at (12.2,0) {};

\draw[thick,-] (2) to (3);
\draw[thick,-] (3) to (0);
\draw[thick,dashed,-] (0) to (00);
\draw[thick,-] (00) to (4);
\draw[thick,-] (4) to (5);
\draw[thick,-] (5) to (6);
\draw[thick,-] (5) to (7);
\draw[thick,-] (8) to (6);
\draw[thick,-] (8) to (7);
\draw[thick,-] (8) to (9);
\draw[thick,-] (9) to (10);
\draw[thick,dashed,-] (10) to (11);
\draw[thick,-] (11) to (12);
\draw[thick,-] (12) to (13);

\end{tikzpicture}

\begin{tikzpicture}

\node[draw, shape=circle, label=below:$s_{m+2}$] (1) at (0,0) {};
\node[draw, shape=circle] (2) at (1,0) {};
\node[] (3) at (2,0) {};
\node[] (4) at (3,0) {};
\node[draw, shape=circle, label=below:$s_{i-1}$] (5) at (4,0) {};
\node[draw, shape=circle, label=below:$s_i$] (6) at (5,0) {};

\draw[thick,-] (1) to (2);
\draw[thick,-] (2) to (3);
\draw[thick,dashed,-] (4) to (3);
\draw[thick,-] (4) to (5);
\draw[thick,-] (5) to (6);

\end{tikzpicture}
\caption{Decomposition diagram in the situation of Equation~\ref{Eqt11}.}
\label{FigEqt11}
\end{figure}
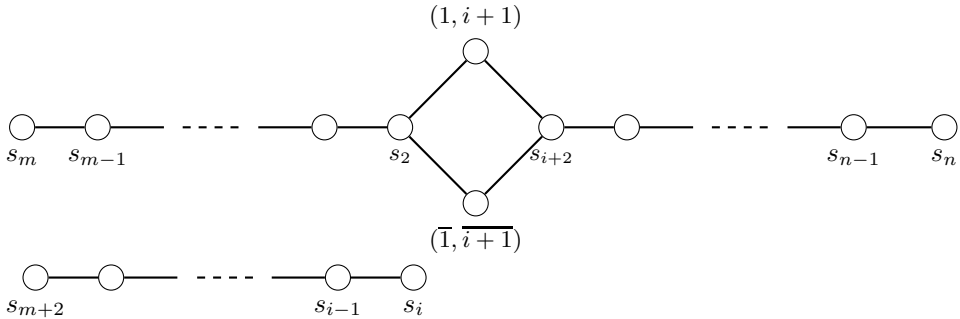

\end{example}

\section{Decomposition of the reflections and their lifts}\label{SecDecompReflections}

\subsection{Interval groups and the claimed presentation}\label{SubsectionConjectPres}

Let $w$ be a quasi-Coxeter element in $W$ of type $D_n$. Consider the interval $[1,w]$ of divisors of $w$ for the absolute order $\preceq$ given in Definition \ref{DefAbsoluteOrder} and the interval group $G([1,w])$ with its presentation given in Definition~\ref{DefIntervalGrpsDn}.

We denote by bold symbols the elements in $G([1,w])$. The copy $\bs{[1,w]}$ of the interval $[1,w]$ contains copies of the 
reflections $(i,j)$ and $(\overline{i},\overline{j})$ for $1 \leq i < j \leq n$, which we denote by $\bs{(i,j)}$ and $\bs{(\overline{i},\overline{j})}$.

By Proposition \ref{PropDualPresDn}, the group $G([1,w])$ is described by a presentation on the 
set \[
\RRefl = \{ \bs{(i,j)},\,\bs{(\overline{i},\overline{j})}: 1\leq i\neq j \leq n \}, \]
with relations the dual braid relations. These relations are described as $\bs{uv} = \bs{vu}$ if $uv \preceq w$ and $uv = vu$, and as $\bs{uv} = \bs{vt} = \bs{tu}$ ($\bs{t} \in \bs{\Refl}$) if $uv \preceq w$ and $uvu = vuv$ for $u, v \in \Refl$ and $u \neq v$.\\

It is convenient to reformulate our main result, Theorem~\ref{ThmPres}, as the following Theorem, which we shall prove in Section~\ref{SecProofPres}. Again we choose $1 \leq m \leq \lfloor \frac{n}{2}\rfloor$, and consider the quasi-Coxeter element $w=(m,m-1, \dotsc , \overline{1})(n,n-1, \dotsc, \overline{m+1})$. Let $\cS=\{s_1,\dotsc,s_n\}$ be the Carter generating set, where $s_{i+1} = (i,i+1)$ for $1 \leq i \leq n-1$ and $s_1 = (\overline{m},\overline{m+1})$. Let $\SS=\{\ss_1,\dotsc,\ss_n\} \subset \RRefl$ be the set in $G([1,w])$ in correspondence with $\cS$.

\begin{theorem}\label{ThmPresDn1}
The interval group $G([1,w])$ is isomorphic to the group $\GG$ defined by the presentation with generating set $\SS$ 
and relations described by the diagram $\Delta$ in Figure~\ref{FigDiagClaimedPres}
	together with the twisted cycle commutator relator $$\tc{\ss_1}{\ss_m}{\ss_{m+1}}{\ss_{m+2}} = [\ss_1,\ss_m^{-1}\ss_{m+1}\ss_{m+2}\ss_{m+1}^{-1}\ss_m] = [\ss_1, \ss_{m+2}^{\ss_{m+1}^{-1}\ss_m}],$$ 
	associated with the cycle $(\ss_1,\ss_m,\ss_{m+1},\ss_{m+2})$,
	that is, \[\GG = A(\Delta)/\langle \langle 
	\tc{\ss_1}{\ss_m}{\ss_{m+1}}{\ss_{m+2}} = [\ss_1,\ss_m^{-1}\ss_{m+1}\ss_{m+2}\ss_{m+1}^{-1}\ss_m] = [\ss_1, \ss_{m+2}^{\ss_{m+1}^{-1}\ss_m}]\rangle\rangle.\] 
	We will always describe this relator by a curved arrow inside the corresponding cycle; see Figure \ref{FigDiagClaimedPres}.
\end{theorem}

\begin{figure}[H]
	\begin{tikzpicture}
		
		\node[draw, shape=circle, label=below:$\ss_2$] (2) at (0,0) {};
		\node[draw, shape=circle, label=below:$\ss_3$] (3) at (1,0) {};
		\node[] (0) at (2,0) {};
		\node[] (00) at (3,0) {};
		\node[draw, shape=circle, label=below:$\ss_{m-1}$] (4) at (4,0) {};
		\node[draw, shape=circle, label=above:$\ss_m$] (5) at (5,0) {};
		
		\node[draw, shape=circle, label=above:$\ss_{m+1}$] (6) at (6,1) {};
		\node[draw, shape=circle, label=below:$\ss_1$] (7) at (6,-1) {};
		\node[draw, shape=circle,label=above:$\ss_{m+2}$] (8) at (7,0) {};
		\node[draw, shape=circle,label=below:$\ss_{m+3}$] (9) at (8,0) {};
		\node[] (10) at (9,0) {};
		\node[] (11) at (10,0) {};
		\node[draw,shape=circle, label=below:$\ss_{n-1}$] (12) at (11,0) {};
		\node[draw, shape=circle, label=below:$\ss_n$] (13) at (12.2,0) {};
		
		\node[] (curve) at (6,0) {\Huge $\circlearrowright$};
		
		\draw[thick,-] (2) to (3);
		\draw[thick,-] (3) to (0);
		\draw[thick,dashed,-] (0) to (00);
		\draw[thick,-] (00) to (4);
		\draw[thick,-] (4) to (5);
		\draw[thick,-] (5) to (6);
		\draw[thick,-] (5) to (7);
		\draw[thick,-] (8) to (6);
		\draw[thick,-] (8) to (7);
		\draw[thick,-] (8) to (9);
		\draw[thick,-] (9) to (10);
		\draw[thick,dashed,-] (10) to (11);
		\draw[thick,-] (11) to (12);
		\draw[thick,-] (12) to (13);
		
	\end{tikzpicture}\caption{The diagram presentation for the claimed presentation.}\label{FigDiagClaimedPres}
\end{figure}
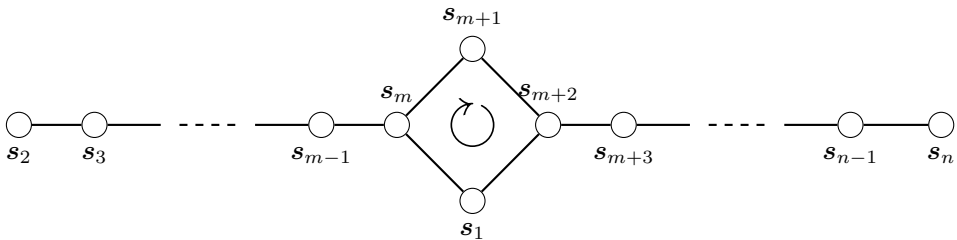

\begin{proposition}\label{PropPresCoxFromClaimedPres}

Adding the quadratic relations to the presentation of $\GG$, we obtain a group that is isomorphic to the Coxeter group $W$ of type $D_n$.

\end{proposition}

\begin{proof}

In fact, the twisted cycle commutator relator becomes the cycle commutator relator $[\ss_1, \ss_{m+2}^{\ss_{m+1}\ss_m}] = (\ss_1\ss_m\ss_{m+1}\ss_{m+2}\ss_{m+1}\ss_m)^2$ introduced in Proposition \ref{PropCameron}. The result follows immediately from the same proposition.
\end{proof}

The proof of the next lemma is easy and left as an exercise. 

\begin{lemma}

The twisted cycle commutator relator $[\ss_1, \ss_{m+2}^{\ss_{m+1}^{-1}\ss_m}] = 1$ can be written as a relation between positive words as follows:
$$\ss_m\ss_1\ss_{m+1}\ss_{m+2}\ss_m\ss_{m+1}\ss_1 = \ss_1\ss_{m+1}\ss_{m+2}\ss_m\ss_{m+1}\ss_1\ss_m,$$
meaning that $\ss_m$ commutes with $\ss_1\ss_{m+1}\ss_{m+2}\ss_m\ss_{m+1}\ss_1$.

\end{lemma}

%
%
%
%
%

\begin{remark}\label{RemOtherTwistedRel}
	Consider the cycle $(\ss_1,\ss_m,\ss_{m+1},\ss_{m+2})$ of the presentation of $\GG$. Consider the twisted cycle commutator relators $[\ss_1, \ss_{m+2}^{\ss_{m+1}^{-1}\ss_m}]$, $[\ss_{m},\ss_{1}^{\ss_{m+2}^{-1}}\ss_{m+1}]$, $[\ss_{m+1},\ss_{m+2}^{\ss_1^{-1}}\ss_{m}]$, and $[\ss_{m+2},\ss_1^{\ss_m^{-1}}\ss_{m+1}]$. It is an easy exercise to check that if one of the twisted cycle commutator relators holds, then the three other relators also hold.    
\end{remark}

\begin{remark}\label{RemNewProofDnBessis}

Suppose that $m=1$, i.e. $w$ is a Coxeter element. In this case, the group $\GG$ is the Artin group of type $D_n$. Our proof of Theorem \ref{ThmPresDn1} establishes a new proof of a result of Bessis showing that the interval group related to a Coxeter element in type $D_n$ is isomorphic to the related Artin group (see \cite{BessisDualMonoid}).

\end{remark}

We end the section by showing that the poset $([1,w],\preceq)$ of a proper quasi-Coxeter element $w$ in type $D_n$ is not a lattice. Hence the monoid defined by the same presentation as $G([1,w])$ viewed as a monoid presentation fails to be a Garside monoid. Note that this fact does not mean that the group $G([1,w])$ does not admit Garside structures.

\begin{proposition}\label{PropFailureLattice}
	Let $w$ be a proper quasi-Coxeter element in type $D_n$ for $n \geq 4$. The poset $([1,w],\preceq)$ is not a lattice.
\end{proposition}

\begin{proof}
	We check using \verb|GAP| that for $w$ a proper quasi-Coxeter element in type $D_4$, there exists a bowtie in $([1,w],\preceq)$, hence it is not a lattice, see Proposition~1.10 in \cite{McCSul}. (The bowtie consists of two reflections $t_1,t_2$ that commute in $W$ such that $t_1 t_2 \not\in [1,w]$.)
	
	Let $w$ be a proper quasi-Coxeter element in type $D_n$ for $n > 4$. Then it contains a subword $w'$ of $w$ whose Carter diagram is a $4$-cycle. According to Theorem~2.1 in \cite{Dyer}, all elements below $w'$ in $([1,w],\preceq)$ are in $([1,w'],\preceq)$ read as a poset in the rank $4$ parabolic subgroup $P_{w'}$. Therefore there is still the bowtie coming from the case $n=4$ inside $([1,w],\preceq)$ for $n > 4$.
\end{proof}

\subsection{Decomposition of the reflections on Carter generators}\label{SubsectionDecompRefCarterGen}

The purpose of this section is to find decompositions of the elements of 
$\Refl$ in terms of Carter generators. This corresponds to Step 3 in the strategy of our proof as explained in Section~\ref{SubsectionStrategyProof}.

We recall that we fix an integer $m$ with $1 \leq m \leq \lfloor n/2 \rfloor$. 
We recall from Proposition~\ref{PropCameron} that a presentation of the 
Coxeter group $W$ is defined on Carter generators $s_1,s_2, \dotsc, s_n$ 
together with the relations described in the diagram presentation illustrated in Figure~\ref{FigureCarterDiagramDn} together with the quadratic relations.

Recall that the reflections $s_2$, $s_3$, $\dotsc$, $s_n$ are the transpositions $(1,2)$, $(2,3)$, $\dotsc$, $(n-1,n)$, respectively, while the reflection $s_1$ is the marked permutation $(\overline{m},\overline{m+1})$. In the next proposition, we decompose
each reflection $(i,j)$ and $(\overline{i},\overline{j})$ over 
the Carter generators $s_1$, $s_2$, $\dotsc$, $s_n$.

\begin{proposition}\label{PropDecReflections}
	
\begin{itemize}
\item[(1)] Let $\refl = (i,j)$ with $1 \leq i < j \leq n$. We have 
\begin{equation}\label{EqtDecomTransp}
	\refl = s_j^{s_{j-1} s_{j-2} \dotsc s_{i+1}}. 
\end{equation} 
\item[(2)] Let $\refl = (\overline{i},\overline{j})$ with $1 \leq i \leq m$ and $m+1 \leq j \leq n$. We have 
\begin{equation}\label{EqtDecomI}
	(\overline{i},\overline{j}) = s_1^{s_{m+2}s_{m+3} \dotsc s_j s_{m}s_{m-1} \dotsc s_{i+1}}.
\end{equation}
\item[(3)] Let $\refl = (\overline{i},\overline{j})$ with $1 \leq i < j \leq m$. We have 
\begin{equation}\label{EqtDecomII}
	(\overline{i},\overline{j}) = s_1^{s_m  s_{m-1} \dotsc s_{i+1}  s_{m+1}  s_m  \dotsc  s_{j+1}}.
\end{equation} 
\item[(4)] Let $\refl = (\overline{i},\overline{j})$ with $m+1 \leq i < j \leq n$. We have 
\begin{equation}\label{EqtDecomIII}
	(\overline{i},\overline{j}) = s_1^{s_{m+2} s_{m+3} \dotsc s_j s_{m+1}s_{m+2} \dotsc s_i}.
\end{equation} 
\end{itemize}

\end{proposition}

\begin{proof}
	The equations are easily obtained by direct calculation using the definition of the reflections as marked permutations.
\end{proof}

\begin{remark}
	
\begin{itemize}
	\item[(1)] In Equation~\ref{EqtDecomTransp}, if $j = i+1$, we get $\refl = (i, i+1) = s_{i+1}$ ($1 \leq i \leq n-1$) visible among Carter generators.
	\item[(2)] In Equation~\ref{EqtDecomI}, if $i=m$ and $j=m+1$, we have $\refl = s_1$ visible among Carter generators.
	\item[(3)] The decompositions of $(i,j)$ and $(\overline{i},\overline{j})$ obtained in Equations~\ref{EqtDecomTransp}-- \ref{EqtDecomIII} are reduced decompositions over Carter generators. This result is not straightforward, but can be established using techniques from \cite{NeaimeIntervals}. Since this fact is not used in the proof of our main result, we will not include its proof in this paper.
\end{itemize}

\end{remark}

\subsection{Lifting the reflections}\label{SubLiftingReflections}

The purpose of this section is to write each element $\bs{(i,j)}$ and 
$\bs{(\overline{i},\overline{j})}$ ($1 \leq i < j \leq n$) of $G([1,w])$ 
in terms of the generators $\ss_1$, $\ss_2$, $\dotsc$, $\ss_n$ that appear 
in the presentation of Theorem \ref{ThmPresDn1}. This corresponds to Step 3 in the strategy of our proof as explained in Section~\ref{SubsectionStrategyProof}.
Recall that $\bs{(i,j)}$ and $\bs{(\overline{i},\overline{j})}$ are the 
copies of the reflections $(i,j)$ and $(\overline{i},\overline{j})$ within $\RRefl$, and 
$\ss_1$, $\ss_2$, $\dotsc$, $\ss_n$ are the copies of the reflections 
$s_1 = (\overline{m},\overline{m+1})$, $s_2$, $s_3$, $\dotsc$, $s_n$. 
We employ the decompositions of the reflections in term of $s_1,s_2, \dotsc, s_n$ that we described in Equations \ref{EqtDecomTransp} to \ref{EqtDecomIII} 
of Section \ref{SubsectionDecompRefCarterGen}. These decompositions 
serve as a guide. In fact, we walk through the reflections in the exponent expressions of these equations and derive our result. 
The explicitness of these equations enables us to describe in a simple way the main result of this section, Proposition \ref{PropLiftReflections}. 

As a preamble, let us illustrate our ideas in the following example.

\begin{example}\label{ExampleLiftTransp}

Let $W$ be the Coxeter group of type $D_4$. Let $m=2$ and let $w = (2,\overline{1})(4,\overline{3})$ be a proper quasi-Coxeter element of length $4$ by Proposition \ref{PropLengthShi}. Consider the reflection $(1,4) \in \Refl$. It is equal to $s_4^{s_3s_2}$ by Equation \ref{EqtDecomTransp}. We decompose the copy
$\bs{(1,4)} \in \bs{\Refl}$ in terms of the generators $\ss_1$, $\ss_2$, $\ss_3$, $\dotsc$, $\ss_n$.
\begin{itemize}
\item Since $w' = s_2(1,4)w = (1,\overline{3},\overline{4})(2)$ is of length $2$ by Proposition \ref{PropLengthShi}, then we have $(1,4)s_2 \preceq w$. So we get $(1,4)s_2 = s_2 {(1,4)}^{s_2} = s_2 s_4^{s_3} = s_2(2,4) \preceq w$. Hence this gives $\bs{(1,4)} = \bs{s_2 (2,4) s_2^{-1}}$.
\item Similarly, for $s_4^{s_3} = (2,4)$, we have that $(2,4)s_3 = s_3(3,4) \preceq w$ also by a direct application of Proposition \ref{PropLengthShi}. Hence we get $\bs{(2,4)} = \ss_3 \bs{(3,4)} \ss_3^{-1} = \ss_3 \ss_4 \ss_3^{-1}$.
\end{itemize}   
It follows that $\bs{(1,4)} = \ss_2\bs{(2,4)}\ss_2^{-1} = \ss_2\ss_3\ss_4\ss_3^{-1}\ss_2^{-1} = \ss_4^{\ss_3^{-1}\ss_2^{-1}}$.

\end{example}

\begin{proposition}\label{PropLiftReflections}

The copies of the reflections to the interval group $G([1,w])$ decompose on the generators $\ss_1$, $\ss_2$, $\dotsc$, $\ss_n$ as follows.
\begin{equation}\label{EqtLiftDecomTransp}
\bs{(i,j)} = \ss_j^{\ss_{j-1}^{-1} \ss_{j-2}^{-1} \dotsc \ss_{i+1}^{-1}}, \hbox{ for }1 \leq i < j \leq n,
\end{equation}
\begin{equation}\label{EqtLiftDecomI}
\bs{(\overline{i},\overline{j})} = \ss_1^{\ss_{m+2} \ss_{m+3} \dotsc \ss_j \ss_{m}^{-1} \ss_{m-1}^{-1} \dotsc \ss_{i+1}^{-1}},\linebreak \hbox{ for } 1 \leq i \leq m, m+1 \leq j \leq n.
\end{equation}
\begin{equation}\label{EqtLiftDecomII}
\bs{(\overline{i},\overline{j})} = \ss_1^{\ss_m^{-1}\ss_{m-1}^{-1} \dotsc \ss_{i+1}^{-1} \ss_{m+1} \ss_m^{-1} \ss_{m-1}^{-1} \dotsc \ss_{j+1}^{-1}}, \hbox{ for } 1 \leq i < j \leq m,
\end{equation}
\begin{equation}\label{EqtLiftDecomIII}
\bs{(\overline{i},\overline{j})} = \ss_1^{\ss_{m+2} \ss_{m+3} \dotsc \ss_j \ss_{m+1}^{-1} \ss_{m+2} \dotsc \ss_i}, \hbox{ for } m+1 \leq i < j \leq n,
\end{equation}

\end{proposition}

\begin{proof}

We employ the decompositions of the reflections $(i,j)$ and $(\underline{i},\underline{j})$ for $1 \leq i < j \leq n$ in term of Carter generators $s_1,s_2, \dotsc, s_n$ that we described in Equations \ref{EqtDecomTransp} to \ref{EqtDecomIII} in Section \ref{SubsectionDecompRefCarterGen}. Each of these equations is of the form $$\refl = y^{x_1x_2 \dotsc x_p},$$ for $p \geq 1$. For $k$ from $p$ down to $1$, we proceed as follows. Let $\refl_k = y^{x_1x_2 \dotsc x_k}$. We have that $\refl_p = \refl$.
\begin{itemize}
\item If $\refl_kx_k \preceq w$, then we have $\refl_kx_k = x_k \refl_k^{x_k} = x_k \refl_{k-1} \preceq w$. It follows that $\bs{\refl_kx_k} = \bs{x_k\refl_{k-1}}$, which implies that $\bs{\refl_k} = \bs{x_k\refl_{k-1}x_k^{-1}}$.
\item If $x_k\refl_k \preceq w$, then we have $x_k\refl_k = \refl_k^{x_k}x_k = \refl_{k-1}x_k \preceq w$. It follows that $\bs{x_k\refl_k} = \bs{\refl_{k-1}x_k}$, which gives $\bs{\refl_k} = \bs{x_k^{-1}\refl_{k-1}x_k}$.
\end{itemize}
It turns out that for all $k$ from $p$ down to $1$, we are in one of the previous two situations in all Equations \ref{EqtDecomTransp} to \ref{EqtDecomIII}. It follows that $\refl = \bs{y^{x_1^{\epsilon_1}x_2^{\epsilon_2}\dotsc x_p^{\epsilon_p}}}$, where $\bs{\epsilon_k} = -1$ or $1$ ($1 \leq k \leq p$), depending whether we apply the first or the second situation, respectively.

Let us explain this for the copy $\bs{(\overline{i},\overline{j})}$ of $(\overline{i},\overline{j})$ for $1 \leq i < j \leq m$. The same argument applies for the other equations.  For $1 \leq i < j \leq m$, by Equation \ref{EqtDecomII}, we have that $$(\overline{i},\overline{j}) = s_1^{s_m  s_{m-1} \dotsc s_{i+1}  s_{m+1}  s_m  \dotsc  s_{j+1}}.$$

Applying Proposition~\ref{PropLengthShi}, we have that $(\overline{i},\overline{j})s_{j+1} \preceq w$, so $(\overline{i},\overline{j}) s_{j+1} = s_{j+1}{(\overline{i},\overline{j})}^{s_{j+1}} = s_{j+1} (\overline{i},\overline{j+1}) \preceq w$. Then we get $$\bs{(\overline{i},\overline{j})} = \ss_{j+1} \bs{(\overline{i},\overline{j+1})} \ss_{j+1}^{-1}.$$

Next, applying Proposition~\ref{PropLengthShi}, we have $(\overline{i},\overline{j+1}) s_{j+2} \preceq w$, meaning that $(\overline{i},\overline{j+1}) s_{j+2} = s_{j+2} {(\overline{i},\overline{j+1})}^{s_{j+2}} = s_{j+2} (\overline{i},\overline{j+2}) \preceq w$. So, we get $$\bs{(\overline{i},\overline{j+1})} = \ss_{j+2} \bs{(\overline{i},\overline{j+2})} \ss_{j+2}^{-1}.$$ Hence we have $$\bs{(\overline{i},\overline{j})} = \bs{(\overline{i},\overline{j+2})}^{\ss_{j+2}^{-1}\ss_{j+1}^{-1}}.$$
And so on, we apply the same computations for $s_{j+2}$, $s_{j+3}$, $\dotsc$, $s_m$ appearing in the exponent part of Equation \ref{EqtDecomII}, and get $$\bs{(\overline{i},\overline{j})} = \bs{(\overline{i},\overline{m})}^{\ss_{m}^{-1} \dotsc \ss_{j+2}^{-1} \ss_{j+1}^{-1}}.$$

Next, applying Proposition~\ref{PropLengthShi}, we have $s_{m+1}(\overline{i},\overline{m}) \preceq w$ (and not $(\overline{i},\overline{m})s_{m+1} \preceq w$), meaning that $s_{m+1} (\overline{i},\overline{m}) = (\overline{i},\overline{m+1}) s_{m+1} \preceq w$. Hence we get $$\bs{ (\overline{i},\overline{m})} = \ss_{m+1}^{-1} \bs{(\overline{i},\overline{m+1})} \ss_{m+1}.$$ Then we obtain $$\bs{(\overline{i},\overline{j})} = \bs{(\overline{i},\overline{m+1})}^{\ss_{m+1} \ss_m^{-1} \dotsc \ss_{j+2}^{-1} \ss_{j+1}^{-1}}.$$

Similarly, we have $(\overline{i},\overline{m+1}) s_{i+1} \preceq w$, so $(\overline{i},\overline{m+1}) s_{i+1} = s_{i+1} (\overline{i+1},\overline{m+1}) \preceq w$. Thus, we get $$\bs{(\overline{i},\overline{m+1})} = \ss_{i+1} \bs{(\overline{i+1},\overline{m+1})} \ss_{i+1}^{-1}.$$ And so on, we apply the same calculation for $s_{i+1}$, $s_{i+2}$, $\dotsc$, $s_{m}$ until we obtain the desired equation:
\begin{equation*}
\hbox{For }1 \leq i < j \leq m, \bs{(\overline{i},\overline{j})} = \ss_1^{\ss_m^{-1} \ss_{m-1}^{-1} \dotsc \ss_{i+1}^{-1} \ss_{m+1} \ss_m^{-1} \ss_{m-1}^{-1} \dotsc \ss_{j+1}^{-1}}.
\end{equation*}

\end{proof}

We provide an example where the decomposition that we obtain will appear in the twisted cycle commutator relators in the next section.

\begin{example}\label{ExampleLiftII}

Let $W$ be a Coxeter group of type $D_4$. Let $m=2$ and $w = (2,\overline{1})(4,\overline{3})$ be a proper quasi-Coxeter element. Consider the reflection $\refl = (\overline{1},\overline{2})$ of type II. By Equation \ref{EqtDecomII}, we have that $(\overline{1},\overline{2}) = s_1^{s_2s_3}$.
\begin{itemize}
\item Since $w' = (\overline{1},\overline{2})s_3w = (2)(4,3,1)$ is of length $2$ by Proposition \ref{PropLengthShi}, then we have that $s_3 (\overline{1},\overline{2}) \preceq w$. Note that we are in the situation of the second bullet in the proof of Proposition \ref{PropLiftReflections}, meaning that $(\overline{1},\overline{2})s_3$ does not divide $w$. Hence we get $s_3 (\overline{1},\overline{2}) = (\overline{1},\overline{2})^{s_3}s_3 = s_1^{s_2} s_3 = (\overline{1},\overline{3}) s_3 \preceq w$. Therefore, we obtain $\ss_3\bs{(\overline{1},\overline{2})} = \bs{(\overline{1},\overline{3})} \ss_3$, which gives $\bs{(\overline{1},\overline{2})} = \ss_3^{-1}\bs{(\overline{1},\overline{3})}\ss_3$.
\item Next, we consider $(\overline{1},\overline{3}) = s_1^{s_2}$. We have that $w' = s_2(\overline{1},\overline{3})w = (1)(4,3,2)$ is of length $2$. Thus, we have $(\overline{1},\overline{3})s_2 \preceq w$, which says that $(\overline{1},\overline{3})s_2 = s_2 {(\overline{1},\overline{3})}^{s_2} = s_2 s_1 \preceq w$. We obtain $\bs{(\overline{1},\overline{3})s_2} = \bs{s_2s_1}$, which is $\bs{(\overline{1},\overline{3})} = \bs{s_2s_1s_2^{-1}}$.
\end{itemize}
Therefore, we obtain $$\bs{(\overline{1},\overline{2})} = \bs{s_3^{-1}(\overline{1},\overline{3})s_3} = \bs{s_3^{-1}s_2s_1s_2^{-1}s_3} = \bs{s_1^{s_2^{-1}s_3}}.$$

\end{example}

We finish this section by the next lemma which is used to show that $f$ is a homomorphism (see Step~1 of the strategy of the proof in Section~\ref{SubsectionStrategyProof}). The proof of the lemma uses Proposition~\ref{PropLiftReflections}.

\begin{lemma}\label{LemmaTC}
The braid relators $\braid{\ss_i}{\ss_j}$ and the twisted cycle commutator relator $\tc{\ss_{m+2}}{\ss_{m+1}}{\ss_m}{\ss_{1}}$ specified by the presentation given for $\GG$ hold in $G([1,w])$.
\end{lemma}

\begin{proof}
Consider case (1) of Proposition~\ref{PropPreparationHomoF}. It implies a commuting braid relation, which lifts to $\bs{s_is_j = s_js_i}$ for $|i-j|>1$.

Consider case (2) of Proposition~\ref{PropPreparationHomoF}. It implies a dual braid relation $\bs{s_is_{i+1}} = \bs{s_{i+1}t}$, for $2 \leq i \leq n-1$, where $\bs{t} = \bs{(i-1,i+1)}$. Applying Equation~(16) of Proposition~\ref{PropLiftReflections}, we have that $\bs{t}$ is equal to $\bs{s_{i+1}^{s_i^{-1}}}$. The dual braid relation becomes $\bs{s_is_{i+1}} = \bs{s_{i+1}s_{i+1}^{s_i^{-1}}}$, that is $\bs{s_is_{i+1}s_i} = \bs{s_{i+1}s_is_{i+1}}$.

Case (3) of Proposition~\ref{PropPreparationHomoF} is treated similarly. We will give details on case (4) where the TC relator will appear.

Consider then Case (4) of Proposition~\ref{PropPreparationHomoF}. It implies that $ts_{m+2} \preceq w$ with $ts_{m+2}$ of order $2$. So we get a commuting dual braid relation. Now, we have to lift the relation to prove that it lives in $G([1,w])$. Applying Equation~(18) of Proposition~\ref{PropLiftReflections} for $i=m-1$ and $j=m$, we have that $\bs{t}$ is equal to $\bs{s_1^{s_m^{-1}s_{m+1}}}$. Then the dual braid relation becomes $[\bs{s_{m+2}},\bs{s_1^{s_m^{-1}s_{m+1}}}] = 1$, which is exactly the TC relator (see Remark~\ref{RemOtherTwistedRel}).
\end{proof}

\section{The proof of the main theorem}
\label{SecProofPres}
\subsection{The case $n=4$}\label{SubsectionN4}

Let $W$ be the Coxeter group of type $D_4$. By Proposition~\ref{PropCameron}, $W$ has a presentation on the four generators $s_2 = (1,2)$, $s_3 = (2,3)$, $s_4 = (3,4)$, and $s_1 = (\overline{2},\overline{3})$. We have $m=2$ and the corresponding proper quasi-Coxeter element is $w = (2,\overline{1})(4,\overline{3})$.

We prove that the interval group $G([1,w])$  is isomorphic to the group $\GG$ with four generators $\ss_1, \ss_2, \ss_3, \ss_4$,
corresponding to reflections $s_1,s_2,s_3,s_4$, with relations described by the corresponding Carter diagram, along with the twisted cycle commutator relator $$\tc{\ss_1}{\ss_2}{\ss_3}{\ss_4} = [\ss_1,\ss_4^{\ss_3^{-1}\ss_2}].$$

We consider a reflection $\refl$ in $\Refl$ to be of type I, II, or III according to the
type of the maximal divisor $w_0=w\refl$ of $w$, and assign the same types
to the elements of $\RRefl$. We collect the information we need in Tables \ref{TableExprrefl} and \ref{TableDiagLiftn4}.

Table \ref{TableExprrefl} provides decompositions 
for the elements of $\RRefl$ by applying Proposition~\ref{PropLiftReflections}.

The second column of Table \ref{TableDiagLiftn4} contains the $12$ divisors (the $w_0$'s) of length 
$3$ of the quasi-Coxeter element $w$ of types I, II, and III, where we separate each type by two lines. We follow 
Section~\ref{SubsectionDivisorsln-1} in order to produce them. 
We also follow Sections \ref{SubsectionRedDecompI} and 
\ref{SubsectionRedDecompII} to produce reduced decompositions 
of these elements and their decomposition diagrams (the $\Delta_0$'s). The last column produces a Coxeter-like diagram related to $\Delta_0$ that we call the lift of $\Delta_0$. It encodes the relations between the lift to the interval group of two reflections that appear in the reduced decomposition of each $w_0$.

\begin{table}[]
	\begin{center}
		\begin{tabular}{|c|c|}
			\hline
			Number & Decomposition of $\rrefl$ (type I)\\
			\hline
			1. & $\bs{(1,3)} = \bs{s_3^{s_2^{-1}}}$\\
			\hline
			2. & $\bs{(1,4)} = \bs{s_4^{s_3^{-1}s_2^{-1}}}$\\
			\hline
			3. & $\bs{(2,3)} = \ss_3$\\
			\hline
			4. & $\bs{(2,4)} = \bs{s_4^{s_3^{-1}}}$\\
			\hline
			5. & $\bs{(\overline{1},\overline{3})} = \bs{s_1^{s_2^{-1}}}$\\
			\hline
			6. & $\bs{(\overline{1},\overline{4})} = \bs{s_1^{s_4s_2^{-1}}}$\\
			\hline
			7. & $\bs{(\overline{2},\overline{3})} = \ss_1$\\
			\hline
			8. & $\bs{(\overline{2},\overline{4})} = \bs{s_3^{s_4}}$\\
			\hline
			& Decomposition of $\rrefl$ (type II)\\
			\hline
			9. &  $\bs{(1,2)} = \ss_2$\\
			\hline
			10. & $\bs{(\overline{1},\overline{2})} = \bs{s_1^{s_2^{-1}s_3}}$\\
			\hline
			& Decomposition of $\rrefl$ (type III)\\
			\hline
			11. &  $\bs{(3,4)} = \ss_4$\\
			\hline
			12. & $\bs{(\overline{3},\overline{4})} = \bs{s_1^{s_4s_3^{-1}}}$\\
			\hline
		\end{tabular}
	\end{center}
	\caption{Decompositions of $\rrefl$ in the case $n=4$.}
	\label{TableExprrefl}
\end{table}

\begin{table}[h]
\begin{center}
\begin{tabular}{|c|c|c|c|}
\hline
Number & Maximal divisor $w_0$ & Reduced Decomposition & Lift of $\Delta_0$\\
\hline \hline
\multirow{2}*{1.} &  \multirow{2}*{$w(1,3) = (4,\overline{1},2,\overline{3})$} &  \multirow{2}*{$(2,3)(\overline{1},\overline{2})(1,4)$} &  \multirow{2}*{ \begin{tikzpicture}
\node[draw, shape=circle, label=below:{$\bs{(1,4)}$}] (1) at (2,0) {};
\node[draw, shape=circle, label=below:{$\bs{(\overline{1},\overline{2})}$}] (2) at (1,0) {};
\node[draw, shape=circle, label=below:{$\bs{(2,3)}$}] (3) at (0,0) {};
\draw[thick,-] (1) to (2);
\draw[thick,-] (2) to (3);
\end{tikzpicture}}\\

& & & \\
\hline

\multirow{2}*{2.} &  \multirow{2}*{$w(1,4) = (4,\overline{3},\overline{1},2)$} &  \multirow{2}*{$(1,2)(\overline{1},\overline{3})(3,4)$} &  \multirow{2}*{ \begin{tikzpicture}
\node[draw, shape=circle, label=below:{$\bs{(3,4)}$}] (1) at (2,0) {};
\node[draw, shape=circle, label=below:{$\bs{(\overline{1},\overline{3})}$}] (2) at (1,0) {};
\node[draw, shape=circle, label=below:{$\bs{(1,2)}$}] (3) at (0,0) {};
\draw[thick,-] (1) to (2);
\draw[thick,-] (2) to (3);
\end{tikzpicture}}\\

& & & \\
\hline

\multirow{2}*{3.} &  \multirow{2}*{$w(2,3) = (4,2,\overline{1},\overline{3})$} &  \multirow{2}*{$(\overline{1},\overline{3})(1,2)(2,4)$} &  \multirow{2}*{ \begin{tikzpicture}
\node[draw, shape=circle, label=below:{$\bs{(\overline{1},\overline{3})}$}] (1) at (0,0) {};
\node[draw, shape=circle, label=below:{$\bs{(1,2)}$}] (2) at (1,0) {};
\node[draw, shape=circle, label=below:{$\bs{(2,4)}$}] (3) at (2,0) {};
\draw[thick,-] (1) to (2);
\draw[thick,-] (2) to (3);
\end{tikzpicture}}\\

& & & \\
\hline

\multirow{2}*{4.} &  \multirow{2}*{$w(2,4) = (4,\overline{3},2,\overline{1})$} &  \multirow{2}*{$(1,2)(\overline{2},\overline{3})(3,4)$} &  \multirow{2}*{ \begin{tikzpicture}
\node[draw, shape=circle, label=below:{$\bs{(1,2)}$}] (1) at (0,0) {};
\node[draw, shape=circle, label=below:{$\bs{(\overline{2},\overline{3})}$}] (2) at (1,0) {};
\node[draw, shape=circle, label=below:{$\bs{(3,4)}$}] (3) at (2,0) {};
\draw[thick,-] (1) to (2);
\draw[thick,-] (2) to (3);
\end{tikzpicture}}\\

& & & \\
\hline

\multirow{2}*{5.} &  \multirow{2}*{$w(\overline{1},\overline{3}) = (\overline{4},\overline{1},\overline{2},\overline{3})$} &  \multirow{2}*{$(\overline{2},\overline{3})(\overline{1},\overline{2})(\overline{1},\overline{4})$} &  \multirow{2}*{ \begin{tikzpicture}
\node[draw, shape=circle, label=below:{$\bs{(\overline{2},\overline{3})}$}] (1) at (0,0) {};
\node[draw, shape=circle, label=below:{$\bs{(\overline{1},\overline{2})}$}] (2) at (1,0) {};
\node[draw, shape=circle, label=below:{$\bs{(\overline{1},\overline{4})}$}] (3) at (2,0) {};
\draw[thick,-] (1) to (2);
\draw[thick,-] (2) to (3);
\end{tikzpicture}}\\

& & & \\
\hline

\multirow{2}*{6.} &  \multirow{2}*{$w(\overline{1},\overline{4}) = (4,3,\overline{1},\overline{2})$} &  \multirow{2}*{$(\overline{1},\overline{2})(1,3)(3,4)$} &  \multirow{2}*{ \begin{tikzpicture}
\node[draw, shape=circle, label=below:{$\bs{(\overline{1},\overline{2})}$}] (1) at (0,0) {};
\node[draw, shape=circle, label=below:{$\bs{(1,3)}$}] (2) at (1,0) {};
\node[draw, shape=circle, label=below:{$\bs{(3,4)}$}] (3) at (2,0) {};
\draw[thick,-] (1) to (2);
\draw[thick,-] (2) to (3);
\end{tikzpicture}}\\

& & & \\
\hline

\multirow{2}*{7.} &  \multirow{2}*{$w(\overline{2},\overline{3}) = (\overline{4},2,1,\overline{3})$} &  \multirow{2}*{$(1,3)(1,2)(\overline{2},\overline{4})$} &  \multirow{2}*{ \begin{tikzpicture}
\node[draw, shape=circle, label=below:{$\bs{(1,3)}$}] (1) at (0,0) {};
\node[draw, shape=circle, label=below:{$\bs{(1,2)}$}] (2) at (1,0) {};
\node[draw, shape=circle, label=below:{$\bs{(\overline{2},\overline{4})}$}] (3) at (2,0) {};
\draw[thick,-] (1) to (2);
\draw[thick,-] (2) to (3);
\end{tikzpicture}}\\

& & & \\
\hline

\multirow{2}*{8.} &  \multirow{2}*{$w(\overline{2},\overline{4}) = (4,3,2,1)$} &  \multirow{2}*{$(1,2)(2,3)(3,4)$} &  \multirow{2}*{ \begin{tikzpicture}
\node[draw, shape=circle, label=below:{$\bs{(1,2)}$}] (1) at (0,0) {};
\node[draw, shape=circle, label=below:{$\bs{(2,3)}$}] (2) at (1,0) {};
\node[draw, shape=circle, label=below:{$\bs{(3,4)}$}] (3) at (2,0) {};
\draw[thick,-] (1) to (2);
\draw[thick,-] (2) to (3);
\end{tikzpicture}}\\

& & & \\
\hline \hline

\multirow{2}*{9.} &  \multirow{2}*{$w(1,2) = (\overline{1})(2) (4,\overline{3})$} &  \multirow{2}*{$(1,3)(\overline{1},\overline{3})(3,4)$} &  \multirow{2}*{ \begin{tikzpicture}
\node[draw, shape=circle, label=below:{$\bs{(\overline{1},\overline{3})}$}] (1) at (0,0) {};
\node[draw, shape=circle, label=below:{$\bs{(3,4)}$}] (2) at (1,0) {};
\node[draw, shape=circle, label=below:{$\bs{(1,3)}$}] (3) at (2,0) {};
\draw[thick,-] (1) to (2);
\draw[thick,-] (2) to (3);
\end{tikzpicture}}\\

& & & \\
\hline

\multirow{2}*{10.} &  \multirow{2}*{$w(\overline{1},\overline{2}) = (1)(\overline{2})(4,\overline{3})$} &  \multirow{2}*{$(2,3)(\overline{2},\overline{3})(3,4)$} &  \multirow{2}*{ \begin{tikzpicture}
\node[draw, shape=circle, label=below:{$\bs{(\overline{2},\overline{3})}$}] (1) at (0,0) {};
\node[draw, shape=circle, label=below:{$\bs{(3,4)}$}] (2) at (1,0) {};
\node[draw, shape=circle, label=below:{$\bs{(2,3)}$}] (3) at (2,0) {};
\draw[thick,-] (1) to (2);
\draw[thick,-] (2) to (3);
\end{tikzpicture}}\\

& & & \\
\hline \hline

\multirow{2}*{11.} &  \multirow{2}*{$w(3,4) = (2,\overline{1})(\overline{3})(4)$} &  \multirow{2}*{$(1,3)(\overline{1},\overline{3})(1,2)$} &  \multirow{2}*{ \begin{tikzpicture}
\node[draw, shape=circle, label=below:{$\bs{(\overline{1},\overline{3})}$}] (1) at (0,0) {};
\node[draw, shape=circle, label=below:{$\bs{(1,2)}$}] (2) at (1,0) {};
\node[draw, shape=circle, label=below:{$\bs{(1,3)}$}] (3) at (2,0) {};
\draw[thick,-] (1) to (2);
\draw[thick,-] (2) to (3);
\end{tikzpicture}}\\

& & & \\
\hline

\multirow{2}*{12.} &  \multirow{2}*{$w(\overline{3},\overline{4}) = (2,\overline{1})(3)(\overline{4})$} &  \multirow{2}*{$(1,4)(\overline{1},\overline{4})(1,2)$} &  \multirow{2}*{ \begin{tikzpicture}
\node[draw, shape=circle, label=below:{$\bs{(\overline{1},\overline{4})}$}] (1) at (0,0) {};
\node[draw, shape=circle, label=below:{$\bs{(1,2)}$}] (2) at (1,0) {};
\node[draw, shape=circle, label=below:{$\bs{(1,4)}$}] (3) at (2,0) {};
\draw[thick,-] (1) to (2);
\draw[thick,-] (2) to (3);
\end{tikzpicture}}\\

& & & \\
\hline

\end{tabular}
\end{center}
\caption{Reduced decompositions and diagram lifts.}
\label{TableDiagLiftn4}
\end{table}

\begin{proposition} \label{PropCheckAllRelationsN4}

All the relations that describe the type $A_3$ diagrams on the last column of Table \ref{TableDiagLiftn4} are consequences of the relations described by the diagram presentation over $\ss_1, \ss_2, \ss_3, \ss_4$ illustrated in Figure~\ref{FigDiagClaimedPres}.

\end{proposition}

We prove the proposition by showing using \verb|kbmag| \cite{KBMAG} within \verb|GAP| \cite{GAP} that the relations appearing in the last column of Table \ref{TableDiagLiftn4} are consequences of the relations between $\ss_1, \ss_2, \ss_3, \ss_4$. For example, let us consider a diagram where some twisted cycle commutator relators appear. Consider the element Number 6 of the table. We have to show that $\bs{(\overline{1},\overline{2})}$ commutes with $\ss_4$, where $\bs{(\overline{1},\overline{2})} = \bs{s_1^{s_2^{-1}s_3}}$. The commuting relation between $\bs{(\overline{1},\overline{2})}$ and $\ss_4$ is precisely the twisted cycle commutator relator $[\ss_4,\bs{s_1^{s_2^{-1}s_3}}]$ that is a consequence of the relations of the claimed presentation.

Now we can show that $G([1,w])$ is isomorphic to $\GG$, that is Theorem~\ref{ThmPresDn1} in the case where $n=4$.

\begin{proposition}\label{PropPresD4}

In the case $n=4$, the groups $\GG$ and $G([1,w])$ are isomorphic.

\end{proposition}

\begin{proof}
By transitivity of the Hurwitz action on the reduced decompositions over 
$\RRefl$ of $w$, the group $G([1,w])$ is generated by a copy 
\[ \RRefl = \{ \bs{(i,j)},\,\bs{(\overline{i},\overline{j})}:  
1 \leq i < j \leq 4\} \]
of the set of reflections in $\Refl$, 
and subject to the dual braid relations
$\rrefl\rrefl' = \rrefl'\rrefl''$ that correspond to relations
$\refl\refl' = \refl'\refl''$ in $W$ where
$\refl\refl' = \refl'\refl'' \preceq w$.

Consider the map $f: \GG \longrightarrow G([1,w]): \ss_i \longmapsto \ss_i$. By Lemma~\ref{LemmaTC}, the relations of the presentation of $\GG$ hold in $G([1,w])$.

Now consider the map $g: G([1,w]) \longrightarrow \GG$ that maps each 
generator $\rrefl$ of $G([1,w])$ to the expression for it over the 
generators $\ss_1$, $\ss_2$, $\ss_3$, and $\ss_4$ that is given by
Proposition \ref{PropLiftReflections}. Let $\rrefl\rrefl'$, 
$\rrefl'\rrefl''$ be the two sides of a dual braid relation. 
Then there exists $w_0 \preceq w$ of length $3$ such that 
$\refl\refl' = \refl'\refl'' \preceq w_0$. By \cite{BessisDualMonoid}, we know that the group 
$G([1,w_0])$ is isomorphic to the group defined by a presentation that we have described by Coxeter diagrams of type $A_3$ in the last column of Table~\ref{TableDiagLiftn4}. 
Hence we obtain that the dual braid relation $\rrefl\rrefl'=\rrefl'\rrefl''$ is a consequence of the relations of the corresponding diagram in the table.

In addition, we have already shown in Proposition~\ref{PropCheckAllRelationsN4} that the 
relations of these diagrams are consequences of the relations we have associated with the diagram $\Delta$. 
Hence the map $g$ is a homomorphism.

Clearly, the composition $f \circ g$ is equal to $id_{G([1,w])}$ and $g \circ f$ equal to $id_{\GG}$. 
Therefore, the groups $\GG$ and $G([1,w])$ are isomorphic.
\end{proof}

\subsection{The case $n=5$}\label{SubsectionN5}

Let $W$ be the Coxeter group of type $D_5$. It has a presentation on the five generators $s_2 = (1,2)$, $s_3 = (2,3)$, $s_4 = (3,4)$, $s_5 = (4,5)$, and $s_1 = (\overline{2},\overline{3})$ (see Proposition \ref{PropCameron}). Here $m$ is equal to $2$ and the corresponding proper quasi-Coxeter element is $w = (2,\overline{1})(5,4,\overline{3})$. 

We prove that the interval group $G([1,w])$  is isomorphic to the 
group $\GG$ with five generators $\ss_1, \ss_2, \ss_3, \ss_4$, and $\ss_5$ corresponding to $s_1,s_2,s_3,s_4$, $s_5$, with relations described by the diagram of Figure \ref{FigDiagClaimedPres},
where the curved arrow describes the twisted cycle commutator relator: $\tc{\ss_1}{\ss_2}{\ss_3}{\ss_4} = [\ss_1,\bs{s_4^{s_3^{-1}s_2}}]$.

Similarly to Table \ref{TableExprrefl}, we provide decompositions of $\rrefl$ by applying Proposition \ref{PropLiftReflections} and we divide them according to the three types I, II, III of elements in $\RRefl$.

\begin{table}[h]
\begin{center}
\begin{tabular}{|c|c|}
\hline
Number & Decomposition of $\rrefl$ (type I)\\
\hline
1. & $\bs{(1,3)} = \bs{s_3^{s_2^{-1}}}$\\
\hline
2. & $\bs{(1,4)} = \bs{s_4^{s_3^{-1}s_2^{-1}}}$\\
\hline
3. & $\bs{(1,5)} = \bs{s_5^{s_4^{-1}s_3^{-1}s_2^{-1}}}$\\
\hline
4. &  $\bs{(2,3)} = \ss_3$\\
\hline
5. &  $\bs{(2,4)} = \bs{s_4^{s_3^{-1}}}$\\
\hline
6. & $\bs{(2,5)} = \bs{s_5^{s_4^{-1}s_3^{-1}}}$\\
\hline
7. & $\bs{(\overline{1},\overline{3})} = \bs{s_1^{s_2^{-1}}}$\\
\hline
8. & $\bs{(\overline{1},\overline{4})} = \bs{s_1^{s_4s_2^{-1}}}$\\
\hline
9. & $\bs{(\overline{1},\overline{5})} = \bs{s_1^{s_4s_5s_2^{-1}}}$\\
\hline
10. & $\bs{(\overline{2},\overline{3})} = \ss_1$\\
\hline
11. & $\bs{(\overline{2},\overline{4})} = \bs{s_1^{s_4}}$\\
\hline
12. &  $\bs{(\overline{2},\overline{5})} = \bs{s_1^{s_4s_5}}$\\
\hline
 & Decomposition of $\rrefl$ (type II)\\
\hline
13. &  $\bs{(1,2)} = \ss_2$\\
\hline
14. & $\bs{(\overline{1},\overline{2})} = \bs{s_1^{s_2^{-1}s_3}}$\\
\hline
 & Decomposition of $\rrefl$ (type III)\\
\hline
15. &   $\bs{(3,4)} = \ss_4$\\
\hline
16. & $\bs{(3,5)} = \bs{s_5^{s_4^{-1}}}$\\
\hline
17. & $\bs{(4,5)} = \ss_5$\\
\hline
18. & $\bs{(\overline{3},\overline{4})} = \bs{s_1^{s_4s_3^{-1}}}$\\
\hline
19. &  $\bs{(\overline{3},\overline{5})} = \bs{s_1^{s_4s_5s_3^{-1}}}$\\
\hline
20. & $\bs{(\overline{4},\overline{5})} = \bs{s_1^{s_4s_5s_3^{-1}s_4}}$\\
\hline
\end{tabular}
\end{center}
\caption{Decompositions of $\rrefl$ in the case $n=5$.}
\label{TableDecompN5}
\end{table}

Table~\ref{TableDiagLiftN5} contains the same information as Table~\ref{TableDiagLiftn4} in the case $n=4$. 
We have $20$ divisors of $w$ of length $4$ (the $w_0$'s) obtained by multiplying $w$ from the right by $(i,j)$ and $(\overline{i},\overline{j})$ for $1 \leq i < j \leq 5$. These divisors belong to types I, II, and III. We separate each type by $2$ lines in the table. The third column produces the reduced decomposition from Sections \ref{SubsectionRedDecompI} and \ref{SubsectionRedDecompII}. The last column describes the lift of the diagram $\Delta_0$ that encodes the relations between the lift to the interval group of two reflections that appear in the reduced decomposition of each $w_0$.

We showed, using \verb|kbmag| within \verb|GAP| that all the relations described in the diagrams of the last column are consequences of the relations of the claimed presentation (see Theorem~\ref{ThmPresDn1}). The only cases that correspond to proper quasi-Coxeter elements are numbers $15$, $17$, and $19$ of Table~\ref{TableDiagLiftN5}. Let $w_0$ be one of these elements. We know from Proposition~\ref{PropPresD4} that $G([1,w_0])$ is isomorphic to the group defined by a presentation associated to the square diagram with the twisted cycle commutator relator. We conclude with the statement of the result for $n=5$, whose proof we omit since it is similar to
the proof of Proposition \ref{PropPresD4}.

\begin{proposition}\label{PropPresD5}

	In the case $n=5$,
the groups $\GG$ and $G([1,w])$ are isomorphic.

\end{proposition}

\newpage
\newgeometry{paperwidth=18cm,
	left=2cm,
	right=2cm,
	paperheight=32cm,
	top=1cm,
	bottom=2cm}

\begin{table}
\caption{Reduced decompositions and diagram lifts.}
\begin{center}
\begin{tabular}{|p{0.3cm}|c|c|c|}
\hline
 & Maximal divisor $w_0$ & Reduced decomposition & Lift of $\Delta_0$\\
\hline \hline
\multirow{2}*{1.} &  \multirow{2}*{$w(1,3) = (5,4,\overline{1},2,\overline{3})$} &  \multirow{2}*{$(2,3)(\overline{1},\overline{2})(1,4)(4,5)$} &  \multirow{2}*{ \begin{tikzpicture}
\node[draw, shape=circle, label=below:{$\bs{(2,3)}$}] (1) at (0,0) {};
\node[draw, shape=circle, label=below:{$\bs{(\overline{1},\overline{2})}$}] (2) at (1,0) {};
\node[draw, shape=circle, label=below:{$\bs{(1,4)}$}] (3) at (2,0) {};
\node[draw, shape=circle, label=below:{$\bs{(4,5)}$}] (4) at (3,0) {};
\draw[thick,-] (1) to (2);
\draw[thick,-] (2) to (3);
\draw[thick,-] (4) to (3);
\end{tikzpicture}}\\

& & & \\
\hline

\multirow{2}*{2.} &  \multirow{2}*{$w(1,4) = (5,\overline{1},2,4,\overline{3})$} &  \multirow{2}*{$(3,4)(2,4)(\overline{1},\overline{2})(1,5)$} &  \multirow{2}*{ \begin{tikzpicture}
\node[draw, shape=circle, label=below:{$\bs{(3,4)}$}] (1) at (0,0) {};
\node[draw, shape=circle, label=below:{$\bs{(2,4)}$}] (2) at (1,0) {};
\node[draw, shape=circle, label=below:{$\bs{(\overline{1},\overline{2})}$}] (3) at (2,0) {};
\node[draw, shape=circle, label=below:{$\bs{(1,5)}$}] (4) at (3,0) {};
\draw[thick,-] (1) to (2);
\draw[thick,-] (2) to (3);
\draw[thick,-] (4) to (3);
\end{tikzpicture}}\\

& & & \\
\hline

\multirow{2}*{3.} &  \multirow{2}*{$w(1,5) = (5,4,\overline{3},\overline{1},2)$} &  \multirow{2}*{$(\overline{1},\overline{2})(\overline{1},\overline{3})(3,4)(4,5)$} &  \multirow{2}*{ \begin{tikzpicture}
\node[draw, shape=circle, label=below:{$\bs{(\overline{1},\overline{2})}$}] (1) at (0,0) {};
\node[draw, shape=circle, label=below:{$\bs{(\overline{1},\overline{3})}$}] (2) at (1,0) {};
\node[draw, shape=circle, label=below:{$\bs{(3,4)}$}] (3) at (2,0) {};
\node[draw, shape=circle, label=below:{$\bs{(4,5)}$}] (4) at (3,0) {};
\draw[thick,-] (1) to (2);
\draw[thick,-] (2) to (3);
\draw[thick,-] (4) to (3);
\end{tikzpicture}}\\

& & & \\
\hline

\multirow{2}*{4.} &  \multirow{2}*{$w(2,3) = (5,4,2,\overline{1},\overline{3})$} &  \multirow{2}*{$(\overline{1},\overline{3})(1,2)(2,4)(4,5)$} &  \multirow{2}*{ \begin{tikzpicture}
\node[draw, shape=circle, label=below:{$\bs{(\overline{1},\overline{3})}$}] (1) at (0,0) {};
\node[draw, shape=circle, label=below:{$\bs{(1,2)}$}] (2) at (1,0) {};
\node[draw, shape=circle, label=below:{$\bs{(2,4)}$}] (3) at (2,0) {};
\node[draw, shape=circle, label=below:{$\bs{(4,5)}$}] (4) at (3,0) {};
\draw[thick,-] (1) to (2);
\draw[thick,-] (2) to (3);
\draw[thick,-] (4) to (3);
\end{tikzpicture}}\\

& & & \\
\hline

\multirow{2}*{5.} &  \multirow{2}*{$w(2,4) = (5,2,\overline{1},4,\overline{3})$} &  \multirow{2}*{$(3,4)(\overline{1},\overline{4})(1,2)(2,5)$} &  \multirow{2}*{ \begin{tikzpicture}
\node[draw, shape=circle, label=below:{$\bs{(3,4)}$}] (1) at (0,0) {};
\node[draw, shape=circle, label=below:{$\bs{(\overline{1},\overline{4})}$}] (2) at (1,0) {};
\node[draw, shape=circle, label=below:{$\bs{(1,2)}$}] (3) at (2,0) {};
\node[draw, shape=circle, label=below:{$\bs{(2,5)}$}] (4) at (3,0) {};
\draw[thick,-] (1) to (2);
\draw[thick,-] (2) to (3);
\draw[thick,-] (4) to (3);
\end{tikzpicture}}\\

& & & \\
\hline

\multirow{2}*{6.} &  \multirow{2}*{$w(2,5) = (5,4,\overline{3},2,\overline{1})$} &  \multirow{2}*{$(1,2)(\overline{2},\overline{3})(3,4)(4,5)$} &  \multirow{2}*{ \begin{tikzpicture}
\node[draw, shape=circle, label=below:{$\bs{(1,2)}$}] (1) at (0,0) {};
\node[draw, shape=circle, label=below:{$\bs{(\overline{2},\overline{3})}$}] (2) at (1,0) {};
\node[draw, shape=circle, label=below:{$\bs{(3,4)}$}] (3) at (2,0) {};
\node[draw, shape=circle, label=below:{$\bs{(4,5)}$}] (4) at (3,0) {};
\draw[thick,-] (1) to (2);
\draw[thick,-] (2) to (3);
\draw[thick,-] (4) to (3);
\end{tikzpicture}}\\

& & & \\
\hline

\multirow{2}*{7.} &  \multirow{2}*{$w(\overline{1},\overline{3}) = (5,\overline{4},\overline{1},\overline{2},\overline{3})$} &  \multirow{2}*{$(\overline{2},\overline{3})(\overline{1},\overline{2})(\overline{1},\overline{4})(4,5)$} &  \multirow{2}*{ \begin{tikzpicture}
\node[draw, shape=circle, label=below:{$\bs{(\overline{2},\overline{3})}$}] (1) at (0,0) {};
\node[draw, shape=circle, label=below:{$\bs{(\overline{1},\overline{2})}$}] (2) at (1,0) {};
\node[draw, shape=circle, label=below:{$\bs{(\overline{1},\overline{4})}$}] (3) at (2,0) {};
\node[draw, shape=circle, label=below:{$\bs{(4,5)}$}] (4) at (3,0) {};
\draw[thick,-] (1) to (2);
\draw[thick,-] (2) to (3);
\draw[thick,-] (4) to (3);
\end{tikzpicture}}\\

& & & \\
\hline

\multirow{2}*{8.} &  \multirow{2}*{$w(\overline{1},\overline{4}) = (\overline{5},\overline{1},\overline{2},4,\overline{3})$} &  \multirow{2}*{$(3,4)(\overline{2},\overline{4})(\overline{1},\overline{2})(\overline{1},\overline{5})$} &  \multirow{2}*{ \begin{tikzpicture}
\node[draw, shape=circle, label=below:{$\bs{(3,4)}$}] (1) at (0,0) {};
\node[draw, shape=circle, label=below:{$\bs{(\overline{2},\overline{4})}$}] (2) at (1,0) {};
\node[draw, shape=circle, label=below:{$\bs{(\overline{1},\overline{2})}$}] (3) at (2,0) {};
\node[draw, shape=circle, label=below:{$\bs{(\overline{1},\overline{5})}$}] (4) at (3,0) {};
\draw[thick,-] (1) to (2);
\draw[thick,-] (2) to (3);
\draw[thick,-] (4) to (3);
\end{tikzpicture}}\\

& & & \\
\hline

\multirow{2}*{9.} &  \multirow{2}*{$w(\overline{1},\overline{5}) = (5,4,3,\overline{1},\overline{2})$} &  \multirow{2}*{$(\overline{1},\overline{2})(1,3)(3,4)(4,5)$} &  \multirow{2}*{ \begin{tikzpicture}
\node[draw, shape=circle, label=below:{$\bs{(\overline{1},\overline{2})}$}] (1) at (0,0) {};
\node[draw, shape=circle, label=below:{$\bs{(1,3)}$}] (2) at (1,0) {};
\node[draw, shape=circle, label=below:{$\bs{(3,4)}$}] (3) at (2,0) {};
\node[draw, shape=circle, label=below:{$\bs{(4,5)}$}] (4) at (3,0) {};
\draw[thick,-] (1) to (2);
\draw[thick,-] (2) to (3);
\draw[thick,-] (4) to (3);
\end{tikzpicture}}\\

& & & \\
\hline

\multirow{2}*{10.} &  \multirow{2}*{$w(\overline{2},\overline{3}) = (5,\overline{4},2,1,\overline{3})$} &  \multirow{2}*{$(1,3)(1,2)(\overline{2},\overline{4})(4,5)$} &  \multirow{2}*{ \begin{tikzpicture}
\node[draw, shape=circle, label=below:{$\bs{(1,3)}$}] (1) at (0,0) {};
\node[draw, shape=circle, label=below:{$\bs{(1,2)}$}] (2) at (1,0) {};
\node[draw, shape=circle, label=below:{$\bs{(\overline{2},\overline{4})}$}] (3) at (2,0) {};
\node[draw, shape=circle, label=below:{$\bs{(4,5)}$}] (4) at (3,0) {};
\draw[thick,-] (1) to (2);
\draw[thick,-] (2) to (3);
\draw[thick,-] (4) to (3);
\end{tikzpicture}}\\

& & & \\
\hline

\multirow{2}*{11.} &  \multirow{2}*{$w(\overline{2},\overline{4}) = (\overline{5},2,1,4,\overline{3})$} &  \multirow{2}*{$(3,4)(1,4)(1,2)(\overline{2},\overline{5})$} &  \multirow{2}*{ \begin{tikzpicture}
\node[draw, shape=circle, label=below:{$\bs{(3,4)}$}] (1) at (0,0) {};
\node[draw, shape=circle, label=below:{$\bs{(1,4)}$}] (2) at (1,0) {};
\node[draw, shape=circle, label=below:{$\bs{(1,2)}$}] (3) at (2,0) {};
\node[draw, shape=circle, label=below:{$\bs{(\overline{2},\overline{5})}$}] (4) at (3,0) {};
\draw[thick,-] (1) to (2);
\draw[thick,-] (2) to (3);
\draw[thick,-] (4) to (3);
\end{tikzpicture}}\\

& & & \\
\hline

\multirow{2}*{12.} &  \multirow{2}*{$w(\overline{2},\overline{5}) = (5,4,3,2,1)$} &  \multirow{2}*{$(1,2)(2,3)(3,4)(4,5)$} &  \multirow{2}*{ \begin{tikzpicture}
\node[draw, shape=circle, label=below:{$\bs{(1,2)}$}] (1) at (0,0) {};
\node[draw, shape=circle, label=below:{$\bs{(2,3)}$}] (2) at (1,0) {};
\node[draw, shape=circle, label=below:{$\bs{(3,4)}$}] (3) at (2,0) {};
\node[draw, shape=circle, label=below:{$\bs{(4,5)}$}] (4) at (3,0) {};
\draw[thick,-] (1) to (2);
\draw[thick,-] (2) to (3);
\draw[thick,-] (4) to (3);
\end{tikzpicture}}\\

& & & \\
\hline \hline

\multirow{4}*{13.} &  \multirow{4}*{$w(1,2) = (\overline{1})(2)(5,4,\overline{3})$} &  \multirow{4}*{$(1,3)(\overline{1},\overline{3})(3,4)(4,5)$} &  \multirow{4}*{ \begin{tikzpicture}
\node[draw, shape=circle, label=below:{$\bs{(3,4)}$}] (1) at (0,0) {};
\node[draw, shape=circle, label=below:{$\bs{(4,5)}$}] (2) at (1,0) {};
\node[draw, shape=circle, label=left:{$\bs{(1,3)}$}] (3) at (-1,0.6) {};
\node[draw, shape=circle, label=left:{$\bs{(\overline{1},\overline{3})}$}] (4) at (-1,-0.6) {};
\draw[thick,-] (1) to (2);
\draw[thick,-] (1) to (3);
\draw[thick,-] (4) to (1);
\end{tikzpicture}}\\

& & & \\
& & & \\
& & & \\
\hline

\multirow{4}*{14.} &  \multirow{4}*{$w(\overline{1},\overline{2}) = (1)(\overline{1})(5,4,\overline{3})$} &  \multirow{4}*{$(2,3)(\overline{2},\overline{3})(3,4)(4,5)$} &  \multirow{4}*{ \begin{tikzpicture}
\node[draw, shape=circle, label=below:{$\bs{(3,4)}$}] (1) at (0,0) {};
\node[draw, shape=circle, label=below:{$\bs{(4,5)}$}] (2) at (1,0) {};
\node[draw, shape=circle, label=left:{$\bs{(2,3)}$}] (3) at (-1,0.6) {};
\node[draw, shape=circle, label=left:{$\bs{(\overline{2},\overline{3})}$}] (4) at (-1,-0.6) {};
\draw[thick,-] (1) to (2);
\draw[thick,-] (1) to (3);
\draw[thick,-] (4) to (1);
\end{tikzpicture}}\\

& & & \\
& & & \\
& & & \\
\hline \hline

\multirow{6}*{15.} &  \multirow{6}*{$w(3,4) = (2,\overline{1})(5,\overline{3})(4)$} &  \multirow{6}*{$(1,3)(\overline{1},\overline{3})(1,2)(3,5)$} &  \multirow{6}*{ \begin{tikzpicture}
\node[draw, shape=circle, label=left:{$\bs{(1,2)}$}] (5) at (5,0) {};
\node[draw, shape=circle, label=right:{$\bs{(1,3)}$}] (6) at (6,1) {};
\node[draw, shape=circle, label=right:{$\bs{(\overline{1},\overline{3})}$}] (7) at (6,-1) {};
\node[draw, shape=circle,label=right:{$\bs{(3,5)}$}] (8) at (7,0) {};
\node[] (curve) at (6,0) {\Huge $\circlearrowright$};
\draw[thick,-] (5) to (6);
\draw[thick,-] (5) to (7);
\draw[thick,-] (8) to (6);
\draw[thick,-] (8) to (7);
\end{tikzpicture}}\\

& & & \\
& & & \\
& & & \\
& & & \\
& & & \\
\hline

\multirow{2}*{16.} &  \multirow{2}*{$w(3,5) = (2,\overline{1})(\overline{3})(5,4)$} &  \multirow{2}*{$(1,3)(\overline{1},\overline{3})(1,2)(4,5)$} &  \multirow{2}*{ \begin{tikzpicture}
\node[draw, shape=circle, label=below:{$\bs{(1,3)}$}] (1) at (0,0) {};
\node[draw, shape=circle, label=below:{$\bs{(1,2)}$}] (2) at (1,0) {};
\node[draw, shape=circle, label=below:{$\bs{(\overline{1},\overline{3})}$}] (3) at (2,0) {};
\node[draw, shape=circle, label=below:{$\bs{(4,5)}$}] (4) at (3,0) {};
\draw[thick,-] (1) to (2);
\draw[thick,-] (2) to (3);
\end{tikzpicture}}\\

& & & \\
\hline

\multirow{6}*{17.} &  \multirow{6}*{$w(4,5) = (2,\overline{1})(4,\overline{3})(5)$} &  \multirow{6}*{$(1,3)(\overline{1},\overline{3})(1,2)(3,4)$} &  \multirow{6}*{ \begin{tikzpicture}
\node[draw, shape=circle, label=left:{$\bs{(1,2)}$}] (5) at (5,0) {};
\node[draw, shape=circle, label=right:{$\bs{(1,3)}$}] (6) at (6,1) {};
\node[draw, shape=circle, label=right:{$\bs{(\overline{1},\overline{3})}$}] (7) at (6,-1) {};
\node[draw, shape=circle,label=right:{$\bs{(3,4)}$}] (8) at (7,0) {};
\node[] (curve) at (6,0) {\Huge $\circlearrowright$};
\draw[thick,-] (5) to (6);
\draw[thick,-] (5) to (7);
\draw[thick,-] (8) to (6);
\draw[thick,-] (8) to (7);
\end{tikzpicture}}\\

& & & \\
& & & \\
& & & \\
& & & \\
& & & \\

\hline
\multirow{2}*{18.} &  \multirow{2}*{$w(\overline{3},\overline{4}) = (2,\overline{1})(\overline{5},\overline{3})(\overline{4})$} &  \multirow{2}*{$(1,4)(\overline{1},\overline{4})(1,2)(\overline{3},\overline{5})$} &  \multirow{2}*{ \begin{tikzpicture}
\node[draw, shape=circle, label=below:{$\bs{(1,4)}$}] (1) at (0,0) {};
\node[draw, shape=circle, label=below:{$\bs{(1,2)}$}] (2) at (1,0) {};
\node[draw, shape=circle, label=below:{$\bs{(\overline{1},\overline{4})}$}] (3) at (2,0) {};
\node[draw, shape=circle, label=below:{$\bs{(\overline{3},\overline{5})}$}] (4) at (3,0) {};
\draw[thick,-] (1) to (2);
\draw[thick,-] (2) to (3);
\end{tikzpicture}}\\

& & & \\
\hline

\multirow{6}*{19.} &  \multirow{6}*{$w(\overline{3},\overline{5}) = (2,\overline{1})(3)(5,\overline{4})$} &  \multirow{6}*{$(1,4)(\overline{1},\overline{4})(1,2)(4,5)$} &  \multirow{6}*{ \begin{tikzpicture}
\node[draw, shape=circle, label=left:{$\bs{(1,2)}$}] (5) at (5,0) {};
\node[draw, shape=circle, label=right:{$\bs{(1,4)}$}] (6) at (6,1) {};
\node[draw, shape=circle, label=right:{$\bs{(\overline{1},\overline{4})}$}] (7) at (6,-1) {};
\node[draw, shape=circle,label=right:{$\bs{(4,5)}$}] (8) at (7,0) {};
\node[] (curve) at (6,0) {\Huge $\circlearrowright$};
\draw[thick,-] (5) to (6);
\draw[thick,-] (5) to (7);
\draw[thick,-] (8) to (6);
\draw[thick,-] (8) to (7);
\end{tikzpicture}}\\

& & & \\
& & & \\
& & & \\
& & & \\
& & & \\
\hline

\multirow{2}*{20.} &  \multirow{2}*{$w(\overline{4},\overline{5}) = (2,\overline{1})(4,3)(\overline{5})$} &  \multirow{2}*{$(1,5)(\overline{1},\overline{5})(1,2)(3,4)$} &  \multirow{2}*{ \begin{tikzpicture}
\node[draw, shape=circle, label=below:{$\bs{(1,5)}$}] (1) at (0,0) {};
\node[draw, shape=circle, label=below:{$\bs{(1,2)}$}] (2) at (1,0) {};
\node[draw, shape=circle, label=below:{$\bs{(\overline{1},\overline{5})}$}] (3) at (2,0) {};
\node[draw, shape=circle, label=below:{$\bs{(3,4)}$}] (4) at (3,0) {};
\draw[thick,-] (1) to (2);
\draw[thick,-] (2) to (3);
\end{tikzpicture}}\\

& & & \\
\hline

\end{tabular}
\end{center}
\label{TableDiagLiftN5}
\end{table}

\restoregeometry

\subsection{Lifting the reduced decompositions}\label{SubsectionLiftRedDecomp}

This section establishes Step 4 in our strategy that we have described in Section \ref{SubsectionStrategyProof}. 

Let $w$ be the quasi-Coxeter element $(m,m-1,\dotsc,2,\bar{1})(n,n-1,\dotsc,\overline{m+1})$ in type $D_n$. We define $g$ to be the map from $G([1,w])$ to $\GG$ that sends $\refl_i$ to its decomposition over the generating set $\SS$ of $\GG$ given by Proposition~\ref{PropLiftReflections}. In this section, we prove the following.

\begin{proposition}\label{PropLiftRedDecomp}
	        Let $w_0$ be a divisor of length $n-1$ of the quasi-Coxeter element $w$, let $\refl_1 \refl_2 \dotsc \refl_{n-1}$ be the reduced decomposition of $w_0$
        obtained using the results of Sections~\ref{SubsectionRedDecompI}, \ref{SubsectionRedDecompII}, and let
        $\Delta_0$ be the associated decomposition diagram
        (described in Propositions~\ref{PropDiagTypeI}, \ref{PropCoxDiagTypeII,III} and \ref{PropQCoxDiagTypeII,III}).
	Then for each of the relators $\braid{\refl_i}{\refl_j}$ and $\tc{\refl_i}{\refl_j}{\refl_k}{\refl_l}$
	between the reflections $\refl_i$ that is implied by the diagram $\Delta_0$,
	the corresponding relators $\braid{g(\rrefl_i)}{g(\rrefl_j)}$ or $\tc{g(\rrefl_i)}{g(\rrefl_j)}{g(\rrefl_k)}{g(\rrefl_l)}$ can be derived from the relations of the presentation of $\GG$ given in
	Theorem~\ref{ThmPresDn1}.
\end{proposition}
\begin{proof}
	The proof is by induction on $n$. 	
	The proposition is proved for $n=4$ and $n=5$ in Sections~\ref{SubsectionN4},\ref{SubsectionN5} within the proofs of
	Propositions~\ref{PropPresD4} and~\ref{PropPresD5}. 
	
	So now let $n \geq 6$.
	Set $w' := (m,m-1,\dotsc,2,\overline{1})(n-1,n-2,\dotsc, \overline{m+1})$. Then $$w' = s_2 s_3 \dotsc s_m s_1 s_{m+1}s_{m+2} s_{m+3} s_{m+4} \dotsc s_{n-1}$$ with 
	diagram $\Delta_{m,n-1}$ by  Proposition~\ref{PropConjQCoxBipartite}, and we have $P:= P_{w'} = \langle s_1, \ldots , s_{n-1}\rangle$. 
	Moreover, the braid relators in the generators $g(\ss_i)$ and $g(\ss_j)$ as well as the twisted cycle relator for $w'$, which we will call the 
	$w'$-relators, are a subset of the $\GG$-relators.
	
	There are $11$ different possibilities for $w_0$ that are described in Section~\ref{SubsectionDivisorsln-1} by Equations~\ref{Eqt1}--\ref{Eqt11}.
	For the $11$-th equation we need to deal separately with the cases $n > i+1 $ and $n = i+1$.
	
	So suppose first that $w_0$  is either as in one of Equations $1-10$ or as in Equation~$11$ with $n > i+1$.
		In any of these cases, in the cycle decomposition of $w_0$ the number $n$ is only overlined in a cycle that has an even number of overlined entries.
	This implies that
	at most one of the reflections $\refl_1 , \ldots , \refl_{n-1}$ is not contained in $P$ and that this reflection corresponds to an end node of $\Delta_0$ and is without 
		loss of generality $t_{n-1}$. Thus we have $\refl_1 , \ldots , \refl_{n-2} \in P$. Set $w_1:= w_0 t_{n-1} =  \refl_1  \cdots  \refl_{n-2} $. Then we get by Lemma~5.3 in \cite{BM} that $w_1$ is a divisor  of length 
	$\ell_T(w_1) = n-2$ of $w'$.  Further $w'$ is of length $n-1$.
	By induction, the relators $\braid{g(\rrefl_i)}{g(\rrefl_j)}$ and $\tc{g(\rrefl_i)}{g(\rrefl_j)}{g(\rrefl_k)}{g(\rrefl_l)}$ are a consequence of the  $w'$-relators, which are $\GG$-relators.
	
	Hence it only remains to show that $\braid{g(\rrefl_i)}{g(\rrefl_j)}$ and $\tc{g(\rrefl_i)}{g(\rrefl_j)}{g(\rrefl_k)}{g(\rrefl_l)}$ are a  consequence of the $\GG$-relators  under the assumption that $i = n-1$. 
	This is done in  the appendix in  Lemmas~\ref{LemLiftRedDecEqt1}, $\dotsc$, \ref{LemLiftRedDecEqt91011}.
	Thereby notice, as $t_{n-1}$ corresponds to an end node of $\Delta_0$, it is not contained in a cycle  of $\Delta_0$ and the relator $\tc{\refl_{n-1}}{\refl_j}{\refl_k}{\refl_l}$ does not appear.\\
	
	Now suppose that $w_0$ is as in Equation~\ref{Eqt11} and that $n$ is overlined. Then $\Delta_0$ is the union of three strings, one of length $1$. The reflections  $\refl_i$ are in $P$ beside
	that one corresponding to the single vertex and one of the other 4 end nodes of the strings of $\Delta_0$. By induction it remains to derive the braid relators  for the two  just mentioned  reflections from the $\GG$-relators, which is 
	treated in Lemma~\ref{LemLiftRedDecEqt91011}.
\end{proof}

In Appendix A we establish the proof of the lemmas we refer to in the proof of Proposition~\ref{PropLiftRedDecomp}.

\subsection{The proof for $n>5$}\label{SubsectionProofn>5}

We are in position to prove Theorem \ref{ThmPresDn1}. The details of the proof are discussed and commented in our strategy developed in Section~\ref{SubsectionStrategyProof}.

Consider the map $f: \GG \longrightarrow G([1,w]): \ss_i \longmapsto \ss_i$. By Proposition \ref{PropPreparationHomoF} and Lemma~\ref{LemmaTC}, the relations of the presentation of $\GG$ hold in $G([1,w])$. This is Step 1 in the strategy of the proof.

Consider the map $g: G([1,w]) \longrightarrow \GG$ that maps each generator 
$\rrefl$ of the generating set $\RRefl$ of $G([1,w])$ to its decomposition
on the generators $\ss_1$, $\ss_2$, $\dotsc$, $\ss_n$ that we described 
in Equations \ref{EqtLiftDecomTransp} to \ref{EqtLiftDecomIII} within Proposition~\ref{PropLiftReflections}. This is Step 3 in the strategy of the proof that was established within Section~\ref{SecDecompReflections}.

We want to show that $g$ is a homomorphism. Consider a dual braid relation of $G([1,w])$, meaning a relation of the form 
$\rrefl\rrefl' = \rrefl'\rrefl''$ for $\rrefl,\rrefl',\rrefl'' \in \RRefl$. It 
corresponds to the fact that $\refl\refl' = \refl'\refl'' \preceq w$. Then we have to prove that the relation $g(\rrefl)g(\rrefl') = g(\rrefl')g(\rrefl'')$ is a consequence of the relations of the presentation of $\GG$. 
As $([1,w], \preceq )$ is a
	graded poset in which all the maximal flags have the same length, there are divisors $w_0$ of length $n-1$ of $w$ such that $\refl\refl'  \preceq w_0$.  By Proposition~\ref{PropLiftRedDecomp}, the braid relations and the twisted cycle commutator relator that correspond to the reduced decomposition and the diagram $\Delta_0$ for $w_0$ (produced in  Sections \ref{SubsectionRedDecompI} and \ref{SubsectionRedDecompII}) are a consequence of the $\GG$-relations
	(see Step 4 of our strategy). By induction, the relation $g(\rrefl)g(\rrefl') = g(\rrefl')g(\rrefl'')$ is a consequence of the  braid relations and the twisted cycle commutator relator related to the reduced decomposition in the $g$-image of the lift of $P_{w_0}$ to $\GG$. Therefore, $g$ is a homomorphism.
	
	
	
	Clearly, the composition $f \circ g$ is equal to $id_{G([1,w])}$ and $g \circ f$ is equal to $id_{\GG}$. Therefore, the groups $\GG$ and $G([1,w])$ are isomorphic.

\begin{small}

\end{small}

\appendix

\section{Proof of the lemmas}

In the calculations within the proofs of the next lemmas, we underline an expression being manipulated for emphasis.

\begin{lemma}\label{LemLiftRedDecEqt1}
	
	Suppose that $w_0=w(i,j)$, with $1\leq i \leq m$, $m+1\leq j \leq n$,
	so that we are in the situation of Equation~\ref{Eqt1}. Let $\refl_1\refl_2 \dotsc \refl_{n-1}$ be the reduced decomposition of $w_0$ described in Lemma~\ref{LemRedDecEqt1}.
	Then we can deduce from the relations of the presentation of $\GG$ that $g(\rrefl_{n-1}) g(\rrefl_{n-2}) g(\rrefl_{n-1}) = g(\rrefl_{n-2}) g(\rrefl_{n-1}) g(\rrefl_{n-2})$ and $g(\rrefl_{n-1})$ commutes with
	each of the elements $g(\rrefl_{k})$  with $k<n-2$.
	
\end{lemma}

\begin{proof}
	We consider the seven different possible decompositions of $w_0$ that are
	described
	in Lemma~\ref{LemRedDecEqt1}. 
	In each case we do not need to consider the relations between $g(\rrefl_{n})$ 
	and $g(\rrefl_k)$ if both elements are within the set $\SS$.  
	
	(1) Where $i \neq m$ and $j \neq n-2, n-1, n$, we need to check that
	the element $g(\ss_n)=g(\bs{(n-1,n)})$ commutes with each of
	$g\left( \bs{(i+1,j)} \right)$, $g\left(\bs{(i,j+1)}\right)$, and $g\left(\bs{(\overline{1},\overline{m})}\right)$.
	
	By Proposition~\ref{PropLiftReflections}, we have that \begin{center}$g\left(\bs{(i+1,j)}\right) = \ss_j^{\ss_{j-1}^{-1} \ss_{j-2}^{-1} \dotsc \ss_{i+2}^{-1}}$ and $g\left(\bs{(i,j+1)}\right) = \ss_{j+1}^{\ss_j^{-1}\ss_{j-1}^{-1}\dotsc \ss_{i+1}^{-1}}$.\end{center}
	It follows from the relations of the presentation of $\GG$ that $g\left(\ss_n\right)$ commutes with both $g\left(\bs{(i+1,j)}\right)$ and $g\left(\bs{(i,j+1)}\right)$.
	
	We also have that $g\left(\bs{(\overline{1},\overline{m})}\right) = s_1^{\ss_m^{-1}\ss_{m-1}^{-1}\dotsc \ss_2^{-1}\ss_{m+1}}$. Then it also follows from the relations of the presentation of $\GG$ that $g\left(\ss_n\right)$ commutes with $g\left(\bs{(\overline{1},\overline{m})}\right)$.\\
	
	(2) Where $i \neq m$ and $j=n-2$, we need to check that
	the element $g\left(\ss_n\right)=g\left(\bs{(n-1,n)}\right)$ commutes with each of
	$g\left(\bs{(i+1,n-2)}\right)$, $g\left(\bs{(\overline{1},\overline{m})}\right)$, and that 
	the relation $g\left(\ss_n\right)g\left(\bs{(i,n-1)}\right)g\left(\ss_n\right) = g\left(\bs{(i,n-1)}\right)g\left(\ss_n\right) g\left(\bs{(i,n-1)}\right)$ holds.
	
	We have already shown in item (1) of this proof that $g\left(\ss_n\right)$ commutes with $g\left(\bs{(\overline{1},\overline{m})}\right)$. 
	
	We have that $g\left(\bs{(i+1,n-2)}\right) = \ss_{n-2}^{\ss_{n-3}^{-1}\ss_{n-4}^{-1}\dotsc \ss_{i+2}^{-1}}$. It follows directly from the relations of $\GG$ that $g\left(\ss_n\right)$ commutes with $g\left(\bs{(i+1,n-2)}\right)$.
	
	Now we prove that  $g\left(\ss_n\right)g\left(\bs{(i,n-1)}\right)g\left(\ss_n\right) = g\left(\bs{(i,n-1)}\right)g\left(\ss_n\right) g\left(\bs{(i,n-1)}\right)$:
	
	\begin{tabular}{lll}
		$g\left(\ss_n\right)g\left(\bs{(i,n-1)}\right)g\left(\ss_n\right)$ & $=$ & $ \ss_{i+1} \dotsc \ss_{n-3} \ss_{n-2} \underline{\ss_n \ss_{n-1} \ss_n} \ss_{n-2}^{-1} \ss_{n-3}^{-1} \dotsc \ss_{i+1}^{-1}$\\
		& $=$ & $ \ss_{i+1} \dotsc \ss_{n-3} \ss_{n-2} \ss_{n-1} \ss_{n} \ss_{n-1} \ss_{n-2}^{-1} \ss_{n-3}^{-1} \dotsc \ss_{i+1}^{-1}$\\
		&$=$ & $ (\ss_{i+1} \dotsc \ss_{n-3} \ss_{n-2} \ss_{n-1} \ss_{n-2}^{-1} \ss_{n-3}^{-1} \dotsc \ss_{i+1}^{-1})$\\
		& & $(\ss_{i+1} \dotsc \ss_{n-3}\ss_{n-2} \underline{\ss_{n}} \ss_{n-1} \ss_{n-2}^{-1} \ss_{n-3}^{-1} \dotsc \ss_{i+1}^{-1})$\\
		& $=$ & $ (\ss_{i+1} \dotsc \ss_{n-3} \ss_{n-2} \ss_{n-1} \ss_{n-2}^{-1} \ss_{n-3}^{-1} \dotsc \ss_{i+1}^{-1})$\\
		& & $\ss_n (\ss_{i+1} \dotsc \ss_{n-3}\ss_{n-2} \ss_{n-1} \ss_{n-2}^{-1} \ss_{n-3}^{-1} \dotsc \ss_{i+1}^{-1})$\\
		& $=$ & $ g\left(\bs{(i,n-1)}\right) g\left(\ss_n\right) g\left(\bs{(i,n-1)}\right)$.
	\end{tabular}
	
	(3) Where $i=m$ and $j=n-2$, we need to check that $g\left(\ss_n\right)$ commutes with $g\left(\bs{(\overline{1},\overline{n-2})}\right)$ and the relation
	$$g\left(\ss_n\right)g\left(\bs{(m,n-1)}\right)g\left(\ss_n\right) = g\left(\bs{(m,n-1)}\right)g\left(\ss_n\right)g\left(\bs{(m,n-1)}\right).$$
	
	We already proved this last relation
	in item (2).
	For the commuting relation,
	we have that $g\left(\bs{(\overline{1},\overline{n-2})}\right) = 
	\ss_1^{\ss_{m+2}\ss_{m+3} \dotsc \ss_{n-2}\ss_m^{-1}\ss_{m-2}^{-1} \dotsc \ss_2^{-1}}$ for $n > 5$. It follows
	that $g\left(\ss_n\right)$ commutes with $g\left(\bs{(\overline{1},\overline{n-2})}\right)$.\\
	
	(4) Where $i \neq m$ and $j = n-1$, we need to prove that
	$g\left(\ss_i\right)g\left(\bs{(i,n)}\right)g\left(\ss_i\right) = g\left(\bs{(i,n)}\right)g\left(\ss_i\right)g\left(\bs{(i,n)}\right)$ 
	and $g\left(\bs{(i,n)}\right)$ commutes with all the images by $g$ of 
	elements of $\RRefl$ that
	correspond to reflections in the string of item (4) in Lemma~\ref{LemRedDecEqt1}:
	\[
	s_{m+2} \dotsc s_{n-2}s_{n-1}(i+1,n-1)s_{i+2} 
	\dotsc s_{m-1}s_m (\overline{1},\overline{m})s_2 \dotsc s_{i-1}.
	\]
	
	We have that $g\left(\bs{(i,n)}\right) = \ss_n^{\ss_{n-1}^{-1}\ss_{n-2}^{-1} \dotsc \ss_{i+1}^{-1}}$. Then, we get
	\begin{center}\begin{tabular}{lll}
			$g\left(\ss_i\right)g\left(\bs{(i,n)}\right)g\left(\ss_i\right)$ & $=$ & $\ss_i(\ss_{i+1} \dotsc \ss_{n-2}\underline{\ss_{n-1}\ss_n\ss_{n-1}^{-1}}\ss_{n-2}^{-1} \dotsc \ss_{i+1}^{-1})\ss_i$\\
			& $=$ & $\ss_i(\ss_{i+1} \dotsc \ss_{n-2}\underline{\ss_{n}^{-1}}\ss_{n-1}\underline{\ss_{n}}\ss_{n-2}^{-1} \dotsc \ss_{i+1}^{-1})\ss_i$\\
			&$=$ & $\ss_n^{-1}\ss_i(i,n-1)\ss_i\ss_n $\\
			& $=$ & $ \ss_n^{-1}\bs{(i,n-1)}\ss_i\bs{(i,n-1)}\ss_n$ \hbox{ by induction hypothesis}\\
			& $=$ & $ \ss_n^{-1}\bs{(i,n-1)}\ss_n\ss_i\ss_n^{-1}\bs{(i,n-1)}\ss_n$,
	\end{tabular}\end{center}
	with $ \ss_n^{-1}g\left(\bs{(i,n-1)}\right)\ss_n = g\left(\bs{(i,n)}\right)$. Hence we get $$g\left(\ss_i\right)g\left(\bs{(i,n)}\right)g\left(\ss_i\right) = g\left(\bs{(i,n)}\right)g\left(\ss_i\right)g\left(\bs{(i,n)}\right).$$
	
	It is clear that $g\left(\bs{(i,n)}\right)$ commutes with $\ss_2$, $\ss_3$, $\dotsc$, $\ss_{i-1}$.
	
	Let us show that $g\left(\bs{(i,n)}\right)$ commutes with $\ss_m$. We have seen that $g\left(\bs{(i,n)}\right) = \bs{s_n^{-1}}g\left(\bs{(i,n-1)}\right)\bs{s_n}$. Then we get
	\begin{center}\begin{tabular}{lll}
			$g\left(\bs{(i,n)}\right)\ss_m $ & $=$ & $\ss_n^{-1}g\left(\bs{(i,n-1)}\right)\underline{\ss_n \ss_m}$\\
			& $=$ & $\ss_n^{-1}g\left(\bs{(i,n-1)}\right)\ss_m \ss_n$\\
			& $=$ & $ \underline{\ss_n^{-1}\ss_m}g\left(\bs{(i,n-1)}\right)\ss_n$ \hbox{ by induction hypothesis}\\
			& $=$ & $ \ss_m\ss_n^{-1}g\left(\bs{(i,n-1)}\right)\ss_n$\\
			& $=$ & $\ss_mg\left(\bs{(i,n)}\right)$.
	\end{tabular}\end{center}
	
	Similarly $g\left(\bs{(i,n)}\right)$ commutes with $\ss_{m-1}$, $\ss_{m-2}$, $\dotsc$, $\ss_{i+2}$, $\ss_{n-2}$, $\bs{s_{n-3}}$, $\dotsc$, $\ss_{m+2}$, and $g\left(\bs{(\overline{1},\overline{m})}\right)$. It is done by just replacing $\ss_m$ by each of the previous elements.
	
	Let us now show that $g\left(\bs{(i,n)}\right)$ commutes with $\ss_{n-1}$. We have that
	\begin{center}\begin{tabular}{lll}
			$g\left(\bs{(i,n)}\right)\ss_{n-1} $ & $=$ & $(\ss_{i+1} \dotsc \ss_{n-2}\ss_{n-1}) \ss_n (\ss_{n-1}^{-1} \ss_{n-2}^{-1} \dotsc \ss_{i+1}^{-1})\underline{\ss_{n-1}}$\\
			& $=$ & $(\ss_{i+1} \dotsc \ss_{n-2}\ss_{n-1}) \ss_n (\underline{\ss_{n-1}^{-1} \ss_{n-2}^{-1}\ss_{n-1}}\ss_{n-3}^{-1} \dotsc \ss_{i+1}^{-1})$\\
			& $=$ & $(\ss_{i+1} \dotsc \ss_{n-2}\ss_{n-1}) \ss_n (\underline{\ss_{n-2}} \ss_{n-1}^{-1}\ss_{n-2}^{-1} \dotsc \ss_{i+1}^{-1})$\\
			& $=$ & $\{(\ss_{i+1} \dotsc \underline{\ss_{n-2}\ss_{n-1}\ss_{n-2}}) \ss_n ( \ss_{n-1}^{-1}\ss_{n-2}^{-1} \dotsc \ss_{i+1}^{-1})$\\
			& $=$ & $(\ss_{i+1} \dotsc \ss_{n-3} \underline{\ss_{n-1}}\ss_{n-2}\ss_{n-1}) \ss_n ( \ss_{n-1}^{-1}\ss_{n-2}^{-1} \dotsc \ss_{i+1}^{-1})$\\
			& $=$ & $\ss_{n-1}g\left(\bs{(i,n)}\right)$.
	\end{tabular}\end{center}
	
	Finally, we show that $g\left(\bs{(i,n)}\right)$ commutes with $g\left(\bs{(i+1,n-1)}\right)$. 
	Since $$g\left(\bs{(i+1,n-1)}\right) = \ss_{n-1}^{\ss_{n-2}^{-1}\ss_{n-3}^{-1} \dotsc \ss_{i+2}^{-1}},$$ we get $g\left(\bs{(i+1,n-1)}\right) = \ss_{n-1}^{-1}g\left(\bs{(i+1,n-2)}\right)\ss_{n-1}$. We obtain
	
	\begin{center}\begin{tabular}{lll}
			$g\left(\bs{(i+1,n-1)}\right) g\left(\bs{(i,n)}\right)$ 
			& $=$ & $\ss_{n-1}^{-1}g\left(\bs{(i+1,n-2)}\right)\underline{\ss_{n-1}g\left(\bs{(i,n)}\right)}$\\
			& $=$ & $\ss_{n-1}^{-1}g\left(\bs{(i+1,n-2)}\right)\underline{g\left(\bs{(i,n)}\right)}\ss_{n-1}$\\
			& & by the previous case\\
			& $=$ & $\ss_{n-1}^{-1}\underline{g\left(\bs{(i+1,n-2)}\right)\ss_n^{-1}}g\left(\bs{(i,n-1)}\right)\ss_n\ss_{n-1}$\\
			& $=$ & $\ss_{n-1}^{-1}\ss_n^{-1}\underline{g\left(\bs{(i+1,n-2)}\right)g\left(\bs{(i,n-1)}\right)}\ss_n\ss_{n-1}$\\ 
			& $=$ & $\bs{s_{n-1}^{-1}s_n^{-1}}g\left(\bs{(i,n-1)}\right)\underline{g\left(\bs{(i+1,n-2)}\right)\ss_n}\ss_{n-1}$\\ 
			& & by induction hypothesis\\
			& $=$ & $\ss_{n-1}^{-1}\underline{\ss_n^{-1}g\left(\bs{(i,n-1)}\right)\ss_n}g\left(\bs{(i+1,n-2)}\right)\ss_{n-1}$\\ 
			& $=$ & $\underline{\ss_{n-1}^{-1}g\left(\bs{(i,n)}\right)}g\left(\bs{(i+1,n-2)}\right)\ss_{n-1}$\\ 
			& $=$ & $g\left(\bs{(i,n)}\right)\underline{\ss_{n-1}^{-1}g\left(\bs{(i+1,n-2)}\right)\ss_{n-1}}$\\
			& $=$ & $g\left(\bs{(i,n)}\right)g\left(\bs{(i+1,n-1)}\right)$.
	\end{tabular}\end{center}
	
	(5) Where $i = m$ and $j = n-1$, we need to check that
	$g\left(\ss_m\right)g\left(\bs{(m,n)}\right)g\left(\ss_m\right) = g\left(\bs{(m,n)}\right)g\left(\ss_m\right)g\left(\bs{(m,n)}\right)$ and that $g\left(\bs{(m,n)}\right)$ commutes with all the images by $g$ of the
	elements of $\RRefl$ that
	correspond to the reflections in the string from item (5) in Lemma~\ref{LemRedDecEqt1}: 
	\[s_{m+2} \dotsc s_{n-2}s_{n-1}(\overline{1},\overline{n-1}) s_2 \dotsc s_{m-1}.\]
	
	This is shown by following the same arguments as in (4).\\
	
	(6) Where $i \neq m$ and $j = n$, we need to check that $\ss_n$ commutes
	with $g\left(\bs{(\overline{1},\overline{m})}\right)$ and $g\left(\bs{(\overline{i},\overline{m+1})}\right)$
	This is straightforward to show.\\
	
	(7) Where $i=m$ and $j=n$, the image by $g$ of the copies of the reflections in the decomposition in item (7) of Lemma~\ref{LemRedDecEqt1} are only elements of $\SS$, so there is nothing to check.
\end{proof}

\begin{lemma}\label{LemLiftRedDecEqt2}
	Suppose that $w_0=w(\overline{i},\overline{j})$, with $1\leq i <m$, $m+1\leq j <n$,
	so that we are in the situation of Equation~\ref{Eqt2}. Let $\refl_1\refl_2 \dotsc \refl_{n-1}$ be the reduced decomposition of $w_0$ described in Lemma~\ref{LemRedDecEqt2}.
	Then we can deduce from the relations of the presentation of $\GG$ given in Theorem~\ref{ThmPresDn1} that $g\left(\rrefl_{n-1}\right) g\left(\rrefl_{n-2}\right) g\left(\rrefl_{n-1}\right) =g\left(\rrefl_{n-2}\right) g\left(\rrefl_{n-1}\right)  g\left(\rrefl_{n-2}\right)$, and that $g\left(\rrefl_{n-1}\right)$ commutes with
	each of the elements $g\left(\rrefl_{k}\right)$ with $k<n-2$.
\end{lemma}

\begin{proof}
	We consider the three different possible decompositions of $w_0$ that are
	described in Lemma~\ref{LemRedDecEqt2}. In each case we do not need to consider the relations between $g(\rrefl_{n})$ 
	and $g(\rrefl_k)$ if both elements are within the set $\SS$.
	
	(1) Where $i \neq m$ and $j \neq n-2, n-1,n$, we need to check that
	$g\left(\ss_n\right)$ commutes with $g\left(\bs{(\overline{1},\overline{m})}\right)$, $g\left(\bs{(\overline{i},\overline{j+1})}\right)$, and $g\left(\bs{(\overline{i+1},\overline{j})}\right)$. 
	
	The fact that $g\left(\ss_n\right) = \ss_n$ commutes with $g\left(\bs{(\overline{1},\overline{m})}\right)$ is done in Lemma \ref{LemLiftRedDecEqt1}(1). We have that
	$g\left(\bs{(\overline{i},\overline{j+1})}\right) = \ss_1^{\ss_{m+2}\ss_{m+3}\dotsc \ss_{j+1}\ss_m^{-1}\ss_{m-1}^{-1}\dotsc \ss_{i+1}^{-1}}$, and $g\left(\bs{(\overline{i+1},\overline{j})}\right) = \ss_1^{\ss_{m+2}\ss_{m+3}\dotsc \ss_j\ss_m^{-1}\ss_{m-1}^{-1}\dotsc \ss_{i+2}^{-1}}$ that both obviously commute with $\ss_n$.\\
	
	(2) Where $i \neq m$ and $j = n-2$, we need to check that
	$g\left(\ss_n\right) = \ss_n$ commutes with $g\left(\bs{(\overline{1},\overline{m})}\right)$, $g\left(\bs{(\overline{i+1},\overline{n-2})}\right)$, and check that $g\left(\bs{(\overline{i},\overline{n-1})}\right)g\left(\ss_n\right)g\left(\bs{(\overline{i},\overline{n-1})}\right) = g\left(\ss_n\right) g\left(\bs{(\overline{i},\overline{n-1})}\right) g\left(\ss_n\right)$. 
	
	The fact that $\ss_n$ commutes with $g\left(\bs{(\overline{1},\overline{m})}\right)$ and $g\left(\bs{(\overline{i+1},\overline{n-2})}\right)$ is done in Lemma \ref{LemLiftRedDecEqt2}(1).
	For the last relation, we have that
	\begin{center}\begin{tabular}{ll}
			$g\left(\bs{(\overline{i},\overline{n-1})}\right)\ss_n g\left(\bs{(\overline{i},\overline{n-1})}\right)$ & $=$ $\ss_{n-1}^{-1}g\left(\bs{(\overline{i},\overline{n-2})}\right)\underline{\ss_{n-1} \ss_n \ss_{n-1}^{-1}}g\left(\bs{(\overline{i},\overline{n-2})}\right)\ss_{n-1}$\\
			& $=$ $\ss_{n-1}^{-1}\underline{g\left(\bs{(\overline{i},\overline{n-2})}\right)\ss_n^{-1}}\ss_{n-1}\underline{\ss_ng\left(\bs{(\overline{i},\overline{n-2})}\right)}\ss_{n-1}$\\
			& $=$ $\ss_{n-1}^{-1}\ss_n^{-1}\underline{g\left(\bs{(\overline{i},\overline{n-2})}\right)\ss_{n-1}g\left(\bs{(\overline{i},\overline{n-2})}\right)}\ss_n\ss_{n-1}$\\
			& $=$ $\underline{\ss_{n-1}^{-1}\ss_n^{-1}\ss_{n-1}}g\left(\bs{(\overline{i},\overline{n-2})}\right)\underline{\ss_{n-1}\ss_n\ss_{n-1}}$\\
			& by induction hypothesis\\
			& $=$ $\ss_n\ss_{n-1}^{-1}\underline{\ss_n^{-1}g\left(\bs{(\overline{i},\overline{n-2})}\right)}\ss_n\ss_{n-1}\ss_n$\\ 
			& $=$ $\ss_n\ss_{n-1}^{-1}g\left(\bs{(\overline{i},\overline{n-2})}\right)\ss_n^{-1}\ss_n\ss_{n-1}\ss_n$\\ 
			& $=$ $\ss_ng\left(\bs{(\overline{i},\overline{n-1})}\right)\ss_n$.
	\end{tabular}\end{center}
	
	(3) Where $i \neq m$ and $j = n-1$, we need to check that $\ss_i g\left(\bs{(\overline{i},\overline{n})}\right) \ss_i = g\left(\bs{(\overline{i},\overline{n})}\right) \ss_i g\left(\bs{(\overline{i},\overline{n})}\right)$ and that $g\left(\bs{(\overline{i},\overline{n})}\right)$ commutes with the images by $g$ of elements that correspond to the reflections in the string from item (3) in Lemma~\ref{LemRedDecEqt2}:
	$$s_{m+2} \dotsc s_{n-2}s_{n-1}(\overline{i+1},\overline{n-1})s_{i+2} \dotsc s_{m-1}s_m(\overline{1},\overline{m})s_2 \dotsc s_{i-1}.$$
	
	First, we have
	\begin{center}\begin{tabular}{p{2cm}p{0.5cm}p{10cm}}
			$\ss_ig\left(\bs{(\overline{i},\overline{n})}\right) \ss_i $ & $=$ & $\underline{\ss_i\ss_n^{-1}}g\left(\bs{(\overline{i},\overline{n-1})}\right) \underline{\ss_n\ss_i}$\\
			& $=$ & $\ss_n^{-1}\underline{\ss_ig\left(\bs{(\overline{i},\overline{n-1})}\right)\ss_i}\ss_n$\\
			& $=$ & $\ss_n^{-1}g\left(\bs{(\overline{i},\overline{n-1})}\right)\ss_ig\left(\bs{(\overline{i},\overline{n-1})}\right)\ss_n$ \hbox{by induction hypothesis}\\
			& $=$ & $\underline{\ss_n^{-1}g\left(\bs{(\overline{i},\overline{n-1})}\right)\ss_n}\ss_i\underline{\ss_n^{-1}g\left(\bs{(\overline{i},\overline{n-1})}\right)\ss_n}$\\
			& $=$ & $g\left(\bs{(\overline{i},\overline{n})}\right)\ss_ig\left(\bs{(\overline{i},\overline{n})}\right)$.
	\end{tabular}\end{center}
	Since $\ss_2$, $\dotsc$, $\ss_{i-2}$, $\ss_{i-1}$, $\ss_{i+2}$, $\dotsc$, $\ss_{m-1}$, $\ss_m$, $\ss_{m+2}$, $\dotsc$, $\ss_{n-3}$, $\ss_{n-2}$ commute with $\ss_n$, then they commute with $g\left(\bs{(\overline{i},\overline{n})}\right) = \ss_n^{-1}g\left(\bs{(\overline{i},\overline{n-1})}\right)\ss_n$ by applying the induction hypothesis. Because $g\left(\bs{(\overline{1},\overline{m})}\right)$ commutes with $\ss_n$, we also get that $g\left(\bs{(\overline{i},\overline{n})}\right)$ commutes with $g\left(\bs{(\overline{1},\overline{m})}\right)$ by applying the same argument. Now, we prove that $g\left(\bs{(\overline{i},\overline{n})}\right)$ commutes with $\ss_{n-1}$. Actually, we have
	\begin{center}\begin{tabular}{lll}
			$g\left(\bs{(\overline{i},\overline{n})}\right) \ss_{n-1} $ & $=$ & $\ss_n^{-1}\ss_{n-1}^{-1}g\left(\bs{(\overline{i},\overline{n-2})}\right)\underline{\ss_{n-1}\ss_n\ss_{n-1}}$\\
			& $=$ & $\ss_n^{-1}\ss_{n-1}^{-1}\underline{g\left(\bs{(\overline{i},\overline{n-2})}\right)\ss_n}\ss_{n-1}\ss_n$\\
			& $=$ & $\underline{\ss_n^{-1}\ss_{n-1}^{-1}\ss_n}g\left(\bs{(\overline{i},\overline{n-2})}\right)\ss_{n-1}\ss_n$\\
			& $=$ & $\ss_{n-1}\ss_n^{-1}\underline{\ss_{n-1}^{-1}g\left(\bs{(\overline{i},\overline{n-2})}\right)\ss_{n-1}}\ss_n$\\
			& $=$ & $\ss_{n-1}\underline{\ss_n^{-1}g\left(\bs{(\overline{i},\overline{n-1})}\right)\ss_n}$\\
			& $=$ & $\ss_{n-1}g\left(\bs{(\overline{i},\overline{n})}\right)$.
	\end{tabular}\end{center}
	
	Finally, we show that $g\left(\bs{(\overline{i},\overline{n})}\right)$ commutes with $g\left(\bs{(\overline{i+1},\overline{n-1})}\right)$. We have that  $g\left(\bs{(\overline{i},\overline{n})}\right) g\left(\bs{(\overline{i+1},\overline{n-1})}\right)$ is equal to 
	\begin{center}\begin{tabular}{l}
			$\ss_n^{-1}g\left(\bs{(\overline{i},\overline{n-1})}\right)\ss_n\ss_{n-1}^{-1}g\left(\bs{(\overline{i+1},\overline{n-2})}\right)\ss_{n-1} = $\\
			$\ss_n^{-1}\ss_{n-1}^{-1}g\left(\bs{(\overline{i},\overline{n-2})}\right)\underline{\ss_{n-1}\ss_n\ss_{n-1}^{-1}}g\left(\bs{(\overline{i+1},\overline{n-2})}\right)\ss_{n-1} = $\\
			$\ss_n^{-1}\ss_{n-1}^{-1}\underline{g\left(\bs{(\overline{i},\overline{n-2})}\right)\ss_n^{-1}}\ss_{n-1}\underline{\ss_ng\left(\bs{(\overline{i+1},\overline{n-2})}\right)}\ss_{n-1} = $\\
			$\underline{\ss_n^{-1}\ss_{n-1}^{-1}\ss_n^{-1}}g\left(\bs{(\overline{i},\overline{n-2})}\right)\ss_{n-1}g\left(\bs{(\overline{i+1},\overline{n-2})}\right)\ss_n\ss_{n-1} = $\\
			$\ss_{n-1}^{-1}\ss_n^{-1}\underline{\ss_{n-1}^{-1}g\left(\bs{(\overline{i},\overline{n-2})}\right)\ss_{n-1}}g\left(\bs{(\overline{i+1},\overline{n-2})}\right)\ss_n\ss_{n-1} = $\\
			$\ss_{n-1}^{-1}\ss_n^{-1}\underline{g\left(\bs{(\overline{i},\overline{n-1})}\right)g\left(\bs{(\overline{i+1},\overline{n-2})}\right)}\ss_n\ss_{n-1} = $ by induction hypothesis\\
			$\ss_{n-1}^{-1}\underline{\ss_n^{-1}g\left(\bs{(\overline{i+1},\overline{n-2})}\right)g\left(\bs{(\overline{i},\overline{n-1})}\right)}\ss_n\ss_{n-1} = $\\
			$\ss_{n-1}^{-1}g\left(\bs{(\overline{i+1},\overline{n-2})}\right)\underline{\ss_n^{-1}g\left(\bs{(\overline{i},\overline{n-1})}\right)\ss_n}\ss_{n-1} = $\\
			$\ss_{n-1}^{-1}g\left(\bs{(\overline{i+1},\overline{n-2})}\right)\underline{g\left(\bs{(\overline{i},\overline{n})}\right)\ss_{n-1}} = $ by the previous case\\
			$\ss_{n-1}^{-1}g\left(\bs{(\overline{i+1},\overline{n-2})}\right)\ss_{n-1}g\left(\bs{(\overline{i},\overline{n})}\right) = $\\
			$g\left(\bs{(\overline{i+1},\overline{n-1})}\right)g\left(\bs{(\overline{i},\overline{n})}\right)$.
	\end{tabular}\end{center}
\end{proof}

\begin{lemma}\label{LemLiftRedDecEqt3}
	Suppose that $w_0=w(\overline{m},\overline{j})$, with $m+1\leq j <n$,
	so that we are in the situation of Equation~\ref{Eqt3}. Let $\refl_1\refl_2 \dotsc \refl_{n-1}$ be the reduced decomposition of $w_0$ described in Lemma~\ref{LemRedDecEqt3}.
	Then we can deduce from the relations of the presentation of $\GG$ given in Theorem~\ref{ThmPresDn1} that $g\left(\rrefl_{n-1}\right) g\left(\rrefl_{n-2}\right) g\left(\rrefl_{n-1}\right) = g\left(\rrefl_{n-2}\right) g\left(\rrefl_{n-1}\right)  g\left(\rrefl_{n-2}\right)$, and that  $g\left(\rrefl_{n-1}\right)$ commutes with
	each of the elements $g\left(\rrefl_{k}\right)$ with $k<n-2$.
\end{lemma}

\begin{proof}
	We consider the three cases of Lemma ~\ref{LemRedDecEqt3}.\\
	
	(1) Where  $j<n-2$, we need to check that $\ss_n$ commutes with
	$g\left(\bs{(1,j)}\right)$ and $g\left(\bs{(\overline{m},\overline{j+1})}\right)$. This is readily checked since $$g\left(\bs{(1,j)}\right) = \ss_j^{\ss_{j-1}^{-1}\ss_{j-2}^{-1} \dotsc \ss_2^{-1}} \hbox{ and } g\left(\bs{(\overline{m},\overline{j+1})}\right) = \ss_1^{\ss_{m+2}\ss_{m+3}\dotsc \ss_{j+1}}.$$
	
	(2) Where $j=n-2$, we need to check that $\ss_n$ commutes with
	$g\left(\bs{(1,n-2)}\right)$ and $g\left(\bs{(\overline{m},\overline{n-1})}\right)$.
	The first check is straightforward.
	For the second, 
	we have that $g\left(\bs{(\overline{m},\overline{n-1})}\right) = \ss_{n-1}^{-1}g\left(\bs{(\overline{m},\overline{n-2})}\right)\ss_{n-1}$ 
	and one shows that $$g\left(\bs{(\overline{m},\overline{n-1})}\right)\ss_n g\left(\bs{(\overline{m},\overline{n-1})}\right) = \ss_n g\left(\bs{(\overline{m},\overline{n-1})}\right) \ss_n$$
	similarly to the case $g\left(\bs{(\overline{i},\overline{n-1})}\right) \ss_n g\left(\bs{(\overline{i},\overline{n-1})}\right) = \ss_n g\left(\bs{(\overline{i},\overline{n-1})}\right) \ss_n$ 
	in the proof of Lemma \ref{LemLiftRedDecEqt2}(1).\\
	
	(3) Where $j = n-1$, we have that $g\left(\bs{(\overline{m},\overline{n})}\right) = \ss_1^{\ss_{m+2}\ss_{m+3}\dotsc \ss_n}
	= \ss_n^{-1}g\left(\bs{(\overline{m},\overline{n-1})}\right)\ss_n$. 
	All the commuting relations between $g\left(\bs{(\overline{m},\overline{n})}\right)$ and
	$\ss_{m-1}$, $\ss_{m-2}$, $\dotsc$, $\ss_{2}$ are obvious. We are 
	left to prove that $g\left(\bs{(\overline{m},\overline{n})}\right) \ss_m g\left(\bs{(\overline{m},\overline{n})}\right) = \ss_m g\left(\bs{(\overline{m},\overline{n})}\right) \ss_m$, 
	and $g\left(\bs{(\overline{m},\overline{n})}\right)$ commutes with $g\left(\bs{(1,n-1)}\right)$, $\ss_{n-1}$, $\ss_{n-2}$, $\dotsc$, $\ss_{m+2}$. This is done similarly to the proof of Lemma \ref{LemLiftRedDecEqt2}(3).
\end{proof}

The next lemma is readily checked.

\begin{lemma}\label{LemLiftRedDecEqt45}
	Suppose that $w_0=w(\overline{i},\overline{n})$, with $1\leq i \leq m$, 
	so that we are in the situation of Equation~\ref{Eqt4} or Equation~\ref{Eqt5}. An identical result to the previous lemmas holds in this situation.
\end{lemma}

This finishes the situation where $w_0$ is of type I. The next two lemmas are for types II and III, respectively.

\begin{lemma}\label{LemLiftRedDecEqt678}
	Suppose that $w_0=w(i,j)$ or $w_0=w(\overline{i},\overline{j})$ with 
	$1\leq i <j \leq m$, 
	so that we are in the situation of one of Equations~\ref{Eqt6}-\ref{Eqt8}. Let $\refl_1\refl_2 \dotsc \refl_{n-1}$ be the reduced decomposition of $w_0$ as described in Section~\ref{SubsectionRedDecompII}.
	Then we can deduce from the relations of the presentation of $\GG$ that $g\left(\rrefl_{n-1}\right) g\left(\rrefl_{n-2}\right) g\left(\rrefl_{n-1}\right) =g\left(\rrefl_{n-2}\right) g\left(\rrefl_{n-1}\right)  g\left(\rrefl_{n-2}\right)$, and that $g\left(\rrefl_{n-1}\right)$ commutes with
	each of the elements $g\left(\rrefl_{k}\right)$ with $k<n-2$.
\end{lemma}

\begin{proof}
	
	In the case of this lemma, we have that $g\left(\rrefl_{n-1}\right) = \ss_n$ and $g\left(\rrefl_{n-2}\right) = \ss_{n-1}$, so the braid relation $\ss_n \ss_{n-1} \ss_n = \ss_{n-1} \ss_n \ss_{n-1}$ is clearly a consequence of the relations of $\GG$. The commuting relations are also obvious since in the situation of Equations~\ref{Eqt6}-\ref{Eqt8}, the indices $i,j$ are such that $1 \leq i < j \leq m$, so that they are far away from $n$ (we have $i,j < n-2$). Hence $\ss_n$ obviously commutes with the image by $g$ of the elements $\rrefl_i$ corresponding to the reflections in the reduced decomposition of $w_0$ described in Section \ref{SubsectionRedDecompII}.  
\end{proof}

\begin{lemma}\label{LemLiftRedDecEqt91011}
	Suppose that $w_0=w(i,j)$ or $w_0=w(\overline{i},\overline{j})$ with 
	$m+1\leq i <j \leq n$, 
	so that we are in the situation of one of Equations~\ref{Eqt9}--\ref{Eqt11}. Let $\refl_1\refl_2 \dotsc \refl_{n-1}$ be the reduced decomposition of $w_0$ as described in Section~\ref{SubsectionRedDecompII}.
	Then we can deduce from the relations of the presentation of $\GG$ that $g\left(\rrefl_{n-1}\right) g\left(\rrefl_{n-2}\right) g\left(\rrefl_{n-1}\right) =g\left(\rrefl_{n-2}\right) g\left(\rrefl_{n-1}\right)  g\left(\rrefl_{n-2}\right)$, and that $g\left(\rrefl_{n-1}\right)$ commutes with
	each of the elements $g\left(\rrefl_{k}\right)$ with $k<n-2$, except for Equation~\ref{Eqt11} (with $i=n-1$) where we need to show one additional non-commuting relation.
\end{lemma}

\begin{proof}
	The argument of the proof is identical to the situation of Equations~\ref{Eqt1}--\ref{Eqt5} treated in Lemmas~\ref{LemLiftRedDecEqt1}--\ref{LemLiftRedDecEqt45}, which is appropriate to leave as an exercise.
\end{proof}

\blfootnote{Barbara Baumeister, Fakultät für Mathematik, Universität Bielefeld, Postfach 10 01 31, 33501 Bielefeld, Germany,
	E-mail: b.baumeister@math.uni-bielefeld.de,\\
	
	Georges Neaime, Fakultät für Mathematik, Universität Bielefeld, Postfach 10 01 31, 33615 Bielefeld, Germany,
	E-mail: gneaime@math.uni-bielefeld.de,\\
	
	Sarah Rees, School of Mathematics and Statistics, University of Newcastle, Newcastle NE1 7RU, UK,
    E-mail: sarah.rees@newcastle.ac.uk.}

\end{document}